\documentclass[10pt]{article}
\usepackage{hyperref}
\usepackage{xcolor}
\usepackage{amssymb}
\usepackage{amsmath}
\usepackage{cancel}
\input{liemacs10.sty} 
\usepackage{color}

\renewcommand{\down}{{\mathop{\downarrow}}}
\newcommand\be{{\bf{e}}}

\renewcommand\up{{\uparrow}}

\newcommand{\sH}{{\sf H}}
\newcommand{\sK}{{\sf K}}

\newcommand{\sV}{{\tt V}}
\newcommand{\sE}{{\tt E}}
\newcommand{\sa}{\mathop{{\rm sa}}\nolimits}

\newcommand\RR{{\mathbb R}}

\newcommand\bC{{\mathbb C}}

\newcommand{\fri}{\mathfrak{i}} 

\newcommand{\Mob}{{\rm\textsf{M\"ob}}}

\renewcommand{\phi}{\varphi}

\newcommand{\Stand}{\mathop{{\rm Stand}}\nolimits}

\newcommand{\nc}{\text{nc}}
\newcommand{\cc}{\text{cc}}

\newcommand{\m}{\mathrm{max}}

\renewcommand\mlabel{\label}

\begin{document}

\title{From local nets to Euler elements}
\author{Vincenzo Morinelli, Karl-Hermann Neeb} 

\maketitle
\abstract{
  Various aspects of the geometric 
  setting of Algebraic Quantum Field Theory (AQFT) models
  related to 
  representations of the Poincar\'e group  can be studied
for general Lie groups, whose Lie algebra contains an Euler element,
i.e., $\ad h$ is diagonalizable with eigenvalues in $\{-1,0,1\}$. 
    This has been explored by the authors and their collaborators during recent years.
A key property in this construction is the Bisognano--Wichmann property
(thermal property for wedge region algebras)
{concerning the geometric implementation of
  modular groups of local algebras}.

In the present paper we prove that under a natural regularity condition,
geometrically implemented modular groups
arising from the Bisognano--Wichmann property, are always generated by
    Euler elements. We also show the converse, namely that in presence of
    Euler elements  and the Bisognano--Wichmann property,
    regularity and localizability hold in a quite general setting.
Lastly we show that, in this generalized AQFT,
in the vacuum representation,  under analogous
assumptions (regularity and Bisognano--Wichmann), the von Neumann algebras associated to wedge regions are type III$_1$ factors, 
a property that is well-known in the AQFT context.

\tableofcontents

\section{Introduction}

This paper is part of a project by the authors and collaborators aiming to deepen the relations between geometric properties of Algebraic Quantum Field Theory (AQFT), Lie theory and unitary
representation theory; see \cite{MN21, MNO23a, MNO23b, NO21, NOO21, FNO23}.

Starting from  fundamental properties of a relativistic quantum theory, the
Bisognano--Wichmann (BW) property and the PT symmetry, 
a generalized setting 
{to study AQFT models  has been developed, that starts from the}
geometry 
and representations of the symmetry group as fundamental input.
Through this
description, it was possible to present a new large set of
mathematical models in an abstract way (nets on abstract
wedge spaces) or a geometric way (nets on open subsets of homogeneous spaces).
A key role is played by the Bisognano--Wichmann property which
in AQFT models ensures
that the vacuum state is thermal 
for any geodesic observer in a wedge region (see e.g.~\cite{Lo97}
and references therein).
In our context the Bisognano-Wichmann property serves
to provide a geometric implementation of
modular groups of some local algebras. 
Along this analysis, a fundamental role has been played by Euler
elements that {also have} been extensively studied in Lie
theory (see e.g.~\cite{MN21} and \cite{MNO23a}) and here  creates  a bridge between Lie theory,
the AQFT localization properties, and the modular theory of operator algebras. 

Nets of standard subspaces (in the one-particle representation)
are fundamental objects to analyze properties of
AQFT Models. In particular, they play a {central} role in the recent study of
entropy and energy inequalities (see  \cite{MTW22,Lo20, CLRR22, CLR20} and references therein),
new constructions in AQFT (\cite{MN22,LL15,LMPR19,MMTS21}),
and in a very large family of models (see references in \cite{DM20}).   Due to the Bisognano--Wichmann property and the PCT
symmetry, the language of standard subspaces deeply relates the geometry of the symmetry group with its representation theory and the algebraic objects related to the local von Neumann algebras. 

\smallskip

To introduce the main ideas of this paper, we first recall
the key steps to understand the setting we developed 
for this generalized AQFT. 

\smallskip

\nin\textit{Geometric setting}: In the physics context, the underlying manifolds are
relativistic spacetimes, {i.e., time-oriented Lorentzian manifolds}.
In Minkowski or de Sitter spacetime
localization regions are called wedges and they are
specified by one-parameter groups of Lorentz boosts fixing them. 
On $2$-dimensional Minkowski spacetime, the conformal chiral
components yield fundamental localization regions, corresponding
to circle intervals, which are also
specified by one-parameter groups of dilations of the M\"obius group.
So one can describe fundamental localization regions in terms of generators of certain one-parameter groups
in the Lie algebra of the symmetry group. 
This framework can be generalized to the context where $G$ is a (connected) Lie group whose Lie algebra $\fg$
contains an Euler element $h$ ($\ad h$ is diagonalizable with eigenvalues in $\{-1,0,1\}$)
to construct an abstract version of the correspondence between wedge regions and boost generators. 
In particular, one can associate to every connected simple Lie group $G$ and 
any  Euler element $h \in \g$ a non-compactly causal symmetric space $M=G/H$
(see Section~\ref{subsec:ncc} and \cite{MNO23a} for details).
{For the Lorentz group $G = \SO_{1,d}(\R)_e$,
  we thus obtain de Sitter space $M = \dS^d$.
In this case we associate to every boost generator
(=Euler element) the corresponding wedge region,
and, in the general context, }
a wedge region in $M$ associated to $h$ is a connected component of the
open subset on which the flow of $h$ is ``future directed''
(timelike in the Lorentzian case). 
More generally, for an Euler element in a reductive Lie algebra $\fg$,
  there exists a non-compactly causal symmetric space $G/H$
  in which one can identify wedge regions $W$,
  but localization extends to general non-empty open subsets,
  see Section~\ref{sect:pregeom} for details. 

\smallskip

\nin\textit{AQFT setting}: Models in AQFT are determined by nets of von Neumann algebras indexed by open regions of the spacetime satisfying fundamental quantum and relativistic assumptions, in particular isotony, locality, Poincar\'e covariance, positivity of the energy, and existence of the vacuum vector
with Reeh-Schlieder property.
Nets of standard subspaces arise at least in two natural ways:
as the one-particle nets in irreducible Poincar\'e representations,
from which  the free fields are obtained by second quantization, and
by  acting with the self-adjoint part of the local von Neumann algebras on a cyclic  separating vacuum vector.
The Bisognano--Wichmann property and the
anti-unitary PCT symmetry determine the wedge subspaces and the
key role in this identification is played by Tomita--Takesaki theory.
This technique has been established by Brunetti, Guido and Longo in \cite{BGL02} for cases of physical relevance.}

This construction has been realized in a much wider generality by the authors  in the current project (cf.~the references above)
with the  following idea: given an involutive automorphism $\sigma$ of a Lie group $G$,
an (anti-)unitary representation $U$ of the extended group $G_\sigma = G \rtimes \{\1,\sigma\}$  on an Hilbert space $\cH$,  
an Euler element $h$ in the  Lie algebra $\fg$ of $G$, 
and a $G$-transitive family $\cW_+$
of abstract wedges (fiber-ed over the adjoint orbit of $h$), 
then one can associate an ``abstract net'' $(\sH(W))_{W \in \cW_+}$
of standard subspaces of $\cH$ giving 
a net only depending on the symmetry group.
This construction builds on the 
Brunetti--Guido--Longo (BGL) construction (\cite{BGL02}
and \cite{LRT78}). 

Often this net can be realized geometrically on a causal
homogeneous space $M$,  
in such a way that the
abstract wedge acquires a geometric interpretation as wedge regions in~$M$.
Here we call a $G$-space {\it causal} if
it contains a family $C_m \subeq T_m(M)$ of a pointed, generating, closed
convex cones which is invariant under the $G$-action.
Typical examples are time-oriented Lorentzian manifolds on which $G$
acts by time-orientation preserving symmetries or conformal maps.
Given a  representation of~$G_\sigma$, 
one can then try to extend the canonical net of standard subspaces 
from the set of wedge regions to arbitrary open subsets $\cO \subeq M$. 
A net of real subspaces
associates to open regions of a causal homogeneous space
real subspaces of localized states satisfying
properties that are  analogous to those of nets of von Neumann algebras: 
For a unitary representation $(U,\cH)$ of a connected a Lie group~$G$ 
and a homogeneous space $M = G/H$, we consider 
families $(\sH(\cO))_{\cO \subeq M}$ of closed real subspaces of~$\cH$, 
indexed by open subsets $\cO \subeq M$ with the following properties:
\begin{itemize}
\item[(Iso)] {\bf Isotony:} $\cO_1 \subeq \cO_2$ 
implies $\sH(\cO_1) \subeq \sH(\cO_2)$ 
\item[(Cov)] {\bf Covariance:} $U(g) \sH(\cO) = \sH(g\cO)$ for $g \in G$. 
\item[(RS)] {\bf Reeh--Schlieder property:} 
$\sH(\cO)$ is cyclic  if $\cO \not=\eset$. 
\item[(BW)] {\bf Bisognano--Wichmann property:}
There exists an open subset $W \subeq M$ (called a {\it wedge region}),  
such that $\sH(W)$ is standard
with modular group $\Delta_{\sH(W)}^{-it/2\pi} = U(\exp th)$, $t\in \R$,
  for some $h \in \g$,  for which $\exp(\R h).W \subeq W$.
\end{itemize}

So one has to specify the real
subspaces associated to wedge regions and
identify their properties. There are
different possibilities to extend to larger classes of open subsets,
that in general do not coincide. One
is based on specifying certain generator spaces in which a linear basis
may correspond to components of a field on $M$ and then obtain local subspaces in terms of test functions, {see \cite{NO21, FNO23} for irreducible
  representations and Theorem~\ref{thm:local-reduct} for general
  representations of reductive groups).} 
Alternatively,  one can specify maximal covariant nets which are isotonic
and have the (BW) property, here discussed in Section~\ref{sect:maxnet}.

\smallskip

\textit{In this paper we discuss the necessity and the consequences of considering Euler elements as fundamental
  objects for our constructions.}
We will further see how this choice will be consistent with AQFT models.
This will be done by facing the following three questions:

\smallskip

\nin\textit{Question 1. Is it necessary to consider Euler elements 
  determining fundamental localization regions for one particle nets?}

Yes, it is a consequence of the Bisognano--Wichmann property and
a natural regularity property: 
Given a standard subspace $\sH$ whose modular group
corresponds to a one-parameter subgroup
$(\exp th)_{t \in \R}$ of $G$ (BW property), in Theorem \ref{thm:2.1} we show
that $h$ is an Euler element if there exists an
$e$-neighborhood $N \subeq G$ for which $\bigcap_{g \in N} U(g)\sV$ is cyclic. 
This result is abstract and does not refer
to any geometry of wedges or subregions
but can be applied to any net of real subspaces
satisfying a minimal set of of axioms,
such as (Iso), (Cov), (RS) and (BW). 
Our Euler Element Theorem
(Theorem~\ref{thm:2.1})
has particularly striking consequences for such nets.
In this setting, it implies in particular that all modular groups 
  that are geometrically implementable by 
  one-parameter subgroups of finite-dimensional Lie groups
in the sense of the (BW) property, 
are generated by Euler elements.
Similar regularity conditions are satisfied in many
  AQFT models and an analogous property has been used
also in  \cite[Def.~3.1]{BB99} and \cite[Sect.~IV.B]{Str08}.

\smallskip
The second question concerns the converse implication:

\nin\textit{Question 2: Are the nets of standard subspaces associated to Euler elements regular?} {More precisely, let
  $h \in \g$ be an Euler element, $\tau_h = e^{\pi i \ad h}$
  the corresponding involution on $\g$, and suppose that
  this involution on $\g$ integrates to an involution
  $\tau_h^G$ on $G$, so that we can form the group 
  $G_{\tau_h} := G \rtimes \{\id_G, \tau_h^G\}$.
Given an (anti-)unitary representation of this group $G_{\tau_h}$,
we consider the canonical standard subspace $\sV = \sV(h,U) \subeq \cH$,
specified by
\[ \Delta_\sV = e^{2 \pi i \partial U(h)} \quad \mbox{ and } \quad
  J_\sV = U(\tau_h^G) \]
(cf.\ \cite{BGL02}). 
A natural way to address such regularity questions is to
associate to $\sV$ a net $\sH^{\rm max}$ defined on open subsets
of a homogeneous space $M = G/H$ by
\[ \sH^{\rm max}(\cO) := \bigcap_{\cO \subeq g.W} U(g)\sV \] 
(cf.\ \eqref{eq:def-ho}).
If every $g \in G$ with $g.W \subeq W$ satisfies $U(g)\sV \subeq \sV$,
this leads to a covariant, isotone map with $\sH^{\rm max}(W) = \sV$.
Regularity now corresponds to the existence of open subsets
$\cO \subeq W$ with $N.\cO \subeq W$ for which $\sH^{\rm max}(\cO)$
is cyclic (Reeh--Schlieder property).
We show that regularity follows} if the representation satisfies certain 
positivity conditions, namely that the ``positive energy'' cones
$C_\pm$ in the abelian Lie subalgebras {$\g_{\pm 1}(h)
= \ker(\ad h \mp \1)$} are generating;
see Theorem~\ref{thm:reg-posen}.
This requirement can be weakened as follows.
If $G = N \rtimes L$ is a semidirect product and we know already
that the restriction $U\res_L$ satisfies the regularity condition,
then it suffices that the intersections $C_\pm \cap \fn_{\pm 1}(h)$
generate $\fn_{\pm 1}(h)$ (Theorem \ref{thm:posreg}). 
This is in particular the case 
for positive energy representations of the
connected Poincar\'e group  $G = \cP = \R^{1,d} \rtimes \cL^\up_+$.
That representations of linear reductive groups always satisfy 
the regularity condition can be derived from some
localizability property asserting for every (anti-)unitary
representation the existence of a net on the associated
non-compactly causal symmetric space, satisfying
(Iso), (Cov), (RS) and (BW) (Theorem \ref{thm:local-reduct}).
In particular, the maximal net $\sH^{\rm max}$ has this property.
As every algebraic linear Lie group 
is a semidirect product $G= N\rtimes L$,
where $N$ is unipotent and $L$ is reductive \cite[Thm. VIII.4.3]{Ho81}, 
many questions related to regularity can be reduced to representations
of nilpotent groups.  
These {regularity} results include all the physically
relevant one-particle models;  for instance
the $\U(1)$-current and its derivatives
(covariant under the M\"obius group) satisfy the hypotheses of
Theorem~\ref{thm:reg-posen} and
so do the one-particle representations of the Poincar\'e group,
to which Theorem~\ref{thm:posreg} applies,
but not  Theorem \ref{thm:reg-posen}.

\smallskip

\nin\textit{Question 3: What can we say on nets of von Neumann algebras?} Once fundamental localization regions are specified, 
it is natural to discuss nets of von Neumann algebras on causal homogeneous
spaces as above. Such nets exist because second quantization of
one-particle nets on causal homogeneous spaces
produces nets of operator algebras. 
For a systematic construction of twisted second quantization
functors, we refer to \cite{CSL23}. 
Second quantization nets 
correspond to bosonic second quantization in AQFT, in general a spin-statistics result is still to be obtained. 
The results on von Neumann algebras presented here
apply to general geometric relative position of
von Neumann algebras, and second quantization
provides examples of nets on  causal $G$-spaces.
In Section~\ref{sect:vna}, Theorem \ref{thm:3.8} implies that,
given a connected Lie group $G$, when the BW property and a
suitable regularity property hold, and there is a
unique $G$-fixed state (the vacuum state), then the wedge algebras are
factors of type III$_1$ with respect to
Connes' classification of factors \cite{Co73}. This extends the
known results in AQFT dealing with more specific groups and
  spaces (see for instance \cite{Dr77, Lo82, Fr85, BDF87, BB99}
  and references therein).
   Here the key property for an Euler element $h\in\fg$
  implementing the modular group through the BW property 
  is to be \textit{anti-elliptic}, i.e.,
  any quotient $\fq = \g/\fn$ ($\fn \trile \g$ an ideal), 
  for which the image of $h$ in $\fq$ is elliptic\begin{footnote}
    {We call $x \in \g$ {\it elliptic} if $\ad x$ is semisimple
      with purely imaginary spectrum, i.e., diagonalizable over
      $\C$ with purely   imaginary eigenvalues.}
  \end{footnote}
  is at most one-dimensional and linearly generated by the image of~$h$.
  If $\g$ is simple, then $\g$ has no non-trivial quotients, so that 
  any Euler element $h \in \g$ is anti-elliptic, but Theorem~\ref{thm:3.8}
  covers much more general situations.
  We actually do not need to start this discussion with a vacuum vector, but with a vector that is invariant under {$U(\exp(\R h))$}.
  The case of non-unique invariant vector  is discussed in 
  Section~\ref{sect:disnet} in terms of a direct integral decomposition
  taking all structures into account.   

\smallskip 
Along the paper, only very few comments on locality, or its twisted version, will come up.
This is because the regularity property as well as the localization property merely refer to a subspace, resp., a subalgebra.
To implement (twisted-) locality conditions, suitable wedge complements  have to be introduced (cf.~\cite{MN21}). 
In our  general setting,  some work still has to be done to adapt the second quantization procedure. 

\smallskip

Recently, operator algebraic techniques have been very fruitful for
the study of energy inequalities. In many of these results
 the modular Hamiltonian is instrumental.
This object corresponds to the logarithm of the modular operator of a local algebra of a specific ``wedge region'',  which in some cases
can be identified with the
generator of a one-parameter group of spacetime symmetries 
by the Bisognano--Wichmann property (see for instance \cite{MTW22,Lo20, CLRR22,Lo19, CLR20, Ara76,LX18,LM23}).  
In our setting, we start with a general Lie algebra element $h\in\fg$
specifying the flow implemented by the modular operator through the
(BW) property. Then
\[ \log\Delta_{\sH(W)}=2\pi i\cdot \partial U(h) \]
is the corresponding modular Hamiltonian. 
In this case, we know from Theorems~\ref{thm:2.1} and \ref{thm:3.8} that $h$ has to be an Euler element.
\textit{In particular we {obtain} an abstract
  {algebraic} characterization of {those elements in the Lie algebra
    of the symmetry group
    that may correspond to} modular Hamiltonians}. The study of the modular flow on the manifold is particularly relevant. In order to find regions where to prove energy inequalities, one may also need to deform the modular flow (\cite{MTW22, CF20}). Due to the recent characterization of modular flows on homogeneous space, a specific geometric analysis is expected to be possible.

\smallskip

This paper is structured as follows:
In Section~\ref{sec:2} we recall  the fundamental geometry
of Euler elements, both abstractly and on causal homogeneous spaces. 
In Section~\ref{sect:pregeom}  we recall the geometry of standard subspaces, properties of nets of standard subspaces and the axioms
(Iso), (Cov), (BW) and (RS). In particular, Section~\ref{sect:maxnet} introduces
minimal and maximal nets on open subsets of a causal homogeneous space
$M = G/H$ that are associated to an Euler element $h \in \g$ and a
corresponding wedge region $W \subeq M$. 

In Sections~\ref{sect:modEul}, \ref{sec:loc-reg} and~\ref{sect:vna}
we discuss Questions 1,2 and 3, respectively. 
Our key result, the Euler Element Theorem (Theorem~\ref{thm:2.1})
is proved in Subsection~\ref{subsec:3.1}.
In Subsection~\ref{subsec:strich-opalg} we describe
{its implications for} operator algebras
with cyclic separating vectors 
(Theorems~\ref{thm:2.1-alg} and ~\ref{thm:3.2}).
The main results of Subsection~\ref{subsec:4.1}
are Theorems~\ref{thm:reg-posen} and \ref{thm:posreg},
  deriving regularity from positive spectrum conditions.
In Subsection~\ref{subsec:4.2}, we turn to localizability
  aspects of nets of real subspaces. Here our main results are
  Theorem~\ref{thm:local-reduct}, asserting localizability 
  for linear reductive groups in all representations in all
  non-empty open subsets of the associated non-compactly causal
  symmetric space for a suitably chosen wedge region. 
  This allows us to derive that,
  for the Poincar\'e group, localizability in spacelike cones
  is equivalent to the positive energy condition
  (Theorem~\ref{thm:poinloc}).
 In Section~\ref{sect:vna} we continue the discussion of  
 applications of our results to standard subspace and
 von Neumann algebras $\cM$ by systematically using
 Moore’s Eigenvector Theorem~\ref{thm:moore}. 
  The first main result in this section are Theorem~\ref{thm:vg-ident},
  characterizing for (anti-)unitary representation
  $(U,\cH)$ of $G_{\tau_h}$ the subspace 
  $\sV_G = \bigcap_{g \in G} U(g) \sV$ as the set of fixed points
  of a certain normal subgroup specified in Moore's Theorem. 
  The second one is Theorem~\ref{thm:3.8} that combines Moore's Theorem
  with Theorem~\ref{thm:2.1-alg} to obtain a criterion for $\cM$ to be
  a factor of type III$_1$. If $\cM$ is not a factor,
  but $\cM'$ and $\cM$ are conjugate under~$G$, we show
  that all the structure we discuss survives the
  central disintegration of $\cM$.
  
We conclude this paper with an outlook section and
four appendices, concerning background on operator algebras,
unitary Lie group representations, direct integrals,
and some more specific observations needed to discuss examples.

\vspace{8mm}
\nin {\bf Notation} 

\begin{itemize}
\item Strips in $\C$: $\cS_\beta = \{  z \in \C \: 0 < \Im z < \beta \}$ and
  $\cS_{\pm\beta} = \{  z \in \C \: |\Im z| < \beta \}$. 
\item The neutral element of a group $G$ is denoted $e$, and
  $G_e$ is the identity component.
\item The Lie algebra of a Lie group $G$ is denoted $\L(G)$ or $\g$.   
\item For an involutive automorphism $\sigma$ of $G$, we write
  $G^\sigma =  \{ g\in G \:  \sigma(g) = g \}$ for the subgroup of fixed points
  and $G_\sigma := G \rtimes \{\id_G, \sigma\}$ for the corresponding
 group extension. 
\item $\AU(\cH)$ is the group  of unitary or antiunitary operators
    on a complex Hilbert space. 
\item An (anti-)unitary representation of $G_\sigma$ is a homomorphism
    $U \: G_\sigma \to \AU(\cH)$ with $U(G) \subeq \U(\cH)$ for which 
    $J := U(\sigma)$ is antiunitary, i.e., a conjugation.
\item Unitary or (anti-)unitary representations on the complex Hilbert space
  $\cH$ are denoted as pairs $(U,\cH)$.
\item $\oline U$ is the canonical unitary representation on the
    complex conjugate space $\oline\cH$, where the operators
    $\oline U(g) = U(g)$ are the same, but the complex structure is given by
    $I\xi := -i\xi$.
\item If $(U,\cH)$ is a unitary representation of $G$ and
  $J$ a conjugation with $J U(g) J = U(\sigma(g))$ for $g \in G$,
  the canonical extension $U^\sharp$ of $U$ to $G_\sigma$ is specified
  by $U^\sharp(\sigma) := J$ (cf.\ Definition~\ref{def:types}).
\item If $G$ is a group acting on a set $M$ and $W \subeq M$ a subset, then
  the stabilizer subgroup of~$W$ in $G$ is denoted 
  $G_W := \{ g \in G \:  g.W = W\}$, and $S_W := \{ g \in G \:  g.W\subeq W\}$.
  \item A closed real subspace $\sV$ of a complex Hilbert space $\cH$ is called
    {\it standard} if $\sV \cap i \sV = \{0\}$ and $\sV + i \sV$ is dense in $\cH$.
  \item If $\g$ is a Lie algebra and $h \in \g$, then
    $\g_\lambda(h) = \ker(\ad h - \lambda \1)$ is the $\lambda$-eigenspace of $\ad h$ and $\g^\lambda(h) = \bigcup_k \ker(\ad h - \lambda \1)^k$ is the
    generalized $\lambda$-eigenspace.
  \item An element $h$ of a Lie algebra $\g$ is called
    \begin{itemize}
    \item     {\it hyperbolic} if $\ad h$ is diagonalizable over $\R$
    \item {\it elliptic} or {\it compact} if $\ad h$ is semisimple
      with purely imaginary spectrum, i.e., $\oline{e^{\R \ad h}}$ is a
      compact subgroup of $\Aut(\g)$. 
    \end{itemize}
  \item {A {\it causal $G$-space} is a smooth $G$-space $M$,
      endowed with a $G$-invariant {\it causal structure}, i.e.,
    a field $(C_m)_{m \in M}$ of closed convex cones
    $C_m \subeq T_m(M)$.}
  \item For a unitary representation $(U,\cH)$ of $G$ we write:
  \begin{itemize}
  \item $\partial U(x) = \derat0 U(\exp tx)$ for the infinitesimal
    generator of the unitary one-parameter group $(U(\exp tx))_{t \in\R}$
    in the sense of Stone's Theorem. 
  \item $\dd U \colon \g \to \End(\cH^\infty)$ for the representation of
    the Lie algebra $\g$ on the space $\cH^\infty$ of smooth vectors. Then
    $\partial U(x) = \oline{\dd U(x)}$ (operator closure) for $x \in \g$. 
  \end{itemize}
\end{itemize}

\nin{\bf Acknowledgment:} {The authors thank
  Roberto Longo and Detlev Buchholz} for helpful discussions.
VM was partially supported by the University of Rome through the MUN Excellence Department Project 2023-2027, the “Tor Vergata” CUP E83C23000330006, Fondi di Ricerca Scientifica d’Ateneo 2021, OAQM, CUP E83C22001800005, and the European Research Council Advanced Grant 669240 QUEST. VM also thanks INdAM-GNAMPA. The research of K.-H. Neeb was partially supported by DFG-grant NE 413/10-2.

\section{Preliminaries}
\mlabel{sec:2}

In this section we recall 
  fundamental geometric structures
  related to Euler elements of Lie algebras and corresponding
  symmetric spaces. Its main purpose is to introduce notation
  and some general techniques that will be used throughout the paper.
  Subsection~\ref{sect:pregeom} deals with abstract
  wedge spaces of graded Lie groups $G_\sigma$ and how they can be
  related to sets of wedge regions in homogeneous causal
  $G$-spaces $M = G/H$.
  Subsection~\ref{sect:geoss}  then turns to nets of real subspaces
  $\sH(\cO)$, associated to open subsets $\cO$ of some
  homogeneous space of~$G$. Here we introduce the basic axioms
  (Iso), (Cov), (RS) and (BW). We also show that,
  if (BW) holds for some $h \in \g$ and some wedge region
  $W \subeq M$, for which $g.W \subeq W$ implies $g.\sH(W) \subeq \sH(W)$,
  we obtain minimal and maximal isotone, covariant nets
  $\sH^{\rm min}$ and $\sH^{\rm max}$ satisfying (BW), such that
  any other net $\sH$ with these properties satisfies
  \[ \sH^{\rm min}(\cO) \subeq \sH(\cO) \subeq \sH^{\rm max}(\cO) \]
  on all open subsets $\cO \subeq M$. 
  We also study basic properties of intersections of standard subspaces
  in $G$-orbits.

\subsection{The geometry of Euler elements}
\label{sect:pregeom}

In this subsection we recall some fundamental geometric structures related to Euler elements in the Lie algebra $\fg$ of a Lie group $G$.
For more details and background, we refer to \cite{MN21,MNO23a,MNO23b,NO22}.

\subsubsection{Euler elements}\label{sect:ee}

Let $G$ be a connected Lie group, the Lie algebra of a Lie group $G$ is denoted $\L(G)$ or $\g$.  For an involutive automorphism $\sigma$ of $G$, we write
  $G^\sigma =  \{ g\in G \:  \sigma(g) = g \}$ for the subgroup of fixed points
  and $G_\sigma := G \rtimes \{\id_G, \sigma\}$ for the corresponding
  group extension. Then
  \[ \eps  \: G_\sigma \to (\{\pm 1\}, \cdot), \quad 
    (g,\id_G) \mapsto 1, \quad    (g,\sigma) \mapsto -1 \]
  is a group homomorphism that defines on $G_\sigma$ the structure of a $\Z_2$-{\it graded Lie group}.

{
  \begin{rem} \mlabel{rem:gsigma-same}
 (a)    The group $G_\sigma$ depends on $\sigma$, 
    but two involutive automorphisms $\sigma_1$ and $\sigma_2$
    lead to isomorphic extensions $G_{\sigma_1} \cong G_{\sigma_2}$ if
    and only if $\sigma_2\sigma_1^{-1}$ is an inner automorphism
    $c_y(x) = yxy^{-1}$  for some $y \in G$ with $\sigma_1(y) = y^{-1}$
    (hence also $\sigma_2(y) = y^{-1}$). Then
    \[ \Phi \:  G_{\sigma_2} \to G_{\sigma_1}, \quad
      (g,\id_G) \mapsto (g,\id_G), \quad
      (e,\sigma_2) \mapsto (y, \sigma_1) \]
    defines an isomorphism because
    \[ (y,\sigma_1)(g,\id_G)(y,\sigma_1)^{-1}  = (y \sigma_1(g) y^{-1},\id_G)
      = (\sigma_2(g), \id_G)\]
    and
    \[ (y,\sigma_1)^2 = (y \sigma_1(y), \id_G) = (e,\id_G).\] 

    \nin (b) If $\sigma$ is inner, then the above argument shows that
    $G_\sigma \cong G \times \{\pm 1\}$ is a product group. Therefore
    (anti-)unitary representations $(U,\cH)$ of $G_\sigma$
    restrict to  
    unitary representations $U$ of $G$ for which there exists a conjugation $J$
    commuting with $U(G)$. Then the real Hilbert space $\cH^J$ is $U(G)$-invariant,
    and $(U,\cH)$ is simply the complexification of the so-obtained
    real orthogonal representation of $G$ on which $J$ acts by complex conjugation. 
  \end{rem}
}
  
\begin{defn} \mlabel{def:euler}
 (a) We call an element $h$ of the finite dimensional 
real Lie algebra $\g$ an
{\it Euler element} if $\ad h$ is non-zero and diagonalizable with 
$\Spec(\ad h) \subeq \{-1,0,1\}$. In particular the eigenspace 
decomposition with respect to $\ad h$ defines a $3$-grading 
of~$\g$: 
\[ \g = \g_1(h) \oplus \g_0(h) \oplus \g_{-1}(h), \quad \mbox{ where } \quad 
\g_\nu(h) = \ker(\ad h - \nu \id_\g)\] 
Then $\tau_h(y_j) = (-1)^j y_j$ for $y_j \in \g_j(h)$ 
defines an involutive automorphism of $\g$. 

We write $\cE(\g)$ for the set of Euler elements in~$\g$. 
The orbit of an Euler element  $h$ under the group 
$\Inn(\g) = \la e^{\ad \g} \ra$ of 
{\it inner automorphisms} 
is denoted with $\cO_h = \Inn(\g)h \subeq \g$.
We say that $h$ is {\it symmetric} if $-h \in \cO_h$.

\nin (b) The set 
  \[ \cG {:= \cG({G_\sigma})} :={ \{ (h,\tau)\in\g \times G_\sigma \:
      \: \tau^2 = e, \eps(\tau) = -1, \Ad(\tau)h = h\}}\] 
  is called the {\it abstract wedge space of ${G_\sigma}$.
An element $(h,\tau) \in \cG$ is called an 
{\it Euler couple} or {\it Euler wedge}   
if  $h\in\cE(\fg)$ and 
\begin{equation}\label{eq:eul}
  \Ad(\tau)=\tau_h. 
\end{equation} 
Then $\tau$ is called an {\it Euler involution}. 
We write $\cG_E\subeq \cG$ for the subset of Euler couples.}

\nin (c) {On $\g$ we consider the 
{\it twisted adjoint action} of $G_\sigma$} which changes the sign on odd group elements: 
\begin{equation}
  \label{eq:adeps}
  \Ad^\eps \: G_\sigma \to \Aut(\g), \qquad 
\Ad^\eps(g) := \eps(g) \Ad(g).
\end{equation}
It extends to an action of $G_\sigma$ on $\cG$  by 
\begin{equation}
  \label{eq:cG-act}
 g.(h,\tau) := (\Ad^\eps(g)h, g\tau g^{-1}).
\end{equation}

\nin (d) (Order structure on $\cG$) For a given 
$\Ad^\eps(G)$-invariant pointed closed 
convex cone $C_\g \subeq \g$, we obtain an order 
structure on $\cG$ as follows (\cite[Def.~2.5]{MN21}).
We associate to $W = (h,\tau) \in \cG$
a semigroup $S_W$ whose unit group is
$S_W \cap S_W^{-1} = G_{W}$, the stabilizer of $W$.  
It is specified by 
\[ S_W := \exp(C_+) G_{W} \exp(C_-) 
  = G_{W} \exp\big(C_+ + C_-\big).\]
Here the convex cones $C_\pm$ are the intersections 
\begin{equation}
  \label{eq:cw}
 C_\pm := \pm C_\g \cap \g^{-\tau} \cap \g_{\pm 1}(h), 
\quad \mbox{ where } \quad \g^{\pm \tau} 
:= \{ y \in \g \: \Ad(\tau)(y) = \pm y \}.
\end{equation} 
That $S_W$ is a semigroup follows from
  \cite[Thm.~2.16]{Ne22}, applied to the Lie subalgebra
  \[ L_W := (C_+ - C_+) + \g_0(h)^\tau + (C_- - C_-),\]
in which $h$ is an Euler element. That $L_W$ is a Lie algebra follows from
  ${[C_+, C_+] = [C_-, C_-] =  \{0\}}$. To see this, observe that 
  $\g_+ := \sum_{\lambda >0} \g_\lambda(h)$ is a nilpotent Lie algebra,
  so that the subspace $\fn := (C_U \cap \g_+) - (C_U \cap \g_+)$ is a nilpotent 
  Lie algebra generated by the pointed invariant cone $C_U \cap \g_+$,
  hence abelian by \cite[Ex.~VII.3.21]{Ne99}.

Then 
$S_W$ defines a $G$-invariant partial order on the orbit 
$G.W \subeq \cG$ by 
\begin{equation}
  \label{eq:cG-ord}
g_1.W \leq g_2.W  \quad :\Longleftrightarrow \quad 
g_2^{-1}g_1 \in S_W.
\end{equation}
In particular, $g.W \leq W$ is equivalent to $g \in S_W$. 

\nin (e) (Duality operation) 
The notion of a  
``causal complement'' is defined on the abstract wedge space as follows: 
For $W = (h,\tau) \in \cG$, we define the {\it dual wedge} by   
$W' := (-h,\tau) {= \tau.W}$. 
Note that $(W')' = W$ and $(gW)' = gW'$ for $g \in G$ 
by \eqref{eq:cG-act}.  
This relation fits the geometric interpretation in the context
of wedge domains in spacetime manifolds.
\end{defn}

{
  \begin{rem} If $h \in \g$ is an Euler element in a simple real Lie algebra,
    then the cases where the involution $\tau_h$ is inner are classified
    in \cite{MNO23c}.
  \end{rem}
  }

\begin{rem} \mlabel{rem:stab} Let $W = (h,\tau) \in \cG$ 
and consider $y \in \g$. Then $\exp(\R y)$ fixes $W$ if and only if 
\[ [y,h] = 0 \quad \mbox{ and  } \quad y = \Ad(\tau)y.\] 
If $(h,\tau)$ is an Euler couple, then 
{$\Ad(\tau)y = \tau_h y = y$ follows from $y\in \g_0(h)$}, so that 
\begin{equation}
  \label{eq:centralizercond}
\g_W := \{ y \in \g \:  \exp(\R y) \subeq G_W\}  
=\g_0(h) =  \ker(\ad h).
\end{equation}
\end{rem}

\begin{defn} {\rm(The abstract wedge space)}
  \mlabel{def:2.6}
For a fixed couple $W_0 = (h,\tau) \in \cG$, the orbits 
\[ \cW_+ (W_0):= G.W_0 \subeq \cG  \quad \mbox{ and } \quad 
\cW(W_0) := G_\sigma.W_0 \subeq \cG  \] 
are called the {\it positive} and the {\it full
 {abstract} wedge space containing $W_0$}. 
\end{defn}

Here is a classification theorem of real Lie algebra supporting Euler elements. The families are determined by their root system:

\begin{theorem} \mlabel{thm:classif-symeuler} {\rm(\cite[Thm.~3.10]{MN21})}
Suppose that $\g$ is a non-compact simple 
real Lie algebra and that  $\fa \subeq \g$ is maximal $\ad$-diagonalizable 
with restricted root system 
$\Sigma = \Sigma(\g,\fa) \subeq \fa^*$ of type $X_n$. 
We follow the conventions of the tables in {\rm\cite{Bo90}}
for the classification of irreducible root systems and the enumeration 
of the simple roots $\alpha_1, \ldots, \alpha_n$. 
For each $j \in \{1,\ldots, n\}$, we consider the uniquely determined element 
$h_j \in \fa$ satisfying $\alpha_k(h_j) =\delta_{jk}$.
Then every Euler element in $\g$ is conjugate under inner automorphism 
to exactly one~$h_j$. For every irreducible root system, 
the Euler elements among the $h_j$ are  the following: 
\begin{align} 
&A_n: h_1, \ldots, h_n, \quad 
\ \ B_n: h_1, \quad 
\ \ C_n: h_n, \quad \ \ \ D_n: h_1, h_{n-1}, h_n, \quad 
E_6: h_1, h_6, \quad 
E_7: h_7.\label{eq:eulelts2}
\end{align}
For the root systems $BC_n$, $E_8$, $F_4$ and $G_2$ no Euler element exists 
(they have no $3$-grading). 
The symmetric Euler elements (see {\rm Definition~\ref{def:euler}(a)})
are 
\begin{equation}
  \label{eq:symmeuler}
A_{2n-1}: h_n, \qquad 
B_n: h_1, \qquad C_n: h_n, \qquad 
D_n: h_1, \qquad 
D_{2n}: h_{2n-1},h_{2n}, \qquad 
E_7: h_7.  
\end{equation}
\end{theorem}

\begin{example} \label{ex:desit}
  (Wedge regions in Minkowski and de Sitter spacetimes)  
The Minkowski spacetime is the manifold $\RR^{1,d}$ endowed with the Minkowski metric $$ds^2=dx_0^2-dx_1^2-\ldots-dx_d^2.$$ The de Sitter spacetime is the Minkowski submanifold 
$\dS^{d}=\{(x_0,\bx)\in\RR^{1,d}: \bx^2-x_0^2=1\}$, 
endowed with the metric obtained by restriction of 
the Minkowski metric  to $\dS^d$.
In the literature the $x_0$-coordinate is often denoted 
$t$ as it is interpreted as a time coordinate.
The symmetry groups of isometries for these spaces are
the (proper) Poincar\'e group ${\cP_+ = \RR^{1,d}\rtimes
\SO_{1,d}(\R)}$ on Minkowski
space $\R^{1,d}$ and the (proper) Lorentz group $\cL_+ = \SO_{1,d}(\R)$ on $\dS^d$.

The generator $h \in \so_{1,d}(\R)$ of the Lorentz boost on the 
$(x_0,x_1)$-plane 
\[  h(x_0,x_1,x_2, \ldots, x_{d}) = (x_1, x_0, 0, \ldots, 0)\] 
is an Euler element. It combines with the spacetime 
reflection 
\[ j_h(x) =(-x_0,-x_1,x_2,\ldots,x_d) \] 
to the Euler couple $(h,j_h)
\in \cG(\cL_+) \subeq \cG(\cP_+)$,
for the graded Lie groups $\cL_+ = \SO_{1,d}(\R)$ and~$\cP_+$.
The spacetime region
\[ W_R=\{x\in\RR^{1,d}: |x_0|<x_1\} \] 
is called the {\it standard right wedge} in Minkowski space, and
\[ W_R^{\dS} := W_R \cap \dS^d \]
{is the corresponding wedge region in de Sitter space.} 
Note that $W_R$ and therefore $W_R^{\dS}$ are 
invariant under $\exp(\R k_{1})$. 
Poincar\'e transformed regions
$W = g.W_R, g \in \cP_+$, are called
  {\it wedge regions in Minkowski space};
  likewise the regions $W^{\dS}=g.W_R^{\dS}$, $g \in \cL_+$,
  are called {\it wedge regions in de Sitter space}.
  To $W = g.W_R$ we associate the boost group
    $\Lambda_W(t) := \exp(t \Ad(g)h)$. 
They are in equivariant one-to-one correspondence with
abstract Euler couples in $\cG_E(\cP_+)$ and
  $\cG_E(\cL_+)$, respectively.
  Here the couple $(h, j_h)$ corresponds to $W_R$ and $W_R^{\rm dS}$,  
  respectively (cf.~\cite[Lemma~4.13]{NO17}, \cite[Rem.~2.9(e)]{MN21}
  and \cite[Sect.~5.2]{BGL02}).
 \end{example}

 \subsubsection{Wedge domains in causal homogeneous spaces}
 \label{sect:chs}

In this subsection we recall how to specify suitable wedge regions~$W\subeq M$ in a causal homogeneous space $M = G/H$. 
Motivated by the Bisognano--Wichmann property (BW) in AQFT, the modular flow, namely the flow of the one-parameter group
generated by an Euler element on a causal homogeneous space $M$
should be timelike future-oriented. 
Indeed, the modular flow is correspond to the
inner  time evolution of Rindler wedges (see \cite{CR94} and also
\cite{BB99, BMS01, Bo09}, \cite[\S 3]{CLRR22}).
In our context this means that
the {\it modular vector field}
\begin{equation}
  \label{eq:xhdef}
 X_h^M(m) 
 :=  \frac{d}{dt}\Big|_{t = 0} \exp(th).m 
\end{equation}
should satisfy
\[ X^M_h(m) \in C_m^\circ \quad \mbox{ for all } \quad m \in W,\]
{where the causal structure on $M$ is specified by the $G$-invariant field $(C_m)_{m \in M}$
  of closed convex cones $C_m \subeq T_m(M)$.}
{If this condition is satisfied in one $m \in M$,
  we may always replace $h$ by a conjugate and thus assume that it holds
  in the base point $m = eH$. Then} 
the connected component
\begin{equation}
  \label{eq:W-def}
  W := W_M^+(h)_{eH}
\end{equation}
of the base point
$eH \in M$ in the {\it positivity region}
\begin{equation}\label{def:WM}
 W_M^+(h) := \{ m \in M \: X^M_h(m) \in C_m^\circ \} 
 \end{equation}
 is the natural candidate for a domain for which (BW) could be satisfied.
 {Note that this domain depends on $h$ and the causal structure
   on $M$ and that $W$ is invariant under the connected stabilizer
   $G^h_e$ of $h$, hence in particular under $\exp(\R h)$.} 
These ``wedge regions'' have been studied for compactly and
non-compactly causal symmetric spaces
in \cite{NO23} and \cite{NO22, MNO23b}, respectively. 

\begin{remark}\label{rem:1-1} If $Z(G) = \{e\}$,
  then each Euler element $h \in \g$
  determines a pair $(h,\tau_h) \in \cG_E$ uniquely.
{So the stabilizers $G^{(h,\tau_h)}$ and $G^h$ coincide
    and we may identify $\cW_+(h,\tau_h) \subeq \cG_E$
    with the adjoint orbit $\cO_h = \Ad(G)h$.
    We thus obtain a natural map
  from $\cW_+(h,\tau_h) \cong \cO_h$ to regions in $M$
 by $g.(h,\tau_h) \mapsto g.W_M^+(h)$. If, in addition, 
  $G^h$ preserves the connected component $W \subeq W_M^+(h)$
  (which is in particular the case if $W_M^+(h)$ is connected, hence
  equal to $W$),} this leads to a 
  map from the abstract wedge space $\cW_+(h,\tau_h)$ to the
  geometric wedge space on~$M$. Proposition~\ref{prop:LSW} below
  implies that it is isotone if the order on $\cW_+(h,\tau_h)$ is specified
  by the invariant cone~$C_M$ from \eqref{eq:cm}.
\end{remark}

\subsubsection*{The compression semigroup of a wedge region}

Let $M = G/H$ be a causal homogeneous space
and $(C_m)_{m \in M}$ its causal structure. {Writing
$G \times TM \to TM, (g,v) \mapsto g.v$ for the action of $G$ on
the tangent bundle, this means that $g.C_m = C_{g.m}$ for
$g \in G, m \in M$. }
Identifying $T_{eH}(M)$ with $\g/\fh$, we consider the projection
$p \: \g \to \g/\fh$ and the cone $C := C_{eH} \subeq \g/\fh$.
For $y \in \g$, the corresponding vector field on $M$ is given by
\[  X_y^M(gH) 
  =  \frac{d}{dt}\Big|_{t = 0} \exp(ty).gH
  =  g.\frac{d}{dt}\Big|_{t = 0} \exp(t \Ad(g)^{-1}y).eH
  = g.p(\Ad(g)^{-1}y).\]
The set
\begin{equation}
  \label{eq:cm}
 C_M
  := \{ y \in \g \:  (\forall m \in M) X_y^M(m) \in C_m \} 
  = \bigcap_{g \in G} \Ad(g) p^{-1}(C)
\end{equation}
is a closed convex $\Ad(G)$-invariant cone in $\g$. If $G$ acts
effectively on $M$, then it is also pointed
because elements in $C_M \cap - C_M$ correspond to vanishing vector fields 
on $M$. This cone is a geometric analog of the positive cone $C_U$
corresponding to a unitary representation of~$G$
(see \eqref{eq:CU}). The following observation shows that it behaves in many
respects similarly (cf.\  \cite{Ne22}).

As any connected component $W \subeq W_M^+(h) \subeq M$ is invariant
under $\exp(\R h)$, the same holds for the closed convex cone 
\[ C_W := \{ y \in \g \:  (\forall m \in W)\ X_y^M(m) \in C_m \} \supeq  C_M.\]
Below we show that this cone determines the tangent wedge of the compression
semigroup of~$W$.

\begin{prop} \mlabel{prop:LSW}
  For a connected component $W \subeq W_M^+(h)$, its compression semigroup
  \[ S_W := \{ g \in M\: g.W \subeq W \} \]
  is a closed subsemigroup of $\g$ with
  $G_W := S_W \cap S_W^{-1} \supeq G^h_e$ and
  \begin{equation}
    \label{eq:lsw}
 \L(S_W) := \{  x \in \g \:  \exp(\R_+ x) \subeq S_W \}
 =  \g_0(h) + C_{W,+} + C_{W,-},
\end{equation}
where the two convex cones $C_{W,\pm}$ are the intersections
$\pm C_W \cap \g_{\pm 1}(h).$
In particular, the convex cone $\L(S_W)$ has interior points if $C_M$ does.   
\end{prop}

\begin{prf} As $W \subeq M$ is an open subset, its complement $W^c := M \setminus W$
  is closed, and thus 
  \[ S_W = \{ g \in G \: g^{-1}.W^c \subeq W^c \} \]
  is a closed subsemigroup of $G$, so that its tangent wedge $\L(S_W)$ is a closed
  convex cone in $\g$ (\cite{HN93}).

  Let $m = gH \in W$, so that $p(\Ad(g)^{-1}h) \in C^\circ$.
  For $x \in \g_1(h)$ we then derive from $\g_2(h) = \{0\}$ that 
  \[ e^{t \ad x} h = h + t [x,h] \quad \mbox{ for } \quad t \in \R.\]
  This leads to 
  \begin{align*}
 p(\Ad(\exp(tx)g)^{-1}h)
&    = p(\Ad(g)^{-1} e^{-t \ad x}h)
    = p(\Ad(g)^{-1} ( h - t[x,h])\\
    &    = p(\Ad(g)^{-1} ( h + tx))
      = p(\Ad(g)^{-1}h) + t p(\Ad(g)^{-1} x).
  \end{align*}
  For $x \in C_{W,+}$, we have $p(\Ad(g)^{-1}x) \in C$, so that
  $p(\Ad(\exp(tx)g)^{-1}h) \in C^\circ$ for $t \geq 0$,
 which in turn implies that
    $\exp(tx).m \in W$ for $m \in W$ and $t\geq 0$.
    So $\exp(\R_+ x) \subeq S_W$, and thus $x \in \L(S_W)$.
 It likewise follows that
 $C_{W,-}  \subeq \L(S_W)$.
 The invariance of $W$ under the identify component
  $G^h_e$ of the centralizer of $h$ further entails
  $\g_0(h) \subeq \L(S_W)$, so that
  \begin{equation}
    \label{eq:incl1}
    C_{W,+} + \g_0(h) + C_{W,-} \subeq \L(S_W).
  \end{equation}
  We now prove the converse inclusion. 
  If $X^M_x(m) \not \in C_m$, i.e., $p(\Ad(g)^{-1}x) \not\in C$,
  then there exists a $t_0 > 0$ with 
\[ p(\Ad(g)^{-1}h) + t_0 \cdot p(\Ad(g)^{-1} x)\not\in C \]  (\cite[Prop.~V.1.6]{Ne99}),
  so that $\exp(t_0 x).m \not \in W$. We conclude that
\[     \L(S_W) \cap \g_1(h) = C_{W,+}.\] 
Further, the  invariance of the closed convex cone
  $\L(S_W)$ under $e^{\R \ad h}$ implies that, for
  $x = x_{-1} + x_0 + x_1 \in \L(S_W)$ and $x_j \in \g_j(h)$, we have
  \[ x_{\pm 1} = \lim_{t \to \infty} e^{\mp t} e^{\pm t \ad h} x \in \L(S_W) \cap
    \g_{\pm 1}(h) = C_{W,\pm}, \]
  which implies the other inclusion 
$\L(S_W) \subeq C_{W,+} + \g_0(h)+ C_{W,-},$ 
  hence equality by \eqref{eq:incl1}.
  
  Let $p_\pm \:  \g \to \g_{\pm 1}(h)$ denote the projection along the other
  eigenspaces of $\ad h$. Then
  \[ C_{W,\pm} \supeq C_{M,\pm} := \pm C_M \cap \g_{\pm 1}(h) = \pm p_{\pm}(C_M) \]
  also follows from  \cite[Lemma~3.2]{NOO21}. Therefore $C_M^\circ \not=\eset$
  implies $C_{W,\pm}^\circ \not=\eset$, and this is equivalent to
  $\L(S_W)^\circ \not=\eset$.   
\end{prf}

\begin{rem} In many situations, such as the action
    of $\PSL_2(\R)$ on the circle $\bS^1 \cong \bP_1(\R)$,
  the cones $C_{W,\pm} \supeq C_{M,\pm}$ 
    coincide, {and we believe that this is probably always the case.}
It  is easy to see that, if $x \in C_{W,+}$, then the positivity region 
    \[ \Omega_x := \{ m \in M \: X^M_x(m) \in C_m \} \]
    contains $W$ (by definition), and it is also invariant under
    $\exp(\R h)$ and $\exp(\R x)$, to that
    \begin{equation}
      \label{eq:omegax}
      \Omega_x \supeq \bigcup_{t > 0} \exp(-tx).W.
    \end{equation}
    Clearly, $\Omega_x = M$ follows if the right hand side of
      \eqref{eq:omegax}  
    is dense in $M$, but we now show that Minkowski space
    provides an example where $\Omega_x = M$
    without the right hand side of \eqref{eq:omegax} being dense in $M$.

    If $G$ is the connected Poincar\'e group acting on Minkowski
    space $M = \R^{1,d}$ and
    \[ W = W_R = \{ (x_0, \bx) \: x_1 > |x_0| \},  \]
    then
    \[ S_{W} = \oline{W} \rtimes
      \big(\SO_{d-1}(\R) \times \SO_{1,1}(\R)^\uparrow\big)\]
    (\cite[Lemma~4.12]{NO17})     implies that
    \[ C_{W,\pm} = \L(S_{W})\cap \g_1(h)
      = \R_+ (\pm\be_0 + \be_1) \]
    consists of constant vector fields, so that
    $C_{W,\pm} = C_{M,\pm}$ in this case. Here we see that, for
    $x = \be_0 + \be_1 \in C_{W,+}$, the domain 
    $\Omega_x = W - \R_+ x$ is an open half space,
    hence in particular not dense in~$M$.
   Therefore we cannot expect the domain  $\Omega_x$ in
    \eqref{eq:omegax} to be dense in $M$.
\end{rem}

\subsubsection{Non-compactly causal spaces}
\mlabel{subsec:ncc}

Let $G$ be a connected simple Lie group 
  and $h \in \g$ be an Euler element. The associated non-compactly
  causal symmetric spaces are obtained as follows
  (see \cite[Thm.~4.21]{MNO23a} for details).
  We choose a Cartan involution $\theta$ on $\g$
  with $\theta(h) = -h$, write $K := G^\theta$ for the corresponding group of
  fixed points, and consider the involution
  $\tau_{{\nc}} := \tau_h \theta \in \Aut(\g)$. 
  Assuming that the involution $\tau_{{\nc}}$ integrates to an involution
  $\tau^G_{{\nc}}$ on $G$, we consider a subgroup $H \subeq \Fix(\tau^G_{{\nc}}) = G^{\tau^G_{{\nc}}}$
  that is open (hence has the same Lie algebra $\fh = \g^{\tau_{{\nc}}}$) 
  and {for which} $H \cap K$ fixes $h$). Then
  $M := G/H$ is the corresponding {\it non-compactly causal symmetric space},
  where the invariant causal structure is determined by the 
  maximal pointed closed convex cone 
  $C \subeq \g^{-\tau_{{\nc}}} \cong T_{eH}(M)$ containing~$h$.
  This construction ensures in particular that $eH \in W^+_M(h)$. 
Assume, in addition, that $G = \Inn(\g)$ is centerfree. 
Then \cite[Cor.~7.2]{MNO23b} identifies $W$ from \eqref{def:WM}
with the
``observer domain'' $W(\gamma)$ associated to the geodesic
$\gamma(t) = \Exp_{eH}(t h)$ in~$M$.
Further, \cite[Prop.~7.3]{MNO23b} thus implies that the 
stabilizer $G_W$ of $W$ coincides with the centralizer $G^h$ of~$h$: 
\[ G_W = G^h,\]
so that, for centerfree groups, we may identify the wedge space
\[ \cW(M,h) := G.W \cong G/G_W = G/G_h \cong  \cO_h \] 
with the adjoint orbit $\cO_h$ of~$h$.

If, more generally, $G$ is only assumed connected
and $M = G/H$ is a corresponding non-compactly causal symmetric space,
then the connected component $W := W_M^+(h)_{eH} \subeq M$ containing
$eH$ is the natural wedge region
and $G_{W_M} \subeq G^h$ may be a proper subgroup.
{Typical examples {arise naturally} for $\g = \fsl_2(\R)$
(see \cite[Rem.~5.13]{FNO23}).

For non-compactly causal symmetric spaces, we typically
have $G_{\tau_\nc}\not\cong G_{\tau_h}$ because the
{product $\tau_{{\nc}}\tau_h$ need not be inner
  (cf.\ Remark~\ref{rem:gsigma-same}). 
 If, for instance, $\g = \fh_\C$ and $\tau_{{\nc}}$ is complex conjugation
  with respect to $\fh$ (non-compactly causal of complex type),
  then $\tau_h$ is complex linear and $\tau_{{\nc}}$ is antilinear, hence
  their product is antilinear and therefore not inner.

  From $\tau_{\nc} = \theta \tau_h$ we derive $\tau_{\nc} \tau_h = \theta$,
  which leads to the question when $\theta$ is inner.
  For a characterization of these case, we refer to \cite{MNO23c}.}

\subsubsection{Compactly causal spaces}
\mlabel{subsec:cc}

Let $G$ be a connected Lie group
and $M = G/H$ be a compactly causal symmetric space,
where $H \subeq G^{\tau_\cc}$ is an open subgroup and
$\tau_\cc$ is an involutive automorphism of $G$.
We assume that there exists an Euler element
$h \in \fh = \g^\tau_\cc$, so that we obtain a so-called modular
compactly causal symmetric Lie algebra $(\g, \tau_\cc, C, h)$
(cf.\ \cite{NO22}).
Here $C \subeq \fq := \g^{-\tau_\cc}$ is a pointed generating closed
convex cone, invariant under $\Ad(H)$, whose interior $C^\circ$
consists of elliptic elements. 
We further assume that {the involution $\tau_h$ on $\g$
  integrates to an involutive automorphism $\tau_h^G$ of $G$ such that}   
$\tau_h^G(H) = H$ {and the existence of a pointed generating
$\Ad(G)$-invariant cone $C_\g \subeq \g$ such that}
\[  -\tau(C_\g) = C_\g \quad \mbox{ and }\quad  C = C_\g \cap \fq.\]

Then $eH \in M$ is a fixed point of the modular flow
and there exists a unique connected component
\[ W = W_M^+(h)_{eH} \]
of the positivity domain $W_M^+(h)$ that contains $eH$ in its boundary.
Theorem~9.1 in \cite{NO22} then asserts that
\[ S_W := \{ g \in G \:  g.W \subeq W \}
  = G_W \exp(C_\g^c),\]
where $G_W = \{ g \in G \:  g.W = W\}$ and
\begin{equation}
  \label{eq:cgc}
 C_\g^c := C_{\g,+} + C_{\g,-} {\subeq \g^{-\tau_h}} \quad \mbox{ for } \quad 
C_{\g,\pm} := \pm C_\g \cap \g_{\pm 1}(h). 
\end{equation}
The cone $C_\g^c$ is $-\tau_\cc$-invariant with 
\begin{equation}
  \label{eq:doubledcone2}
  (C_\g^c)^{-\tau_\cc} = C_\g^c \cap \fq
  = C_+ + C_-
  \quad \mbox{ for } \quad
  C_\pm :=
  {\pm C_\g \cap \fq_{\pm 1}(h) = \pm C \cap \fq_{\pm 1}(h).}
\end{equation}
Here
\[ G_W = G^h_e H^h \subeq G^h \]
is an open subgroup {with the Lie algebra $\g_0(h)$}
and the wedge space
\[ \cW(M,h) := G.W \cong G/G_W \]
carries the structure of a symmetric space (\cite[Prop.~9.2]{NO22}).
Covering issues related  to $\cW(M,h)$
are discussed in \cite[Prop.~9.4]{NO22}.

\begin{rem} In general $\tau_\cc\not=\tau_h$
    and also $\tau_\cc \not= \tau_h \theta$ for Cartan involutions
    $\theta$ with $\theta(h) = -h$.
The latter products $\tau_h \theta$ are precisely
    the involutions $\tau_{\rm nc}$, corresponding to non-compactly
  causal symmetric spaces.  In general we also have 
  $G_{\tau_\cc} \not\cong G_{\tau_h}$ because the product 
  $\tau_\cc\tau_h$ need not be inner (cf.~Remark~\ref{rem:gsigma-same}),
  as the following example shows. 
If $(\g,\tau_\cc)$ is compactly causal of group type, then 
$\g \cong \fh \oplus \fh$ with $\tau_{\cc}(x,y) = (y,x)$, whereas 
$\tau_h$ preserves both ideals. Therefore $\tau_\cc\tau_h$
flips the ideals, hence cannot be inner.
If $(\g,\tau_\cc)$ is of Cayley type, then (by definition) 
$\tau_\cc = \tau_h$ for an Euler element~$h$.

If $\g$ is simple, then it is of hermitian type, so that
all Euler elements in $\g$ are conjugate. The relation
\[ \tau_{\cc} \Ad(g)\tau_h \Ad(g)^{-1} = \tau_{\cc}\tau_h  \Ad(\tau_h^G(g)g^{-1}) \]
then shows that $\tau_\cc \tau_h$ is inner for one Euler element if and only
this is the case for all Euler elements. As we have seen above,
this is true for Cayley type spaces.
\end{rem}

\subsection{ The geometry of nets of real subspaces}
\label{sect:geoss}

In this section we recall some fundamental properties of the
{geometry of} standard subspaces on generalized one-particle nets.
{We refer to \cite{Lo08, MN21, NO17} for more details}.
Sections \ref{sect:interss} and \ref{sect:maxnet} contains
{some new observations that will become relevant below.}

\subsubsection{Standard subspaces}
We call a closed real subspace $\sH$ of the complex Hilbert space 
$\cH$ \textit{cyclic} if $\sH+i\sH$ is dense in $\cH$, \textit{separating} if $\sH\cap i\sH=\{0\}$, and \textit{standard} 
if it is cyclic and separating. We write $\Stand(\cH)$ for
the set of standard subspaces of $\cH$. The symplectic
orthogonal of a real subspace $\sH$ is defined by the symplectic form 
$\Im \langle\cdot,\cdot\rangle$ on $\cH$ via 
\[ \sH'=\{\xi\in\cH:(\forall \eta \in \sH)\ 
 \Im\langle\xi,\eta \rangle=0 \}.\] 
Then $\sH$ is separating if and only if $\sH'$ is cyclic, hence $\sH$ is standard if and only if $\sH'$ is standard.
For a standard subspace $\sH$, we 
define the {\it Tomita operator} as the closed antilinear involution
\[ \sH+i\sH \to \sH + i \sH,\quad 
\xi+i\eta\mapsto \xi-i\eta. \] 
The polar decomposition $J_\sH\Delta_\sH^{\frac12}$ {of this operator}
defines an antiunitary involution $J_\sH$ (a conjugation) 
and the modular operator~$\Delta_\sH$. 
For the modular group  $(\Delta_\sH^{it})_{t \in \R}$, 
we then have 
\[  J_\sH\sH=\sH', \quad 
 \Delta^{it}_\sH\sH=\sH \qquad \mbox{ for every  } \quad 
t\in \R\] and
the modular relations
\[J_{\sH}\Delta^{it}_{\sH}J_{\sH}=\Delta^{it}_{\sH}\qquad \mbox{ for every  } \quad t\in \R. \]
One also has $\sH=\Fix(J_\sH \Delta_\sH^{1/2})$
(\cite[Thm.~3.4]{Lo08}). This construction 
leads to a one-to-one correspondence between
{couples $(\Delta, J)$ satisfying the modular relation
  and} standard subspaces: 

\begin{proposition}\label{prop:11}{\rm (\cite[Prop.~3.2]{Lo08})} 
{  The map $\sH\mapsto (\Delta_\sH, J_\sH)$ 
  is a bijection between the set of standard subspaces of $\cH$ 
  and the set of pairs $(\Delta, J)$, where $J$ is a conjugation,
  $\Delta > 0$ selfadjoint} with $J\Delta J =\Delta^{-1}$. 
\end{proposition}

From Proposition~\ref{prop:11} we easily deduce: 
\begin{lemma}\label{lem:sym}{\rm(\cite[Lemma 2.2]{Mo18})}
Let $\sH\subset\cH$ be a standard subspace  and $U\in\AU(\cH)$ 
be a {unitary or anti-unitary} operator. 
Then $U\sH$ is also standard and 
$U\Delta_\sH U^*=\Delta_{U\sH}^{\eps(U)}$ and $UJ_\sH U^*=J_{U\sH}$, 
where $\eps(U) = 1$ if $U$ is unitary and 
$\eps(U) = -1$ if it is antiunitary. 
\end{lemma}

\begin{prop} \mlabel{prop:standchar} {\rm(\cite{Lo08},\cite[Prop.~2.1]{NOO21})}
  Let $\sV \subeq \cH$ be a standard subspace 
with modular objects $(\Delta, J)$. For 
$\xi \in \cH$, we consider the orbit map $\alpha^\xi \: \R \to \cH, t \mapsto 
\Delta^{-it/2\pi}\xi$. Then the following are equivalent:
\begin{itemize}
\item[\rm(i)] $\xi \in \sV$. 
\item[\rm(ii)] $\xi \in \cD(\Delta^{1/2})$ with $\Delta^{1/2}\xi = J\xi$. 
\item[\rm(iii)] The orbit map $\alpha^\xi \: \R \to \cH$ 
  extends to a continuous map  $\{z \in \C \: 0 \leq \Im z \leq \pi\}
\to \cH$ which is 
holomorphic on the interior and satisfies $\alpha^\xi(\pi i) = J\xi$. 
\item[\rm(iv)] There exists $\eta \in \cH^J$ 
whose orbit map $\alpha^\eta$ 
extends to a map  $\{ z \in \C \: |\Im z| \leq \pi/2\}
\to \cH$ which is continuous, 
holomorphic on the interior, and satisfies $\alpha^\eta(-\pi i/2) = \xi$. 
\end{itemize}
\end{prop}

\subsubsection{The Brunetti--Guido--Longo (BGL) net}
Here we recall a construction we introduced in \cite{MN21} that generalize the algebraic construction of free fields for AQFT models presented in \cite{BGL02}.

If $(U,G)$ is an (anti-)unitary representation of $G_\sigma$,   
then we obtain a standard subspace $\sH_U(W)$ determined for {$W
  = (h, \tau) \in \cG$}
by the couple of operators (cf.~Proposition \ref{prop:11}):
\begin{equation}
  \label{eq:bgl}
J_{\sH_U(W)} = 
U(\tau) \quad \mbox{ and } \quad \Delta_{\sH_U(W)} = e^{2\pi i \partial U(h)}, 
\end{equation}
and thus a $G$-equivariant map $\sH_U \:  \cG \to \Stand(\cH)$. 
This is the so-called BGL net
  \[ \sH_U^{\rm BGL} \: \cG(G_\sigma) \to \Stand(\cH).\]

In the following theorem, we need the {\it positive cone}
  \begin{equation}
    \label{eq:CU}
 C_U := \{ x \in \g \: -i \cdot \partial U(x) \geq 0\}, \qquad 
 \partial U(x) = \frac{d}{dt}\Big|_{t = 0} U(\exp tx)
  \end{equation}
  of a unitary representation~$U$. It
  is a closed, convex, $\Ad(G)$-invariant cone in $\g$.

  \begin{theorem} \mlabel{thm:BGL}
    Let $C_\g \subeq \g$ be a pointed generating
    closed convex cone contained in the positive cone $C_U$
    of the (anti-)unitary representation $(U,\cH)$ of $G_\sigma$.
    Then the BGL net 
\[ \sH_U^{\rm BGL} \: \cG(G_\sigma) \to \Stand(\cH) \] 
    is $G_\sigma$-covariant and isotone with respect to the $C_\g$-order on     $\cG(G_\sigma)$.
\end{theorem}

  The BGL net also satisfies twisted locality conditions and
  PT symmetry. We refer to \cite{MN21} for a detailed discussion.
{In this picture we have not required $\sigma$ to be an 
Euler involution so $\cG_E(G_\sigma)$ may  in particular be  trivial 
(see Example~\ref{ex:no-euler}). This general presentation 
is motivated by the results in Section~\ref{sect:modEul} 
that will exhibit the existence of an Euler element in $\fg$ 
and an involution $\tau_h^G$, defining a graded 
group $G_{\tau_h}$, as a consequence of 
a certain regularity condition for associated standard subspaces 
in unitary representations of $G$.} 

\begin{ex}
  \mlabel{ex:no-euler}
  It is easy to construct graded groups $G_\sigma$ for which
  $\cG_E(G_\sigma) = \eset$, i.e., no Euler couples exist.
  For example, we may consider
  $G = \SL_2(\R)$ and the involutive automorphism
  $\theta(g) = (g^\top)^{-1}$ (Cartan involution).
  We claim that $G_\theta = G \rtimes \{\1,\theta\}$ contains
  no Euler couples. In fact, if $(h,\tau)$ is an Euler couple,
  then $\Ad(\tau) = \tau_h$.
  Identifying the Lie algebra $\fsl_2(\R)$, endowed with its
  Cartan--Killing form, with $3$-dimensional Minkowski space
  $\R^{1,2}$, we have $\Ad(G) = \Ad(G_\theta)  = \SO_{1,2}(\R)_e$,
  a connected group. But the automorphisms $\tau_h$
  are contained in $\SO_{1,2}(\R)^\down$ because they reverse the causal
  orientation. Hence no involution $\tau = (g,\theta) \in G_\theta$
  satisfies $\Ad(\tau) = \tau_h$. Clearly, the picture changes
  if we replace $\theta$ by an involution $\tau_h^G$, where $h \in \fsl_2(\R)
  \cong \so_{1,2}(\R)$ is an Euler element. 
\end{ex}

\subsubsection{Nets on homogeneous spaces}
\label{sect:onenet}

For a unitary representation $(U,\cH)$ of a connected a Lie group~$G$ 
and a homogeneous space $M = G/H$, we are interested in
families $(\sH(\cO))_{\cO \subeq M}$ of closed real subspaces of $\cH$, 
indexed by open subsets $\cO \subeq M$;
so-called {\it nets of real subspaces on $M$}. 
Below we work in a more general context, where
the connection between the abstract and the geometric wedges
is less strict. 
For such nets, we
consider the following properties:
\begin{itemize}
\item[(Iso)] {\bf Isotony:} $\cO_1 \subeq \cO_2$ 
implies $\sH(\cO_1) \subeq \sH(\cO_2)$ 
\item[(Cov)] {\bf Covariance:} $U(g) \sH(\cO) = \sH(g\cO)$ for $g \in G$. 
\item[(RS)] {\bf Reeh--Schlieder property:} 
$\sH(\cO)$ is cyclic  if $\cO \not=\eset$. 
\item[(BW)] {\bf Bisognano--Wichmann property:}
{There exists an open subset $W \subeq M$ (called a {\it wedge region}),  
such that $\sH(W)$ is standard
with modular operator $\Delta_{\sH(W)} = e^{2\pi i \partial U(h)}$
  for some $h \in \g$.}
\end{itemize}
Nets satisfying (Iso), (Cov), (RS), (BW) 
on non-compactly causal symmetric spaces have been constructed on
  non-compactly causal symmetric spaces in \cite{FNO23}, 
  and on compactly causal spaces in \cite{NO23}. 

In some cases there is a one-to-one
correspondence between the abstract wedge space
$\cW_+ \subeq \cG_E(G_\sigma)$
  and the set $\cW_M := \{g.W \: g \in G\}$
of wedge regions in $M$, see Remark \ref{rem:1-1}.
In these cases, the BGL net on $\cW_+$ can be considered
as a net on concrete wedge regions in $M$,
satisfying the previous assumptions, on the
set $\cW_M$ of wedge regions in $M$.
A general correspondence theorem still has to be established.
If $\sV$ is a standard subspace with
  $\Delta_\sV = e^{2\pi i \partial U(h)}$, then
  $\sH(g.W) := U(g) \sV$ yields a well-defined net on $\cW_M$
  if $g.W = W$ implies $U(g)\sV = \sV$.
 If $\ker U$ is discrete, the latter condition means that
  $\Ad(g)h = h$ and $U(g) J_\sV U(g)^{-1} = J_\sV$.}

\subsubsection{Minimal and maximal nets of real subspaces}\label{sect:maxnet}

To add a geometric context to the nets of standard subspaces
  that we have already encountered in terms of the BGL construction
  (cf.\ Theorem~\ref{thm:BGL}),
  we now fix an Euler element $h\in \g$ and
a homogeneous space $M = G/H$ of $G$, in which we 
consider an open subset $W$ invariant under the one-parameter group
$\exp(\R h)$. We call $W$ and its translates $gW$, $g \in G$,
  ``wedge regions''.
At the outset,  we do not assume any specific properties
of~$W$, but Lemma~\ref{lem:direct-net} will indicate which
properties good choices of $W$ should have. 
Let $(U,\cH)$ be an (anti-)unitary representation of $G_{\tau_h}$ and 
$\sV = \sV(h,U)$ the corresponding standard subspace. 
For an open subset $\cO \subeq M$, we put
\begin{equation}
  \label{eq:def-ho}
  \sH^{\m}(\cO) := \bigcap_{g\in G, \cO \subeq gW} U(g)\sV
  \quad \mbox{ and } \quad 
{  \sH^{\rm min}(\cO) := \oline{\sum_{g\in G, gW \subeq \cO} U(g)\sV}.}
\end{equation}
We call {$\sH^{\rm max}$ the {\it maximal net},
  in accordance} with \cite{SW87}.

  This leads to $\sH^{\m}(\cO) = \cH$ (the empty intersection) if there exists
no $g \in G$ with $\cO \subeq gW$, i.e., $\cO$ is not contained in
any wedge region. We likewise get
$\sH^{\rm min}(\cO) := \{0\}$ (the empty sum) if there exists
no $g \in G$ with $gW \subeq \cO$, i.e., $\cO$ contains no
wedge region.

We also note that, if we write
  \[ \cO^\wedge := \Big(\bigcap_{gW \supeq \cO} gW\Big)^\circ \supeq \cO
    \quad \mbox{ and } \quad 
    \cO^\vee := \bigcup_{gW \subeq \cO} gW \subeq \cO,\]
  then $\cO^\wedge$ and $\cO^\vee$ are open subsets satisfying
  $(\cO^\wedge)^\wedge = \cO^\wedge$, $(\cO^\vee)^\vee = \cO^\vee$, and 
  \begin{equation}
    \label{eq:enla}
 \sH^{\rm max}(\cO^\wedge) = \sH^\m(\cO) \quad \mbox{ and }  \quad   
 \sH^{\rm min}(\cO^\vee) = \sH^{\rm min}(\cO).
  \end{equation}
  So, effectively, the maximal net ``lives'' on all open subsets $\cO$
  satisfying $\cO = \cO^\wedge$ (interiors of intersections of
  wedge regions)
  and the minimal net on those open subsets
satisfying $\cO = \cO^\vee$ (unions of wedge regions).

\begin{lem} \mlabel{lem:direct-net}
    The following assertions hold:
  \begin{itemize}
  \item[\rm(a)] The nets $\sH^{\m}$ { and $\sH^{\rm min}$} 
    on $M$ satisfy {\rm(Iso)} and {\rm(Cov)}. 
  \item[\rm(b)] The set of all open subsets $\cO \subeq M$
  for which $\sH^{\m}(\cO)$ is cyclic is $G$-invariant.
\item[\rm(c)] The following are equivalent:
  \begin{itemize}
  \item[\rm(i)] ${S_W := \{ g \in G \:  gW \subeq W \}} \subeq S_\sV$.
  \item[\rm(ii)] $\sH^{\m}(W) = \sV$. 
  \item[\rm(iii)] $\sH^{\m}(W)$ is standard. 
  \item[\rm(iv)] $\sH^{\m}(W)$ is cyclic. 
{  \item[\rm(v)] $\sH^{\rm min}(W) = \sV$. 
  \item[\rm(vi)] $\sH^{\rm min}(W)$ is standard. 
  \item[\rm(vii)] $\sH^{\rm min}(W)$ is separating. }
  \end{itemize}
\item[\rm(d)] The cyclicity of a subspace $\sH^{\m}(\cO)$
  is inherited by subrepresentations, direct sums, direct integrals
  and finite tensor products.
\end{itemize}
\end{lem}

\begin{prf} (a) Isotony is clear and covariance of the maximal net
  follows from
  \[ \sH^{\m}(g_0\cO)
    = \bigcap_{g_0\cO \subeq gW} U(g)\sV
    = U(g_0)\bigcap_{g_0\cO \subeq gW} U(g_0^{-1} g)\sV
    = U(g_0)\sH^{\m}(\cO).\]
{The argument for the minimal net is similar.} 

  \nin (b) follows from covariance.

  \nin (c) (i) $\Leftrightarrow$ (ii): Clearly,
    $\sH^{\rm max}(W) \subeq \sV$, and equality holds if and only if
    $W \subeq gW$ implies $U(g)\sV \supeq \sV$,
    which is equivalent to $S_W^{-1} \subeq S_\sV^{-1}$, and this is equivalent
    to~(i).

 (ii) $\Rarrow$ (iii) $\Rarrow$ (iv) are trivial.

(iv) $\Rarrow$ (ii): By covariance
 and $\exp(\R h).W = W$, the subspace $\sH^{\m}(W) \subeq \sV$
  is invariant under the modular group $U(\exp \R h)$ of $\sV$.
  If $\sH^{\m}(W)$ is cyclic, then it is also standard, as a subspace of $\sV$,
  so that \cite[Prop.~3.10]{Lo08} implies $\sH^{\m}(W) = \sV$.

(i) $\Leftrightarrow$ (v) follows with a similar argument
    as the equivalence of (i) and (ii).

 (v) $\Rarrow$ (vi) $\Rarrow$ (vii) are trivial.

 (vii) $\Rarrow$ (v): By covariance
 and $\exp(\R h).W = W$, the subspace
  $\sH^{\rm min}(W) \supeq \sV$
  is invariant under the modular group $U(\exp \R h)$ of $\sV$.
  If $\sH^{\rm min}(W)$ is separating,
  then it is also standard, because it contains~$\sV$.
  Now \cite[Prop.~3.10]{Lo08} implies $\sH^{\rm min}(W) = \sV$.

\nin (d) We use that
  \begin{equation}
    \label{eq:sh=va}
 \sH^{\m}(\cO) = \sV_A \quad \mbox{ for } \quad
A := \{ g \in G \: g^{-1} \cO \subeq W \}.
  \end{equation}
  Now \eqref{eq:v-dirsum} implies that, for $U = U_1 \oplus U_2$, we have
  \[ \sH^{\m}(\cO) = \sH_{1}^{\m}(\cO) \oplus \sH_{2}^{\m}(\cO).\]
  This proves that cyclicity of $\sH^{\m}(\cO)$ is inherited
  by subrepresentations and direct sums. 
  For finite tensor products, the assertion 
  follows from Lemma~\ref{lem:tensprostand}.
  If $U = \int_X^\oplus U_m\, d\mu(m)$ is a direct integral,
  then  \eqref{eq:sh=va} and   Lemma~\ref{lem:g-inter}(a) 
imply that \begin{equation}
  \label{eq:net-dirint}
  \sH^{\m}(\cO) = \int_X^\oplus \sH_{m}^{\m}(\cO)\, d\mu(m)
\end{equation}
for direct integrals. 
So Lemma~\ref{lem:di1} implies that $\sH^{\m}(\cO)$ is cyclic if 
every $\sH_{m}^{\m}(\cO)$ is cyclic in~$\cH_m$.
 \end{prf}

  Lemma~\ref{lem:direct-net}(d)
  implies in particular that a direct integral representation
  $(U,\cH)$ is $(h,W)$-localizable 
  in a family of subsets of $M$
  in the sense of Definition~\ref{def:local}
  if $\mu$-almost all representations $(U_m, \cH_m)$ have this property.
For the case where $G$ is the Poincar\'e group and $M = \R^{1,d}$,  
a similar argument can be found in \cite[Lemma~4.3]{BGL02}.

\begin{rem} {(The case where $S_W$ is a group)} 
  If the semigroup $S_W$ {is a group, i.e.,}
  $S_W = G_W = \{g \in G \: g.W = W\}$
  is a group and $\ker(U)$ is discrete, then the inclusion 
   $S_W \subeq S_\sV$ is equivalent to
   {   \begin{equation}
     \label{eq:gwx}
     G_W \subeq G_\sV = G^{h,J} = \{ g \in G^h \: J U(g) J = U(g)\}
   \end{equation}
(cf.\ Lemma~\ref{lem:sym}).}   
 In the context of causal homogeneous spaces, the 
     definition of $W$ as a connected component of
     $W_M^+(h)$ {(see \S~\ref{sect:chs})
     implies that $\exp(\R h) \subeq G^h_e \subeq G_W$,} and we 
have in many concrete examples that $G_W \subeq G^h$
 and $\L(G_W) =  \g^h$ (see \cite{NO22, NO23, MNO23b}
 and \S\S~\ref{subsec:ncc} and \ref{subsec:cc}).
 However, $U(G_W)$ need not commute with~$J$,
 so that \eqref{eq:gwx} may fail.
   Examples arise already for $\g = \fsl_2(\R)$; see
   \cite[Rem.~5.13]{FNO23}.  
 \end{rem}

\begin{lemma} \mlabel{lem:maxnet-larger}
  Let $(U,\cH)$ be an (anti-)unitary representation of $G_{\tau_h}$
  and $\sH$ a net of real subspaces on open subsets of $M$ satisfying
  {\rm (Iso), (Cov)} {and $\sH(W) = \sV$ with respect to 
    $h \in \g$ and $W \subeq M$.} 
Then   
  \[ \sH^{\rm min}(\cO) \subeq \sH(\cO) \subeq \sH^{\rm max}(\cO) \]
  for each  open subset $\cO \subeq M$
  and equality holds for all domains of the form
  $\cO = g.W$, $g \in G$ (wedge regions in $M$).
 \end{lemma}

If $\eset \not= W \not= M$, then we have in particular
 \[  \sH^{\rm min}(\eset) = \{0\} \subeq
   \sH^{\rm max}(\eset) = \bigcap_{g \in G} U(g)\sV \quad \mbox{ and } \quad 
   \sH^{\rm min}(M) = \oline{\sum_{g \in G} U(g)\sV} \subeq
   \sH^{\rm max}(M) = \cH.\]

 \begin{proof} First we show that the three properties
   (Iso), (Cov) and {$\sH(W) = \sV$} of the net $\sH$
   imply that $S_W \subeq S_\sV$.
   In fact, $g.W \subeq W$ implies
    \[ U(g)\sV
      = \ U(g) \sH(W)
     \ {\buildrel{\rm(Cov)}\over= }\ \sH(g.W)
     \ {\buildrel{\rm(Iso)}\over \subeq }\ \sH(W)
     =  \sV.\]
 From Lemma~\ref{lem:direct-net}(c) we thus obtain
$\sH^{\rm max}(W)= \sH^{\rm min}(W) = \sV$. 
     Hence $\sH(gW) = U(g) \sV = \sH^{\rm max}(gW) = \sH^{\rm min}(gW)$
     by covariance for any $g \in G$ (Lemma~\ref{lem:direct-net}(a)).
     Further, isotony shows that $\cO \subeq g W$ implies
     $\sH(\cO) \subeq \sH(gW) = U(g)\sV$, so that
     $\sH(\cO) \subeq \sH^{\rm max}(\cO)$.
     Likewise, $gW \subeq \cO$ implies
     $U(g)\sV  = \sH(gW) \subeq \sH(\cO)$, and thus
     $\sH^{\rm min}(\cO) \subeq \sH(\cO)$. 
\end{proof}

\begin{defn} \label{def:minkcaus}
(a)     {\rm(Causal complement)}
Let $M=\RR^{1,d}$ be Minkowski space. Its causal structure allows us to
define the \textit{causal complement (or the  spacelike complement)} of
an open subset $\cO\subset M$ by 
\begin{equation}\label{eq:ccomp}
  \cO'=\{x\in M: (\forall y \in \cO)\, (y-x)^2<0\}^\circ. \end{equation}
This is the interior of the set of all the points that cannot
be reached from $E$ with a timelike or lightlike curve. 

\nin (b)  {\rm(Spacelike cones)}
  In Minkowski space $\R^{1,d}$, we call an open subset
  $\cO$ {\it spacelike} if $x_0^2 < \bx^2$ holds for all $(x_0,\bx) \in \cO$.
  A spacelike open subset is called a {\it spacelike
    (convex) cone} if, in addition, it is a (convex) cone.
       
\nin (c)  {\rm(Double cone)} A {\it double cone}  is, up to Poincar\'e covariance,
the causal closure
\[ \bB_r'' = (r \be_0 - V_+) \cap (-r \be_0 + V_+) \] 
  of an open ball of the time zero hyper-plane
    $\bB_r=\{(0,\bx) \in \RR^{1,d}: \bx^2< r^2\}$.
\end{defn}

\begin{remark} \mlabel{rem:x} 
We continue to use the notation from Example \ref{ex:desit} and
  Definition~\ref{def:minkcaus}. 
Let $d\geq2$ and $M \supset \cD\mapsto \sH(\cD)\subset\cH$  be a net of standard subspaces on double cones (cf.~Definition~\ref{def:minkcaus}(c)), let 
$U$ be a representation of the Poincar\'e group $\cP_+^\uparrow$
satisfying (Iso), (Cov), (RS) and the following properties
\begin{enumerate}
\item\textit{Positivity of the energy:} 
The support of the spectral measure of the space-time
    translation group is contained in  
  \[ \oline{V_+}=\{(x_0,\bx) \in\RR^{1,d}:x_0^2 - \bx^2 \geq 0,x_0 \geq 0\}.\]
\item Locality: $\cD_1\subset \cD'_2 \Rarrow
  \sH(\cD_1)\subset\sH(\cD_2)'. $

\item Bisognano--Wichmann property: Let $W\subset M$ be a wedge
  region, as introduced in \ref{ex:desit}. Then 
  \begin{equation}\label{eq:Wsub}\sH(W)=\overline{\sum_{\cD\subset W}\sH(\cD)},\end{equation} 
  is standard with $\Delta^{-it/2\pi}_{\sH(W)}=U(\Lambda_W(t))$,
  where $\Lambda_W(t)$ is the corresponding one-parameter group
of boosts    (cf.\ Example~\ref{ex:desit}).
\end{enumerate}
The Bisognano--Wichmann property implies
\textit{wedge duality (or essential duality)}:
\[ \sH(W') = \sH(W)'. \]
Here $W'$ is the causal complement of the wedge $W$,
as in \eqref{eq:ccomp} (see \cite[Prop.~2.7]{Mo18}). 

For a  double cone $\cD$ we define 
\begin{equation}
  \label{eq:dual-dob}
  \sH(\cD'):=\overline{\sum_{\cD_1\subset\cD'}\sH(\cD_1)}
\end{equation}
  and obtain the following net  on double cones
\[M\supset\cD\longmapsto\sH^d(\cD):=\sH(\cD')'
= \bigcap_{\cD_1\subset\cD'}\sH(\cD_1)'. \]
By locality one has in general that $\sH(\cD)\subset\sH^d(\cD)$. 
The net $\sH^d(\cD)$ is called the \textit{dual net} of~$\sH$. 
If $\sH(\cD)=\sH^d(\cD)$, then the net $\sH$ is said to satisfy
\textit{ Haag duality}. 
Given two relatively spacelike double cones $\cD_1$ and $\cD_2$,
there always exists a wedge region $W$ such that  $\cD_1\subset W$
and $\cD_2\subset W'$ (\cite[Prop.~3.1]{TW97}). For
every double cone $\cD$, we further have $\cD=\bigcap_{W\supset \cD}W$.
As a consequence $\sH(\cD')=\overline{\sum_{W\supset \cD'}\sH(W)}$ (with the definition of $\sH(W)$ given in \eqref{eq:Wsub}). With respect to
  $\sV = \sH(W_R)$, this leads to 
  \[ \sH^{\rm min}(\cD') = \sH(\cD') \quad \mbox{ and } \quad
    \sH^d(\cD) = \sH^{\rm min}(\cD')'.\]
 We further obtain
$$\sH^d(\cD)=\bigcap_{W\supset \cD}\sH(W)
=\bigcap_{g\in \cP_+^\uparrow, gW_R\supset\cD}\sH(gW_R)=\sH^\m(\cD).$$

For the case $d=1$ one still has
\[ \sH^d(\cD)=\bigcap_{W\supset \cD}\sH(W)
  =\bigcap_{g\in G, gW_R\supset\cD}\sH(gW_R)=\sH^\m(\cD),\]
but, to this end, one has to consider the maximal net with respect
to a unitary representation $(U,\cH)$ of the group 
$G=\cP^\up=\langle\cP_+^\uparrow, r\rangle$,
where $r(x_0,x_1)=(x_0,-x_{1})$ and $\sH$ is also covariant for~$U(r)$.
Indeed, every double cone is the intersection of $W_R+a$ and $W_R'+b$ for some $a,b\in\RR^{1,d}$, but $W_R$ and $W_R'$ belong to disjoint orbits of wedges with respect to $\cP_+^\uparrow$. However,
they belong to the same orbit 
of $\cP^\up$ because $W_R'=rW_R$.  

Alternatively, starting with a unitary representation
  $(U,\cH)$ of $\cP^\up_+$ for which $\sH$ is covariant, we can use Theorem \ref{thm:loc-net}
  to extend $U$ to an (anti-)unitary representation of
  $\cP_+$ by $U(\tau_h) := J_{\sH(W_R)}$. Then
      $\cP_+$ acts covariantly on the net on wedge regions.\footnote{One
    can also argue with Borchers' Theorem,
    positivity of the energy and the Bisognano--Wichmann property.}
    Hence $\tau_hW_R=W_R'$ implies  the equality
$$\bigcap_{W\supset \cD}\sH(W)=\bigcap_{g\in \cP_+, gW_R\supset\cD}U(g)\sH(W_R)  =\bigcap_{g\in \cP_+, gW_R\supset\cD}\sH(gW_R)=:\tilde\sH^\m(\cD),$$
where $\tilde\sH^\m(\cD)$ now is defined with respect to the
(anti-)unitary representation  of~$\cP_+$.
If both constructions apply, then $\sH^{\rm max}(\cD) = \tilde\sH^{\rm max}(\cD)$. 

We can conclude a correspondence between the maximal net construction and the dual net construction but, since  we will not deal with locality in this paper,
a more detailed analysis is postponed to future works.
\end{remark}

\subsubsection{Intersections of standard subspaces} \label{sect:interss}

\nin\textbf{Standing assumption in the remainder of this section:}
Let $G$ be a connected Lie group with
    Lie algebra $\g$ and $h \in \g$ an Euler element. 
    Assume that the involution $\tau_h$ integrates to an involution
    $\tau_h^G$ on $G$. For an   (anti-)unitary representation $(U,\cH)$ of
$G_{\tau_h} := G \rtimes \{\id_G, \tau_h^G\}$, we call
  \begin{equation}
    \label{eq:def-V(h,U)}
    \sV := \sV(h,U) := \sH_U^{\rm BGL}(h,\tau_h^G) 
  \end{equation}
  the {\it canonical standard subspace associated to $(h,U)$}.
{Its modular objects are $J = U(\tau_h^G)$ and $\Delta = e^{2\pi i \partial U(h)}$.}

For a subset $A \subeq G$, we consider the closed real subspace 
  \begin{equation}
    \label{eq:def-VA(h,U)}
    \sV_A := \sV_A(h,U) := \bigcap_{g \in A} U(g)\sV.
  \end{equation}
We shall be interested in criteria for these real subspaces to be cyclic.  
  An important property of these subspaces is that they
 are well adapted to direct sums and direct integrals
 (Lemma~\ref{lem:g-inter}). For a direct sum representation
 $U = U_1 \oplus U_2$ we have in particular
 $\sV = \sV_1 \oplus \sV_2$, which leads to
 \begin{equation}
   \label{eq:v-dirsum}
\sV_A = \sV_{1,A} \oplus \sV_{2,A}    
\end{equation}
because $U(g)^{-1}(v_1, v_2) \in \sV$ is equivalent to
$U_j(g)^{-1}v_j \in \sV_j$ for $j= 1,2$. 

These concepts require (anti-)unitary representations of
$G_{\tau_h}$, but often unitary representations of $G$
are easier to deal with. The following lemma 
translates between unitary and (anti-)unitary representations and
their properties. It is our version of a closely related technique
developed in \cite[Props.~4.1, 4.2]{BGL02},  which is based on
density properties of intersections of dense complex
subspaces of~$\cH$.

\begin{lem} {\rm(The (anti-)unitary extension)}
  \mlabel{lem:3.4}
  Let $(U,\cH)$ be a unitary representation of $G$
  and write~$\oline\cH$ for the Hilbert space $\cH$, endowed with the
  opposite complex structure. Then the following assertions hold:
  \begin{itemize}
  \item[\rm(a)] On $\tilde\cH := \cH \oplus \oline\cH$ we obtain by
    $\tilde U(g) := U(g) \oplus U(\tau_h^G(g))$ a unitary representation
    which extends by $\tilde U(\tau_h)(v,w) := \tilde J(v,w) := (w,v)$ to an
    (anti-)unitary representation of $G_{\tau_h}$.
    The corresponding standard subspace $\tilde \sV := \sV(h, \tilde U)$
    coincides with the graph
    \begin{equation}
      \label{eq:tildesv}
      \tilde\sV = \Gamma(\Delta^{1/2}), 
    \end{equation}
    and its modular operator is $\tilde\Delta := \Delta \oplus \Delta^{-1}$.
  \item[\rm(b)] If $U$ extends to an (anti-)unitary representation
    of $G_{\tau_h}$ by $J = U(\tau_h)$, then the following assertions hold:
    \begin{itemize}
    \item[\rm(1)] $\Phi \:  \cH^{\oplus 2} \to \tilde \cH, \Phi(v,w) = (v,Jw)$ 
    is a unitary intertwiner of $\tilde U$ and the (anti-)unitary representation
    $U^\sharp$ of $G_{\tau_h}$ on $\cH^{\oplus 2}$, given by
    \[ U^\sharp\res_G = U^{\oplus 2} \quad \mbox{ and }  \quad
      U^\sharp(\tau_h)(v,w) := J^\sharp(v,w) :=  (Jw,Jv).\]
  \item[\rm(2)] The standard subspace $\sV^\sharp := \sV(h,U^\sharp)$
    coincides with the
    graph $\Gamma(T_\sV)$ of the Tomita operator
    $T_\sV = J\Delta^{1/2}$ of~$\sV$.
    \item[\rm(3)] 
    The (anti-)unitary representation $\tilde U$ is equivalent to the (anti-)unitary
    representation $U^{\oplus 2}$ of $G_{\tau_h}$ on~$\cH^{\oplus 2}$.
    \item[\rm(4)]  If $A \subeq G$ is a subset, then 
    $\tilde\sV_A$ is cyclic in $\tilde\cH$ if and only if
    $\sV_A$ is cyclic in $\cH$.     
    \end{itemize}
  \end{itemize}
\end{lem}

\begin{prf} (a) The first assertion
  is a direct verification (cf.\ \cite[Lemma~2.10]{NO17}).
  Since
  \[ \tilde\Delta = e^{2\pi i \partial \tilde U(h)}
    = \Delta \oplus \Delta^{-1}, \]
  the description of the standard subspace $\tilde \sV = \Fix(\tilde J
  \tilde\Delta^{1/2})$ follows immediately.  

  \nin (b)  (1) Clearly, $\Phi$ is a complex linear isometry that intertwines 
  the (anti-)unitary representation $\tilde U$ with
  the (anti-)unitary representation $U^\sharp$.

  (2)  As $\Delta^\sharp = \Phi^{-1}\tilde\Delta \Phi = \Delta \oplus \Delta$,
  the relation
  \[ (v,w) = J^\sharp(\Delta^\sharp)^{1/2}(v,w)
    = (J \Delta^{1/2} w, J \Delta^{1/2} v) = (T_\sV w, T_\sV v) \]
  is equivalent to $w = T_\sV v$. Hence $\sV^\sharp = \Gamma(T_\sV)$.

(3) As the restrictions of $U^{\oplus 2}$ and $U^\sharp$ to $G$ coincide,
  \cite[Thm.~2.11]{NO17} implies their equivalence as (anti-)unitary
  representations. However, in the present concrete case it is easy to
  see an intertwining operator. The matrix
  \[  A := \frac{1}{2} \pmat{
      (1+i)\1 & (1-i)\1 \\ 
      (1-i)\1 & (1+i)\1}
  \quad \mbox{ with } \quad A^2 = \pmat{ \0 & \1 \\ \1 & \0} \]
  defines a unitary operator on $\cH^{\oplus 2}$ commuting with $U^\sharp(G)$.
  It satisfies $J^{\oplus 2} A J^{\oplus 2} = A^* = A^{-1}$, so that
  \[ A J^{\oplus 2} A^{-1} = A^2 J^{\oplus 2} = J^\sharp.\] 

(4) If $U\res_G$ extends to an (anti-)unitary representation~$U$ of
  $G_{\tau_h}$ on $\cH$, then (3)  implies that $\tilde U \cong U^{\oplus 2}$,
  and any equivalence $\Psi \: (\tilde U,\tilde \cH)  \to (U^{\oplus 2}, \cH^{\oplus 2})$ maps $\tilde \sV_A$ to
  $(\sV \oplus \sV)_A = \sV_A \oplus \sV_A$
  (see \eqref{eq:v-dirsum}).
  Therefore $\tilde \sV_A$ is cyclic if and only if
  $\sV_A$ is cyclic in~$\cH$.
\end{prf}

 The following definition extends the classical type of irreducible
  complex representations to the case where the involution on $G$ is
  non-trivial. For a unitary representation
    $(U,\cH)$, we write $(\oline U, \oline\cH)$ for 
    the canonical unitary representation on the complex conjugate space
  $\oline\cH$ by $\oline U(g) = U(g)$. 
  We observe that, for an (anti-)unitary representation
  $(U,\cH)$ of $G_{\tau_h}$, its {\it commutant} 
  \[ U(G_{\tau_h})'
    = \{ A \in B(\cH) \: (\forall g \in G_{\tau_h})
    \, A U(g) = U(g) A \} 
{= \{ A \in U(G)' \:  U(\tau_h^G) A = AU(\tau_h^G) \}}\]
is only a real subalgebra of $B(\cH)$ because some $U(g)$ are antilinear.

  \begin{defn} \mlabel{def:types} (\cite[Def.~2.12]{NO17})
    Let $(U, \cH)$ be an irreducible unitary representation
    of~$G$. 
We say that $U$ is (with respect to $\tau_h$), of 
\begin{itemize}
\item {\it real type} if there exists an antiunitary 
  involution $J$ on $\cH$ such that {$U^\sharp(\tau_h) := J$
  extends $U$ to an (anti-)unitary representation $U^\sharp$
  of $G_{\tau_h}$ on~$\cH$,
  i.e., $J U(g) J = U(\tau_h^G(g))$ for $g \in G$. 
  Then the commutant of $U^\sharp(G_{\tau_h})$ is~$\R$.}
\item {\it quaternionic type} if there exists an antiunitary 
  complex structure $I$ on $\cH$ satisfying $I U(g) I^{-1}
  = U(\tau_h^G(g))$ for $g \in G$. Then $\oline U \circ \tau_h^G \cong U$, 
  $U$ has no extension  on the same space, and
  the (anti-)unitary representation
  $(\tilde U, \tilde \cH)$ of $G_{\tau_h}$ with $\tilde U\res_G \cong
  U \oplus (\oline U\circ \tau_h^G)$ is irreducible with commutant $\bH$.
\item {\it complex type} if
  {$\oline U \circ \tau_h^G \not\cong U$.
    This is equivalent to the non-existence of $V\in\AU(\cH)$
    such that $U(\tau_h^G(g)) = V U(g) V^{-1}$ for all $g \in G$.}
  Then $(\tilde U,\tilde \cH)$ is an irreducible (anti-)unitary 
  representation of $G_{\tau_h}$ with commutant $\C$. 
\end{itemize}
  \end{defn}

  \begin{ex}  (a) {On the Poincar\'e group 
    $\cP = \R^{1,d} \rtimes \cL^\up_+$ we consider the involution 
      $\tau_h^G(g) = j_h g j_h$,
    corresponding to conjugation with
    \[ j_h(x_0,x_1,\ldots, x_d) = (-x_0, -x_1, x_2, \ldots, x_d),\]
    so that $\cP_{\tau_h} \cong \cP_+$.
    Then all irreducible positive energy representations of $\cP$   }
  are of real type except the massless finite helicity representations 
that are of complex type {(see \cite[App.~A]{Mu01} for $m > 0$, and 
 \cite[Thm.~9.10]{Va85} for the general case)}.

{\nin (b) (cf.\ \cite[Ex.~2.16(c)]{NO17})
Consider the irreducible unitary representation of 
$G = \SU_2(\C)\cong \Spin_3(\R)$
on $\C^2 \cong \bH$ (by left multiplication) 
where the complex structure on $\bH$ 
is defined by the right multiplication with~$\C$.
This representation
is of quaternionic type with respect to {$\sigma = \id$}, but of
real type with respect to the involution $\sigma(g) = \oline g$. 
}  \end{ex}

    \begin{rem} \mlabel{rem:tenspro-anti} (Antiunitary tensor products) 
    Let $G = G_1 \times G_2$ be a product of type I groups  
  and $\tau$ an involutive automorphism of $G$ preserving both factors, i.e., 
$\tau = \tau_1 \times \tau_2.$ 
  We want to describe irreducible (anti-)unitary representations
  $(U,\cH)$  of the group $G_\tau = G \rtimes \{\id_G, \tau\}$
  using \cite[Thm.~2.11(d)]{NO17}.
  
  \nin (a) The first possibility is that $U\res_G$ is irreducible, so
  that $U(G)' \cong \R$. Then
\[ (U\res_G,\cH) \cong (U_1,\cH_1) \otimes (U_2, \cH_2) \]
with irreducible unitary representations $(U_j, \cH_j)$ of $G_j$
both extending to (anti-)unitary representations $U_j^\sharp$ of $G_j$.
{Hence both $U_1$ and $U_2$ are of real type. }

\nin (b) The second possibility is that $U\res_G$ is reducible
with $U(G)' \cong \C$ or $\bH$, so that
{\[ U\res_G \cong V \oplus (\oline V \circ \tau), \]}
where $(V,\cK)$ is an irreducible unitary representation of $G$
of complex or quaternionic type.
Now $V = U_1 \otimes U_2$, and thus
{
  \[ \cH \cong (\cH_1 \otimes \cH_2) \oplus
  (\oline\cH_1 \otimes \oline\cH_2), \quad 
  U\res_G \cong (U_1 \otimes U_2)  \oplus
  (\oline{U_1} \circ \tau_1 \otimes 
  \oline{U_2} \circ \tau_2).\]}

If $U_j$ is of complex type, then {$\oline{U_j} \circ \tau_j \not\cong U_j$}
implies that $V$ is of complex type. If both $U_1$ and $U_2$ are
of quaternionic type, then {$\oline{U_j} \circ \tau_j \cong U_j$}
for $j = 1,2$ implies {$\oline V \circ \tau \cong V$,} so that
$V$ is of quaternionic type.
\end{rem}

\begin{prop} \mlabel{prop:3.2}
  Assume that $G$ has at most countably many connected
  components and that $A \subeq G$ is a subset. Then the following
  are equivalent:
  \begin{itemize}
  \item[\rm(a)] For all (anti-)unitary representations 
    $(U,\cH)$ of $G_{\tau_h}$, the subspace $\sV_A$ is cyclic. 
  \item[\rm(b)] For all irreducible (anti-)unitary representations
    $(U,\cH)$ of $G_{\tau_h}$, the subspace $\sV_A$ is cyclic. 
  \item[\rm(c)] For all irreducible unitary representations 
    $(U,\cH)$ of $G$, the subspace $\tilde\sV_A$ is cyclic
    in $\tilde \cH$. 
  \item[\rm(d)] For all unitary representations 
    $(U,\cH)$ of $G$, the subspace $\tilde\sV_A$ is cyclic in $\tilde\cH$
  \end{itemize}
\end{prop}

\begin{prf} (a) $\Rarrow$ (b) is trivial. 

  \nin (b) $\Rarrow$ (c): Let $(U,\cH)$ be an irreducible
  unitary representation and $(\tilde U,\tilde \cH)$ its natural
  (anti-)unitary extension. Then either
  $\tilde U$ is an irreducible (anti-)unitary representations
  (if $U$ is of complex or quaternionic type) or
  a direct sum of two irreducible representations
  (if $U$ is of real type) (cf.\ Definition~\ref{def:types}). 
  In view of \eqref{eq:v-dirsum},
  the cyclicity of $\sV_A$ is inherited by direct sums, so that 
  (c) follows from (b).

  \nin (c) $\Rarrow$ (d): Let $(U,\cH)$ be a
  unitary representation of $G$. Decomposing
  $U$ into a direct sum of cyclic representations,
  we may assume that $U$ is cyclic, hence that $\cH$
  is separable. Using \cite[Thm.~8.5.2, \S 18.7]{Di64}, we can write
  $U$ as a direct   integral
  \[ U = \int^\oplus_X U_x\, d\mu(x)\]
  of irreducible representations $(U_x)_{x \in X}$.
  Then
  \[ \tilde U = \int^\oplus_X \tilde U_x\, d\mu(x)\]
  implies that $\tilde \sV_A = \int^\oplus_X \tilde \sV_{x,A}\, d\mu(x)$ 
  by \eqref{eq:v-dirint} and
  Lemma~\ref{lem:g-inter}(a). Further, Lemma~\ref{lem:g-inter}(b)
  implies that $\tilde \sV_A$ is cyclic because all subspaces
  $\tilde \sV_{x,A}$ are cyclic by~(c).

  \nin (d) $\Rarrow$ (a): If $(U,\cH)$ is an (anti-)unitary
  representation of $G_{\tau_h}$, then its restriction to $G$ has
  an (anti-)unitary extension   $(\tilde U, \tilde \cH)$
  which by Lemma~\ref{lem:3.4}(b)(1) is equivalent to $U^{\oplus 2}$.
  Hence the cyclicity of $\tilde \sV_A \cong \sV_A \oplus \sV_A$ implies that
  $\sV_A$ is cyclic.   
\end{prf}

\section{Modular groups are generated by  Euler elements}
\label{sect:modEul}

In this section we show that, if the modular group of a
standard subspace $\sV$ is obtained from a unitary representation
of a finite-dimensional Lie group $G$ and a certain regularity
condition is satisfied, then its infinitesimal generator is an
Euler element $h \in \g$ and the modular conjugation~$J_\sV$ induces on $G$
the involution corresponding to $\tau_h = e^{\pi i \ad h}$ on~$\g$
(Theorem~\ref{thm:2.1} in Section~\ref{subsec:3.1}).
In Subsection~\ref{subsec:strich-opalg} we describe the
implications of this result in the context of operator algebras
with cyclic separating vectors 
(Theorem~\ref{thm:2.1-alg}). In this context, we also obtain an
explicit description of the identity component
of the subsemigroup $S_\cM$ of $G$ leaving a von Neumann algebra
$\cM$ invariant.

\subsection{{The Euler Element Theorem}}
\mlabel{subsec:3.1}

The following theorem is a key result of this paper on which
all other discussion builds. An important consequence is
relation \eqref{eq:J-rel} which provides an extension of $U$
to an (anti-)unitary representation of $G_{\tau_h}$ on the same space space. {Note that, besides connectedness,
  no assumptions are made on the structure of $G$,
  in particular $G$ does not have to be semisimple.}

\begin{thm} \mlabel{thm:2.1} {\rm(Euler Element Theorem)} 
  Let $G$   be a connected finite-dimensional Lie group with 
  Lie algebra $\g$ and $h \in \g$. 
  Let $(U,\cH)$ be a unitary 
  representation of $G$ with discrete kernel. 
  Suppose that $\sV$ is a standard subspace
    and $N \subeq G$ an identity neighborhood such that 
  \begin{itemize}
  \item[\rm(a)] $U(\exp(t h)) = \Delta_\sV^{-it/2\pi}$ for $t \in \R$,
    i.e., $\Delta_\sV = e^{2\pi i \, \partial U(h)}$, and 
  \item[\rm(b)] $\sV_N := \bigcap_{g \in N} U(g)\sV$ is cyclic.
  \end{itemize}
  Then $h$ is an Euler element and the conjugation $J_\sV$ satisfies
  \begin{equation}
    \label{eq:J-rel}
 J_\sV U(\exp x) J_\sV = U(\exp \tau_h(x)) \quad \mbox{ for } \quad
 \tau_h = e^{\pi i \ad h}, x \in \g.
  \end{equation}
\end{thm}

{In Theorem~\ref{thm:hsym} we characterize those Euler elements
  for which a standard subspace satisfying (a) exists in every
unitary representation of $G$.
}

\begin{prf} {\bf Part 1: $\ad h$ is diagonalizable with integral eigenvalues:}
  For $x \in \g$, we write
  \[ x(s) := e^{s \ad h} x \in \g. \]
 Pick $\xi \in \sV_N$. Then we have for $\psi \in \cH$
  \begin{align}
    \label{eq:e2x}
 \la \psi, U(\exp(sh)\exp(tx)) \xi \ra
& =  \la \psi, U(\exp(t x(s)) \exp(sh)) \xi \ra \notag \\
& =  \la U(\exp(-t x(s))) \psi, U(\exp(sh)) \xi \ra.
  \end{align}
  By assumption, there exists a $\delta > 0$ such that 
  $U(\exp tx)\xi \in \sV$ for $|t| < \delta$,  so that
  $U(\exp tx)\xi$ is contained in the domain
  of $\Delta_\sV^{1/2} = e^{\pi i \cdot \partial U(h)}$.
  Therefore the left hand side of \eqref{eq:e2x} can be continued analytically
  in $s$ to a continuous function on the closure of the strip 
  $\cS_\pi$   which is holomorphic in the interior
  (Proposition~\ref{prop:standchar}).

  To obtain an analytic extension of the right hand side, we assume that
  $\psi \in \cH^\omega$ is an analytic vector for~$U$. Then there exists an
open convex $0$-neighborhood $B \subeq \g_\C = \g + i \g$ (depending on $\xi$)
  and a holomorphic map
  \[  \eta_\psi \: B\to \cH \quad \mbox{ with } \quad
    \eta_\psi(x) = U(\exp x)\psi \quad \mbox{ for }\quad  x\in B \cap \g\]
  and
  \begin{equation}
    \label{eq:etapsi}
 \eta_\psi(z) = \sum_{n = 0}^\infty\frac{1}{n!}    (\dd U(z))^n \psi
 \quad \mbox{ for } \quad z \in B.
  \end{equation}
Writing $\cH(B)$ for the set of all these vectors~$\psi$,
we know that $\bigcup_{n \in \N} \cH(\frac{1}{n}B)$ is dense in $\cH$
(\cite{Nel59}). 
  Shrinking~$\delta$, we may assume that
  \[ e^{z \ad h} t x \subeq B \quad \mbox{ for } \quad |t| \leq \delta,
      |z| \leq 2 \pi.\]
      Then, for a fixed $t$ with $|t| \leq \delta$, the function 
  $s \mapsto U(\exp(-t x(s))) \psi$ can be continued analytically
  to the open disc $\cD := \{z \in \C \: |z| < 2\pi\}$. 
Further, $s \mapsto  U(\exp sh)\xi$ has an analytic
  continuation to the strip $\cS_\pi$. 
  We conclude that both sides of \eqref{eq:e2x} extend analytically to
  $\cD \cap \cS_\pi$ with continuous boundary values.
  We thus obtain for any fixed $t$ with $|t| \leq \delta$ and $s = \pi i$
  the equality
  \begin{equation}
    \label{eq:e3b}
    \la \psi, e^{\pi i \cdot \partial U(h)} U(\exp tx) \xi \ra
    = \la \eta_\psi(-t e^{-\pi i \ad h}x), e^{\pi i \cdot \partial U(h)}\xi \ra.
  \end{equation}
  As $U(\exp tx)\xi \in \sV$ and $\Delta_\sV^{1/2} = e^{\pi i \cdot \partial U(h)}$,
  this is equivalent to
  \begin{equation}
    \label{eq:e4x}
    \la \psi, J_\sV U(\exp tx) \xi \ra
    = \la \eta_\psi(-t e^{-\pi i \ad h}x), J_\sV \xi \ra.
  \end{equation}
  The real subspace $\sV_N$ spans a dense subspace of $\cH$, so that,
  for each analytic vector $\psi \in \cH^\omega$, there exists a
  $\delta_\psi > 0$, such that 
\begin{equation}
    \label{eq:e5x}
     U(\exp -tx) J_\sV \psi  
     =  J_\sV \eta_\psi(-t e^{-\pi i \ad h}x)
    \quad \mbox{ for } \quad |t| \leq \delta_\psi. 
  \end{equation}
  Multiplication with $J_\sV$ on the left yields
\begin{equation}
  \label{eq:e5y}
 J_\sV U(\exp -tx) J_\sV \psi  
 =  \eta_\psi(-t e^{-\pi i \ad h}x) 
  \end{equation}
  For a fixed $t_0 = \delta_\psi$, \eqref{eq:e5x} shows 
    in particular that the $G$-orbit map of 
    $J_\sV \psi$ is real analytic in an $e$-neighborhood
    because 
    \[ z \mapsto \eta_\psi(-t e^{-\pi i \ad h}z) \]
    defines a holomorphic function on a $0$-neighborhood of $\g_\C$. 
  We therefore
  have $J_\sV\cH^\omega \subeq \cH^\omega$.
As  both sides are differentiable in $t = 0$, we now obtain 
\begin{equation}
    \label{eq:e6x}
    J_\sV \dd U(x) J_\sV  \psi    =  \dd U(e^{-\pi i \ad h}x) \psi
    \quad \mbox{ for } \quad \psi \in \cH^\omega.
  \end{equation}
  The left hand side is a skew-symmetric operator on $\cH^\omega$, so that
  $\dd U(e^{-\pi i \ad h}x)$ is skew-symmetric on $\cH^\omega$.
 As $\ker(\dd U) = \L(\ker U) = \{0\}$, it follows that 
\begin{equation}
    \label{eq:e8x}
    \tau_h(x)  :=  e^{-\pi i \ad h}x \in \g \quad \mbox { for } \quad x \in \g 
  \end{equation}
 because $\dd U(z)$ is skew hermitian on $\cH^\omega$
  if and only if $z \in \g$.
  
This means that the automorphism $\tau_h \in \Aut(\g_\C)$
preserves the real subspace $\g \subeq \g_\C$ and that we have
\begin{equation}
    \label{eq:e6xb2}
    J_\sV \dd U(x) J_\sV   =  \dd U(e^{-\pi i \ad h}x) 
    \quad \mbox{ on } \quad \cH^\omega \quad \mbox{ for every } \quad x \in \g. 
  \end{equation}
Applying this relation twice, we arrive at
\begin{equation}
    \label{eq:e7y}
\dd U(x) =   J_\sV^2 \dd U(x) J_\sV^2   =  \dd U(\tau_h^2x) 
    \quad \mbox{ on } \quad \cH^\omega\quad \mbox{ for every } \quad x \in \g. 
  \end{equation}
As $\dd U$ is injective, this shows that
$e^{-2\pi i \ad h} = \tau_h^2 = \id_\g$.
This in turn implies that
  $\ad h$ is diagonalizable with integral eigenvalues
(\cite[Exer.~3.2.12]{HN12}). We also note that \eqref{eq:e6xb2} entails
\[ J_\sV U(\exp x) J_\sV = U(\exp \tau_h(x)) \quad \mbox{ for } \quad x \in \g\]
because any dense subspace consisting of analytic vectors is a core
  by Nelson's Theorem.

  \nin  {\bf Part 2: $h$ is an Euler element:}
  Let $k \in \Z$ be an eigenvalue of $\ad h$. We have to show that $|k| \leq 1$.
  So let us assume that  $|k| \geq 2$ and show that this leads to a contradiction.
  Let $x \in \g$ be a corresponding eigenvector, so that
  $[h,x] = k x$. In view of (b), there exists a $\delta > 0$ such that
  \[ U(\exp tx) U(\exp sh)  \sV_N \subeq \sV \quad \mbox{ for }  \quad |t| + |s| < \delta.\]
  Let  
  \[ M := \partial U(h) \quad \mbox{ and } \quad
 Q := \partial U(x) \]
  denote the infinitesimal generators of the
  $1$-parameter groups $U(\exp th)$ and $U(\exp tx)$, respectively.
  Suppose that
  $\xi = U(\exp rh)\eta = e^{rM}\eta$ for $\eta \in \sV_N$ and $|r| < \delta$,
  so that $\xi \in \sV$. As in {Part~1}, for
  $|t| + |r| < \delta$ and any entire vector $\psi \in \cH$ of $Q$,
  both sides of
  \begin{equation}
    \label{eq:e2c}
 \la \psi, U(\exp(sh)\exp(tx)) \xi \ra
 =  \la \psi, U(\exp(t e^{sk}x) \exp(sh)) \xi \ra
  \end{equation}
  extend analytically in $s$ into $\cS_\pi$.
  For $s := \frac{\pi i}{|k|}$ we have $\Im s < \pi$, so that
  we obtain for any $\eta \in \sV_N$ 
  \begin{equation}
    \label{eq:e4c}
    \la \psi, e^{\frac{\pi i}{|k|}M} e^{tQ} e^{rM} \eta \ra
    = \la \psi, e^{-tQ} e^{\frac{\pi i}{|k|}M} e^{rM} \eta \ra 
    \quad \mbox{ for } \quad |t| + |r| < \delta.
  \end{equation}
As this holds for a
  dense set of vectors $\psi$, we derive that
  \begin{equation}
    \label{eq:e5c}
    e^{\frac{\pi i}{|k|}M} e^{tQ} e^{rM} \eta  
    =  e^{-tQ} e^{\frac{\pi i}{|k|}M} e^{rM} \eta 
    \quad \mbox{ for } \quad |t| + |r| < \delta.
  \end{equation}

  Now let $E \subeq \R$ be a bounded Borel subset and
  $P_{iM}(E)$ the corresponding spectral projection of the selfadjoint
  operator $iM$ on $\cH$. We multiply the relation \eqref{eq:e5c}
  on the left with $P_{iM}(E)$ to obtain 
  \begin{equation}
    \label{eq:e6}
    e^{\frac{\pi i}{|k|}M} P_{iM}(E) e^{tQ} e^{rM} \eta  
    =  P_{iM}(E) e^{-tQ} e^{\frac{\pi i}{|k|}M} e^{rM} \eta. 
  \end{equation}
  Next we observe that $e^{\frac{\pi i}{|k|} M} P_{iM}(E)$
  is a bounded operator and, as $\pi\geq \frac{2\pi}{|k|}$,
  the vector $\eta$ is contained in the domain of
  $e^{\frac{2\pi i}{|k|} M}$, so that its orbit map 
    $t \mapsto e^{tM}\eta$ extends analytically to the strip
    $\cS_{\frac{2\pi}{k}}$.
  So both sides of \eqref{eq:e6} have 
  analytic continuations in $r$ to the strip $\cS_{\frac{\pi}{|k|}}$.
  Hence by uniqueness of analytic continuation, \eqref{eq:e6}
  also holds for all real $r$ and $|t| < \delta$.
  Let
  \[ \cH_\eta  := \oline{\Spann\{ e^{rM}\eta \: r \in \R} \} \]
  denote the cyclic subspace generated by $\eta$ under
  $e^{\R M} = U(\exp \R h)$. We then obtain from \eqref{eq:e6} that
    \begin{equation}
    \label{eq:e6xb}
    e^{\frac{\pi i}{|k|}M} P_{iM}(E) e^{tQ} \zeta 
    =  P_{iM}(E) e^{-tQ} e^{\frac{\pi i}{|k|}M} \zeta \quad \mbox{ for } \quad
    \zeta \in \cH_\eta. 
  \end{equation}
    As $\cH_\eta$ is invariant under
  the von Neumann algebra generated by $e^{\R M}$, it is invariant under all spectral projections, i.e. $P_{iM}(E)\cH_\eta\subset\cH_\eta$. 
  This shows that
\begin{equation}
    \label{eq:e6b}
    e^{\frac{\pi i}{|k|}M} P_{iM}(E) e^{tQ} P_{iM}(E)  \eta  
    =  P_{iM}(E) e^{-tQ} e^{\frac{\pi i}{|k|}M} P_{iM}(E) \eta. 
  \end{equation}
  As all operators in this identity are bounded and
  $\sV_N$ spans a dense subspace of $\cH$, we arrive at the relation
\begin{equation}
    \label{eq:e7yb}
    e^{\frac{\pi i}{|k|}M} P_{iM}(E) e^{tQ} P_{iM}(E)  
    =  P_{iM}(E) e^{-tQ} P_{iM}(E) e^{\frac{\pi i}{|k|}M} 
\quad \mbox{ for } \quad |t| < \delta.  \end{equation}
Hence
\[ P_{iM}(E)e^{\frac{2\pi i}{|k|}M} P_{iM}(E)
= \big(P_{iM}(E)e^{\frac{\pi i}{|k|}M} P_{iM}(E)\big)^2 \] 
commutes with $ P_{iM}(E) e^{tQ} P_{iM}(E)$ for $|t| < \delta$.
As the von Neumann algebra on $P_{iM}(E)\cH$
generated by
$P_{iM}(E)e^{\frac{2\pi i}{|k|}M} P_{iM}(E)$
contains the unitary one-parameter group $P_{iM}(E)e^{\R M}P_{iM}(E)$,
it follows that
\begin{align*}
  P_{iM}(E)e^{sM} e^{tQ} P_{iM}(E)
&=  P_{iM}(E)e^{sM} P_{iM}(E) e^{tQ} P_{iM}(E)\\
&=P_{iM}(E) e^{tQ}  P_{iM}(E)e^{sM} P_{iM}(E)\\
&=  P_{iM}(E) e^{tQ} e^{sM} P_{iM}(E)
  \quad \mbox{ for } \quad s \in \R, |t| < \delta.
\end{align*}
As $E$ was arbitrary, this implies that $e^{\R M}$ commutes with $e^{\R Q}$,
 contradicting   the assumption $|k| \geq 2$. 
 We therefore have $|k|\leq 1$ and thus $h$ is an Euler element.
\end{prf}

\begin{remark} If $N$ is an
      $e$-neighborhood in $G$, then so is 
      $N^{-1}$. Therefore condition (b) in Theorem \ref{thm:2.1} is equivalent to the following:\\
    \nin (b') There exists a cyclic subspace
    $\sK\subset \sH$ such that $U(g)\sK\subset \sV$ for every $g\in N$.

    Indeed, if (b) holds, then $\sK := \sV_N$ satisfies (b') for the
    $e$-neighborhood $N^{-1}$. If, conversely, (b') holds,
    then $\sV_{N^{-1}} \supeq \sK$ is cyclic.
    When nets of standard subspaces are considered in the next sections, then Property (b) and (b')  will be related to
    regularity and localizability in a specific region, respectively (cf.~Definition \ref{def:reg} and Lemma \ref{lem:loc-imp-reg})
\end{remark}

Starting points for the development of
  the proof of Theorem \ref{thm:2.1}  were \cite{BB99}
  for  Part 1  and \cite{Str08}
  for Part 2. Accordingly, we recover
  one of R.~Strich's results as the following corollary.

\begin{cor} \mlabel{cor:1.1} {\rm(Strich's Theorem for standard subspaces)}
Let $\lambda \in \R^\times$ and consider a
  {two-dimensional} connected Lie group $G$ whose Lie algebra
  is $\g = \R x + \R h$ with $[h,x] = \lambda x$.
  Let $(U,\cH)$ be a unitary
  representation of $G$ with $\partial U(x) \not=0$. Suppose that $\sH \subeq \sV$ are standard subspaces such that
  \begin{itemize}
  \item[\rm(a)] $U(\exp(-\beta t h)) = \Delta_\sV^{it}$ for $t \in \R$.
  \item[\rm(b)] $U(\exp t x)U(\exp sh) \sH \subeq \sV$
    for $|s| + |t| < \delta$ and some $\delta > 0$. 
  \end{itemize}
  Then $\beta = \frac{2\pi}{|\lambda|}$. 
\end{cor}

\begin{prf} Theorem~\ref{thm:2.1} implies that
  $\frac{\beta}{2\pi} h$ is an Euler element in $\g$, so that
  $\frac{\beta |\lambda|}{2\pi} = 1$.   
\end{prf}

  \begin{thm} \mlabel{thm:loc-net}
    Let $(U,\cH)$ be a unitary representation of the
    connected Lie group $G$ with $\ker(U)$ discrete.
    If $(\sH(\cO))_{\cO\subeq M}$ is a net of real subspaces on
   {(the open subsets of)} a $G$-manifold $M$ that satisfies
        {\rm(Iso), (Cov), (RS)} and {\rm(BW)}, 
    then the Lie algebra element $h$ satisfying
\[ \Delta_{\sH(W)} = e^{2\pi i\, \partial U(h)} \]  is an Euler element, and
    the conjugation $J := J_{\sH(W)}$ satisfies 
  \[ J U(\exp x) J = U(\exp \tau_h(x))\quad \mbox{ for } \quad
  \tau_h = e^{\pi i \ad h}, x \in \g.\]
\end{thm}

\begin{prf} Let $\cO \subeq W$ be a non-empty open, relatively compact subset.
  Then $\oline\cO$ is a compact subset of the open set $W$, so that
  \[ N := \{ g \in G \: g.\oline\cO \subeq  W \} \]
  is an open $e$-neighborhood in $G$. For every $g \in N$ we
  have by (Cov) and (Iso) 
  \[  g.\sH(\cO) = \sH(g.\cO) \subeq \sH(W)\ {\buildrel {\rm(BW)}\over =}\ \sV.\]
  Further (RS) implies that $\sH := \sH(\cO)$ is cyclic, hence standard because
  it is contained in $\sV$. Now the assertion follows from
  Theorem~\ref{thm:2.1}.    
\end{prf}

Theorem 6.2 in \cite{BB99}  can be rephrased for standard subspaces. Then it becomes a consequence of {our}
Theorem \ref{thm:loc-net}.  With the notations introduced in Example \ref{ex:desit}, we state the following corollary:

\begin{corollary} {\rm(Borchers-Buchholz Theorem for standard subspaces)} Let $U$ be a unitary representation of the Lorentz group $G = \SO_{1,d}(\R)^\up$
  on a Hilbert space $\cH$, acting covariantly on an isotone net
{ $(\sH(\cO))_{\cO \subeq \dS^d}$} 
of standard subspace on open regions of de Sitter spacetime.
{If $\beta>0$ is such that 
\begin{equation}\label{eq:bb99} U(\exp(th))
  =\Delta^{-\frac{it}{\beta}}_{\sH(W_R^{\dS})}\quad \mbox{ for } \quad t\in\RR,
\end{equation}
then $\beta=2\pi$.}
\end{corollary}

\begin{proof}
The net of standard subspaces $(\sH(O))_{\cO\subset\dS^d}$ with the Lorentz group representation $(U,\cH)$ fit the hypotheses of Theorem \ref{thm:loc-net} with respect to the Lie algebra element $\tilde h=\frac{\beta }{2\pi}h$, as
\[ \Delta_{H(W^{\dS}_R)}={e^{2\pi i\partial U(\tilde h)}}. \] 
We conclude that $\tilde h$ is an Euler element. Since $h$
is also an Euler element in $\so(1,d)$ and $\beta > 0$, we must have $\beta=2\pi$.
\end{proof}

\begin{rem}   (a) An important consequence of Theorem~\ref{thm:2.1} 
    is that $\tau_h$ integrates to an
    involutive automorphism $\tau_h^G$ on the group $U(G) \cong G/\ker(U)$
    that is uniquely determined by
    \[ \tau_h^G(\exp x) = \exp(\tau_h(x)) \quad \mbox{ for }\quad x \in\g.\]
    {To see this, let $q_G \: \tilde G \to G$ denote the universal
      covering of $G$ and $\tau_h^{\tilde G}$ the automorphism of
      $\tilde G$ integrating $\tau_h \in \Aut(\g)$.
      Replacing $G$ by $\tilde G$ and $U$ by $U \circ q_G$, we may
      assume that $G = \tilde G$. Then \eqref{eq:J-rel} implies that
      \begin{equation}
        \label{eq:JUg}
        J U(g) J = U(\tau_h^G(g))\quad \mbox{ for } \quad g \in G.
      \end{equation}
      It follows that $\tau_h^G(\ker U) = \ker U$, and hence that
      $\tau_h^G$ factors through an automorphism of the quotient
      group $G/\ker U \cong U(G)$.

      Whenever $\tau_h^G$ exists (which by the preceding is the case if $G$ is
      simply connected or if $U$ is injective),}
 $U$ extends to an (anti-)unitary representation of the Lie group
 \begin{equation}
   \label{eq:gtauh1}
 G_{\tau_h} = G \rtimes \{\id_G, \tau_h^G\}
 \quad \mbox{  by  } \quad U(\tau_h^G) := J.
 \end{equation}

 In the setting of Theorem~\ref{thm:2.1}, $(U, \cH)$ cannot be a 
multiple of an irreducible representation of complex type.   Indeed, 
in this case there exists no anti-unitary operator $J$
on $\cH$ such that
\begin{equation}
  \label{eq:jrelx}
  U(\tau_h(g))=JU(g)J^{-1} \quad \mbox{ for } \quad g \in G.
\end{equation}
{So the conclusion of Theorem \ref{thm:2.1} fails, and therefore 
  one of the two assumptions (a) and  (b) must be violated.} Given $h\in\fg$, it is easy to construct a standard subspaces satisfying (a) by taking
{$\Delta_\sV:=e^{2\pi i \partial U(h)}$} as Tomita operator 
and any {conjugation $J$ commuting with $\partial U(h)$. The existence
  of such a conjugation only requires the unitary equivalence of the
  selfadjoint operators  $i\partial U(h)$ and $-i\partial U(h)$
  (\cite[Prop.~3.1]{NO15}). This
  is much weaker than \eqref{eq:jrelx} and satisfied in all
  unitary representations if $\g$ is semisimple and 
  $h$ an Euler element (Theorem~\ref{thm:hsym}).
So Hypothesis (b) has to fail and thus regularity is lost. 
However, the doubling process from Lemma \ref{lem:3.4}(a)
leads to a context where \eqref{eq:jrelx} can be implemented.

This accords with the comment after Theorem~4.13 in \cite{DM20},
  where is has been argued, with a similar argument,
  that factorial representations with finite non-zero  helicity 
  of the Poincar\'e group $\cP_+^\up$ of $\RR^{1,3}$ cannot
  act on a net of standard subspaces on spacelike cones (cf.~notation in Def.~\ref{def:minkcaus}).
  We briefly recall the ideas here. Let $(U,\cH)$ be a
  factorial   representation of finite non-zero helicity,  
  acting covariantly on a net of standard subspaces on spacelike cones $\cC\mapsto\sH(\cC)$. 
  By \cite[Cor.~4.4]{DM20}, $\sH$ has the (BW) property with respect to 
  the pair $(h,W_R)$ (see Example~\ref{ex:desit}).\footnote{It actually suffices to require {the net to assign standard subspaces} to wedge regions.}
Following \cite[Prop.~2.4]{GL95} (or in our general setting \cite[Thm.~4.28]{MN21}), a representation of finite non-zero  helicity
acting on a net of standard subspaces on spacelike cones extends to a covariant (anti-)unitary representation of the proper
Poincar\'e group $\cP_+$ as in \eqref{eq:jrelx}. 
As representations of finite non-zero  helicity 
are of complex type (\cite[Thm.~9.10]{Va85}), we arrive at a
contradiction.

Clearly, this example is compatible with the
(BW) property in the form of condition (a) in Theorem \ref{thm:2.1}.
By continuity of the Poincar\'e action on $\RR^{1,3}$, there always exists a spacelike cone $\cC \subeq \bigcap_{g\in N} gW$ if
$N\subset\cP_+^\uparrow$ is a sufficiently small
neighborhood of the identity and $W$ is a wedge region.
For $\sV = \sH(W)$, we then obtain 
$\sH(\cC)\subset\sV_N=\bigcap_{g\in N} g\sH(W)$,
and thus $\sV_N$ is cyclic whenever $\sH(\cC)$ is
(which follows from (RS)).}
In particular, spacelike cone localization of standard subspaces ensures the regularity condition (b) in the setting of Theorem \ref{thm:2.1} and this
regularity condition for $\sH(\cC)$ ensures the geometric property
used in \cite[Prop.~2.4]{GL95} to obtain an
extension to an (anti-)unitary representation
of~$\cP^\up$. As stressed for this specific case in \cite{DM20}, one needs to
couple finite non-zero helicity representations with opposite
helicities to {provide an environment for non-trivial nets}
of standard subspaces.

\smallskip
\nin   (b)  If $\sV_N= \sV$, then $\sV$ is $U(G)$-invariant
  because the connected Lie group $G$ is generated by the identity
  neighborhood~$N$. In this case $h \in \g$ is central, which follows from
  the discreteness of $\ker(U)$ because $U(G)$ commutes with $\Delta_\sV$.
  Then we obtain on $\cH^J$ a real representation of~$G$.
  
\smallskip  
  
  \nin (c) If $\g$ is a compact Lie algebra, then every Euler element
    $h \in \g$ is central, so that ${\tau_h = \id_\g}$.
    Therefore the cyclicity of $\sV_N$ as in Theorem~\ref{thm:2.1}
    implies that $J_\sV$ and $\Delta_\sV$ commute with $U(G)$, and thus
    $U(g)\sV = \sV$ for $g \in G$.
 Therefore, a standard subspace $\sV$ associated
    to a pair $(h,\tau) \in \cG(G_\sigma)$  by the BGL construction  
    can only satisfy the regularity condition in Theorem~\ref{thm:2.1}(b)
    if $\sV$ and $\cH^{J_\sV}$ are $U(G)$-invariant.
  Therefore the representation $(U,\cH)$ is the complexification
  of the real representation of $U$ on $\cH^J = \sV$.
  Conversely, for every real representation $(U,\cE)$ of~$G$,
  the real subspace $\cE \subeq \cE_\C$ is standard with $\Delta_\cE = \1$
  and $U_\C(G)$ leaves $\cE$ invariant, so that the regularity condition is
  satisfied for trivial reasons.
\end{rem}

\subsection{An application to operator algebras}
\mlabel{subsec:strich-opalg}

The following theorem is a version of the Euler Element Theorem~\ref{thm:2.1}
for operator algebras.
We consider the following setup:

\begin{itemize}
\item[\rm(Uni)] Let $(U,\cH)$ be a unitary representations
  of the {\bf connected} Lie group
  $G$ with discrete kernel, so that the derived representation
  $\dd U$ is injective. 
\item[\rm(M)] Let $\Omega$ be a unit vector  and
  $\cM \subeq B(\cH)$ be a von Neumann algebra
  for which $\Omega$ is cyclic and generating.
  We write $(\Delta_{\cM,\Omega}, J_{\cM,\Omega})$ for the corresponding
  modular objects.
\item[\rm(Fix)] $\Omega \in \cH^G$, i.e., $\Omega$ is fixed by $U(G)$.   
\item[\rm(Mod)] \textbf{Modularity:} There exists an  element $h \in \g$ for which
  $e^{2\pi i \partial U(h)} = \Delta_{\cM,\Omega}$. As $\ker(U)$ is discrete, $h$ is
  uniquely determined. 
\item[\rm(Reg)] \textbf{Regularity:} For some $e$-neighborhood $N \subeq G$, the vector
  $\Omega$ is still cyclic (and obviously separating) for the
  von Neumann algebra 
  \[ \cM_N := \bigcap_{g \in N} \cM_g, \quad \mbox{ where } \quad
    \cM_g= U(g)\cM U(g)^{-1}.\]
  This implies that $(\cM_N)'$ is a von Neumann algebra
  containing $\cM_g' = U(g) \cM' U(g)^{-1}$ for $g \in N$ and that
  $\Omega$ is cyclic and separating for $(\cM_N)'$.   
\end{itemize}

\begin{thm} \mlabel{thm:2.1-alg}
  Assume {\rm(Uni)}, {\rm(M), (Fix)}, {\rm(Reg)} and {\rm (Mod)}.
  Then $h$ is an Euler element and the modular conjugation $J = J_{\cM,\Omega}$
of the pair $(\cM,\Omega)$ satisfies
  \[ J U(\exp x) J = U(\exp \tau_h(x)) \quad \mbox{ for } \quad
  \tau_h = e^{\pi i \ad h}.\]
\end{thm}

\begin{prf} Clearly, $\Omega$ is also separating for $\cM_N$.
  Let $\cM_{\sa} := \{ M \in \cM \: M^* = M\}$
 be the subspace of hermitian
  elements in $\cM$. Then we obtain
the two standard subspaces 
\begin{equation}
  \label{eq:vh}
\sV := \oline{\cM_{\sa}\Omega}  \supeq
\sH := \oline{(\cM_N)_{\sa}\Omega}.
\end{equation}
Further $U(g)^{-1}\cM_N U(g) \subeq \cM$ for $g \in N$ implies
$U(g)^{-1}\sH \subeq \sV$. Hence $\sH \subeq \sV_N$, and
the assertion follows from Theorem~\ref{thm:2.1}
\end{prf}

\begin{ex} \mlabel{ex:3.7}
  (The minimal group) For $G = \R$, $\g = \R h$, and
    the unitary one-parameter group
    $U(t) := \Delta_{\cM,\Omega}^{-it/2\pi}$,
    the conditions (Uni), (M), (Fix), (Mod) and (Reg) are satisfied
    because the Tomita--Takesaki Theorem ensures that
    $\cM_g = \cM$ for every $g \in G$. The conclusion
    of Theorem~\ref{thm:2.1-alg} then reduces to the fact that
    $J_{\cM.\Omega}$ commutes with the modular group.   
\end{ex}

\subsubsection*{Endomorphism semigroups}

We consider the context from Theorem~\ref{thm:2.1-alg},
    where $G$ is a connected finite-dimensional Lie group with
    Lie algebra $\g$, $h \in \g$ is an Euler element,
    $(U,\cH)$ is an (anti-)unitary 
    representation of $G_{\tau_h}$ with discrete kernel,
    $J = U(\tau_h^G)$, 
  and $\sV = \sV(h,U) \subeq \cH$ is the associated standard subspace.
  We also have a von Neumann algebra $\cM$ with cyclic separating vector
  $\Omega$ for which
  \[ \sV = \sV_\cM := \oline{\cM_{\sa}\Omega}. \]
  Here the equality of $\sV$ and $\sV_\cM$ follows from the equality
  of their modular objects and Proposition~\ref{prop:11}.

We consider {\it the endomorphism semigroup of $\cM$ in $G$} by 
  \[ S_\cM := \{ g \in G \: U(g) \cM U(g)^{-1} \subeq \cM \}. \]
Typically it is hard to get fine information on the semigroup
  $S_\cM$, but combining results from \cite{Ne22} with 
  Theorem~\ref{thm:2.1-alg}, we actually get a full description
  of its identity component by comparing it with the
  endomorphism semigroup 
 \[ S_{\sV} := \{ g \in G \: U(g) \sV\subeq U(g)\}.\]

\begin{thm} \mlabel{thm:3.2} {\rm(The endomorphism semigroup)}
  Suppose that {\rm(Uni)}, {\rm(M), (Fix)}, {\rm(Reg)} and {\rm (Mod)}
  are satisfied. With the pointed cones $C_\pm := \pm C_U \cap \g_{\pm 1}(h)$, we have
  the following description of the identity component of the semigroup~$S_\cM$: 
  \[ (S_\cM)_e = (G_\cM)_e \exp(C_+ + C_-) = \exp(C_+) (G_\cM)_e \exp(C_-)
    \quad \mbox{ and } \quad \L(G_\cM) = \g_0(h).\]
In particular $(G_\cM)_e = \la \exp \g_0(h) \ra$.
\end{thm}

\begin{prf}
  As $U$ has discrete kernel, $h$ is an Euler element and $\sV = \sV_\cM$, 
  \cite[Thms.~2.16, 3.4]{Ne22} imply that
\begin{equation}
  \label{eq:SV}
  S_{\sV_\cM} = G_{\sV_\cM} \exp(C_+ + C_-) = \exp(C_+) G_{\sV_\cM} \exp(C_-).
\end{equation}
  Further, $g \in S_\cM$ yields 
  $U(g) \sV_{\cM} = \sV_{\cM_g} \subeq \sV_\cM$
  because $U(g)$ fixes $\Omega$, and therefore 
  \begin{equation}
    \label{eq:semincl}
    S_\cM \subeq S_{\sV_\cM}.
  \end{equation}

  Let $N$ be an $e$-neighborhood as in (Reg)
    and  $g \in S_{\sV_\cM} \cap N$.
    Then $\cM_N'$ contains both algebras $\cM'$
    and $\cM'_g = U(g)\cM' U(g)^{-1}$,
    and $U(g) \sV_\cM \subeq \sV_\cM$ implies 
$U(g) \sV_\cM' \supeq \sV_\cM'$. 
Further, $\Omega$ is cyclic and separating for $\cM_N'$
  and
  \[\sV_{\cM_g'} =  U(g) \sV_{\cM'} = U(g) \sV_{\cM}' \supeq \sV_{\cM}' = \sV_{\cM'}.\]
  As $\Omega$ is cyclic and separating for
 $\cM_g'$ and $\cM'$, 
  \cite[Prop.~3.24]{Lo08} implies that
  $\cM_g' \supeq \cM'$,
  which leads to $\cM_g \subeq \cM$, i.e.,
  $g \in S_\cM$. This proves that
    \[ S_\cM \cap N = S_{\sV_\cM} \cap N.\] 
  Since the semigroups $\exp(C_\pm)$ and $(G_{V_\cM})_e$ are generated by
  their intersections with $N$, it follows that
  ${(S_{\sV_\cM})_e} = \exp(C_+)(G_{V_\cM})_e \exp(C_-)\subeq S_\cM$.
  Now the assertion follows  from the fact that the
    connected components of $S_{\sV_\cM}$ are products
    of connected components of the group $G_{\sV_\cM}$
    and $\exp(C_+ + C_-)$ (polar decomposition of $S_{\sV_\cM}$).
\end{prf}

\begin{rem} Davidson's paper \cite{Da96} contains
    interesting results on the relation
between the stabilizer groups $G_\cM$ and $G_{\sV_\cM}$, 
also on the level of endomorphism semigroups.

\nin (a) \cite[Thm.~4]{Da96}
considers a unitary
one-parameter group $U_t = e^{itH}$ that fixes $\Omega$ and 
leaves the standard subspace $\sV_{\cM}$ invariant. 
It asserts that, if the set
\[ \cD(\delta) := \{ X \in \cM \: [H,X] \in \cM \} \]
is such that $\cD(\delta)\Omega$ is a core for $H$ in $\cH$, then 
$\Ad(U_t)\cM = \cM$ for all $t\in \R$.

\nin (b) \cite[Thm.~5]{Da96} considers a unitary
one-parameter group $U_t = e^{itH}$ fixing $\Omega$ such that 
$U_t \sV_{\cM}\subeq \sV_{\cM}$ for $t \geq 0$.
He shows that, if 
\[ \sV_\eps := \bigcap_{0 \leq t \leq \eps} U_t \sV_{\cM} \]
is cyclic for some $\eps > 0$, then
$\Ad(U_t)\cM \subeq \cM$ for $t \geq 0$.
This condition is rather close to the assumption in our
Theorem~\ref{thm:2.1} and the regularity conditions discussed 
in the following section. 
\end{rem}

\section{Regularity and Localizability}
\mlabel{sec:loc-reg}

{If $(U,\cH)$ is a unitary representation
  of the Lie group $G$ and $\sV\subeq \cH$ a standard subspace with
  $\Delta_\sV = e^{2\pi i \partial U(h)}$ for some $h \in \g$, then the
  Euler Element Theorem (Theorem~\ref{thm:2.1}) 
  describes a sufficient condition for $h$ to be an Euler element,
  and in this case it even implies the extension of $U$ to an
  (anti-)unitary extension of $G_{\tau_h}$ by~$J_\sV$.
  In this section we study the converse problem: Assuming
  that $h$ is an Euler element and $(U,\cH)$ an (anti-)unitary
  representation of $G_{\tau_h}$, when is $\sV_N$ cyclic for some
  $e$-neighborhood $N \subeq G$. We then call $U$ regular with
  respect to~$h$.
  In Subsection~\ref{subsec:4.1} we discuss various permanence properties
  of regularity and also sufficient conditions, such as
  Theorems~\ref{thm:reg-posen} and \ref{thm:posreg},
  deriving regularity from positive spectrum conditions.

  In Subsection~\ref{subsec:4.2}, we turn to localizability
  aspects of nets of real subspaces.
  Starting with an (anti-)unitary representation of
  $G_{\tau_h}$ and the corresponding standard subspace
  $\sV = \sV(h,U)$, we consider an maximal net
  $\sH^{\rm max}$ associated to some wedge region $W \subeq M = G/H$.
  We then say that $(U,\cH)$ is $(h,W)$ localizable in those subsets
  $\cO \subeq M$ for which the real subspace $\sH^{\rm max}$ is cyclic.
  Here the starting point is to assume this for $W$, which 
  by Lemma~\ref{lem:direct-net} implies that $\sH^{\rm max}(W) = \sV$,
  so that the net $\sH^{\rm max}$ satisfies (Iso), (Cov) and (BW),
  but not necessarily the Reeh--Schlieder condition.
  In this context our main results are
  Theorem~\ref{thm:local-reduct}, asserting localizability
  for linear reductive groups in all representations in all
  non-empty open subsets of the associated non-compactly causal
  symmetric space for a suitably chosen wedge region. 
  For the Lorentz group $\SO_{1,d}(\R)_e$ and its simply
  connected covering $\Spin_{1,d}(\R)$, this leads to localization
  in open subsets of de Sitter space $\dS^d$.
  Relating open subsets of $\dS^d$ with open spacelike cones in
  Minkowski space $\R^{1,d}$, this allows us to derive that,
  for the Poincar\'e group, localizability in spacelike cones
  is equivalent to the positive energy condition
  (Theorem~\ref{thm:poinloc}).

\subsection{Regularity}
\mlabel{subsec:4.1}

\begin{defn}\label{def:reg}
  We call an (anti-)unitary representation 
  $(U,\cH)$ {of $G_{\tau_h}$} {\it regular with respect to $h$}, or
{\it $h$-regular}, if there exists an 
$e$-neighborhood $N \subeq G$ such that
$\sV_N = \bigcap_{g \in N} U(g)\sV$ is cyclic.
{Replacing $N$ by its interior, we may always assume that $N$ is
  open.}
\end{defn}

\begin{rem}
In these terms, Theorem~\ref{thm:2.1} asserts that, 
if $U$ is a unitary representation with discrete kernel,
$\sV$ is a standard subspace and $h \in \g$ with
$\Delta_\sV = e^{2\pi i \, \partial U(h)}$, then $h$-regularity 
implies that $h$ is an Euler element and
that the prescription $U(\tau_h) := J$ extends $U$ to an 
(anti-)unitary representation of $G_{\tau_h}$. 
\end{rem}

This leads us to the problem to determine which
 (anti-)unitary  representations $(U,\cH)$
of~{$G_{\tau_h}$} are
$h$-regular. We start with a few general observations

\begin{examples}\label{ex:reg}
(a) If $G$ is abelian, then $\tau_h = \id_\g$
  and $J$ commutes with $U(G)$.
  Therefore $U(g)\sV = \sV$ for all $g \in G$ and
  thus all representations are regular.

  \nin (b) From \cite{FNO23} it follows that all irreducible  (anti-)unitary 
  representations are regular for any Euler element 
  if $G$ is a simple linear Lie group
  or $\g \cong \fsl_2(\R)$. 
  In Corollary~\ref{cor:real-red} below, this is extended
  to all connected linear real reductive Lie groups. 
    
  \nin (c) Let $L =  \SO_{1,d}(\R)_e$ be the connected Lorentz group
  and $h \in \so_{1,d}(\R)$ a boost generator.
Then  all  (anti-)unitary  representations
of the proper Lorentz group $L_+ \cong L_{\tau_h}$ are $h$-regular.
This follows from  $d = 1$ from (a) and, for $d \geq 2$, from (b).
\end{examples}

\begin{lem} \mlabel{lem:3.4b}
  For an  (anti-)unitary  representation $(U,\cH)$ of $G_{\tau_h}$,
  the  following assertions hold:
  \begin{itemize}
  \item[\rm(a)] If $U = U_1 \oplus U_2$ is a direct sum, then
    $U$ is $h$-regular if and only if $U_1$ and $U_2$ are $h$-regular.
  \item[\rm(b)] If $U$ is $h$-regular, then every subrepresentation is
    $h$-regular. 
  \end{itemize}
\end{lem}

\begin{prf} (a) If $U \cong  U_1 \oplus U_2$,
  then \eqref{eq:v-dirsum} implies that
  $\sV_N = \sV_{1,N} \oplus \sV_{2,N}$ for every
  $e$-neighborhood $N \subeq G$.
  In particular, $\sV_N$ is cyclic if and only if $\sV_{1,N}$ and
  $\sV_{2,N}$ are. 

  \nin (b) follows immediately from (a). 
\end{prf}

Applying Lemma~\ref{lem:g-inter}(b) to $A := N$, we obtain 
the following generalization to direct integrals:

\begin{lem} \mlabel{lem:g-intera}
  Assume that $G$ has at most countably many components.
  Then a direct integral $U = \int_X^\oplus U_m \, d\mu(m)$
  is regular if and only if there exists an
    $e$-neighborhood $N \subeq G$ such that, for $\mu$-almost every $m \in X$,
    the subspace $\sV_{m,N}$ is cyclic.
  \end{lem}

To deal with tensor products, we need the following observations 
from \cite{LMR16}: 

\begin{lem} Let $\sV_j \subeq \cH_j$, $j = 1,\ldots, n$, be standard 
  subspaces with the modular data $(\Delta_j, J_j)$.
  Then the closed real span 
  \[\sV :=  \sV_1 \otimes \cdots \otimes \sV_n \]
  of the elements $v_1 \otimes \cdots \otimes v_n$,
  $v_j \in \sV_j$, is a standard subspace of
  \[ \cH := \cH_1 \otimes \cdots \otimes \cH_n\]
  with modular data
  \[ \Delta = \Delta_1 \otimes \cdots \otimes \Delta_n \quad \mbox{ and } \quad
    J = J_1 \otimes \cdots \otimes J_n.\]
  Moreover,
  \[ \sV' = \sV_1' \otimes \cdots \otimes \sV_n'.\] 
\end{lem}

\begin{prf} The first assertion follows easily by induction from the case
  $n = 2$ (\cite[Prop.~2.6]{LMR16}).
The second assertion follows by induction from \cite[Prop.~2.5]{LMR16}. 
\end{prf}

\begin{ex} \mlabel{ex:sl2-inc} 
Consider the group
  $G = \tilde\SL_2(\R)$, an Euler element
  $h \in \g = \fsl_2(\R)$ (they are all conjugate)
  and an irreducible  (anti-)unitary  representation
  $(U_1, \cH_1)$ of $G_{\tau_h}$
  for which $U_1(Z(G)) \not\subeq \{\pm \1\}$.
  We then consider the antiunitary representation 
 \[ U := U_1 \otimes \oline{U_1} \quad \mbox{ of $G_{\tau_h}$ on }  \quad
   \cH_1 \otimes \oline{\cH_1}\]
and observe that $U_1(Z(G)) \subeq \T \1$ implies that
  $U$ factors through the group $G/Z(G) \cong \PSL_2(\R)$.
    For $\sV_1 := \sV(h, U_1)$, $\sV_1' = \sV(h, \oline U_1)$,
    and $\sV := \sV(h,U)$, we then have
    \[ \sV_{Z(G)} = \sV = \sV_1 \otimes \sV_1'
      \subeq \cH = \cH_1 \otimes \oline{\cH_1}.\]
  However, $U_1(Z(G))  \subeq \T\1$ is a subgroup containing
  non-real numbers, so that
  \[ \sV_{1,Z(G)} =\bigcap_{z \in Z(G)} U_1(z) \sV_1 = \{0\}.\]
  We therefore have
  \[ \sV_{Z(G)} = \sV  \not=    \sV_{1,Z(G)} \otimes \sV_{1,Z(G)}' = \{0\}.\]
\end{ex}

\begin{ex} \mlabel{ex:sl2-incb}
 Another example from AQFT, where
  strict inclusions of the type \eqref{eq:vtens} arise, 
  is contained in
  \cite[Sect. 4.2.2]{MT19}.
  We present the example in a slightly different way from \cite{MT19} in order to fit it with the language introduced in this paper.
  It is obtained by
  second quantization of the 
  tensor product of $\U(1)$-current chiral one-particle nets.
  Consider the $1+1$-dimensional Minkowski spacetime $\RR^{1,1}$  
with the quadratic form $x^2=x_0^2-x_1^2$, where {spacetime events
  are denoted $x=(x_0,x_1)$}. One can now pass to chiral coordinates: 
\begin{equation}\label{eq:+-} (x_+, x_-) = \Big(\frac{x_0+x_1}{\sqrt2},\frac{x_0-x_1}{\sqrt2}\Big)\end{equation}
In these coordinates, the right and left wedge in $\R^{1,1}$ are given by
\[ W_R = \R_+ \times \R_-  \quad \mbox{ and } \quad
  W_L = \R_- \times \R_+.\]
Consider the BGL net $(\sH(I))_{I \subeq \R_\infty}$  indexed by intervals on the 
compactified real line $\R_\infty = \R \cup \{\infty\}$,
associated with the (anti-)unitary lowest weight 1 representation
$(U,\cH)$ of the M\"obius group $\Mob_{\tau_h}$ with respect to
the Euler element 
$h\in\fsl_2(\RR)$, the generator of the dilations, acting by $\exp(th)x=e^tx$.
We form the tensor product net
\[ \RR^{1,1}\supset I_1 \times I_2 \mapsto \tilde \sH(I_1 \times I_2) 
  :=\sH(I_1)\otimes \sH(I_2)\subset\cH\otimes\cH, \] 
where $I_1$ and $I_2$ are intervals in  $\RR_\infty$.
A pair of intervals specifies a region
\[ \cD_{I_1,I_2} := \{(x_+,x_-)\in\RR^{1,1}: x_+\in I_1 , x_-\in I_2\}.\] 
Here we only consider intervals $I_1, I_2 \subeq \R$, so that the
product set $I_1  \times I_2 \subeq \R_\infty^2$ can be identified with 
$\cD_{I_1, I_2}$, and this set is connected.

The net $\tilde \sH$ on ``rectangles'' in $\R_\infty^2$ is covariant for the representation $U\otimes U$ of the group
$\Mob^2_{\tau_h} := (\Mob\times\Mob)_{(\tau_h,\tau_h)}$.
Note that
the identity component of the Poincar\'e group $\cP_+^\uparrow$ and the
dilation group $(D(t))_{t \in \R^+}$ are contained in the group 
$\Mob^2$.  Let $r$ be the space reflection  $r(x_0,x_1) =
(x_0, -x_1)$, resp., by $r(x_+,x_-) = (x_-, x_+)$.
We consider the group
$\Mob^2_{r,\tau_h}$, generated by $\Mob^2_{\tau_h}$ and $r$.
We implement the reflection $r$ unitarily
on $\cH\otimes\cH$ as the flip, acting on simple tensors
by $U(r)(\xi\otimes\eta) =\eta\otimes\xi$. 
This extends $U\otimes U$ to an (anti-)unitary representation 
  $U^{(2)}$ of $\Mob^2_{r,\tau_h}$ 
  for which the net $\tilde\sH$
  is covariant. 
Now let
  \[ G \cong \R^{1,1} \rtimes (\R^+ \times \OO_{1,1}(\R))^\up
    \cong \cP^\up \rtimes \R^+ \]
be the subgroup of $\Mob^2_r$ 
generated by $\cP^\up=\R^{1,1}\rtimes \OO_{1,1}(\RR)^\up$
and positive dilations.  
Clearly, 
  \[ \tilde\sH(W_R)= \sH(\RR^+)\otimes \sH(\RR^-) \quad \mbox{ and } \quad
    \tilde\sH(W_L)= \sH(\RR^-)\otimes \sH(\RR^+).\]
Let $I_1=(a,b)$ and $I_2=(c,d)$ be bounded real intervals.
Then
\[ I_1\times I_2= W^R_{a,c} \cap W^L_{b,d},\]
where
\[ W^R_{a,c}=(\RR^++a)\times(\RR^-+c) \quad \mbox{ and } \quad
  W^L_{b,d}=(\RR^-+b)\times(\RR^++d). \]

Let $A= \{g_1, g_2\}\subeq  \Mob\times\Mob$,
where $g_1 W_R=W^R_{a,c}$ and $g_2W_R=W^L_{b,d}$.
{For
\[ \sV := \tilde \sH(W_R),\]
we now derive from isotony }
  \begin{equation}
    \label{eq:prop-incl}
 \sV_A= \tilde\sH(W^R_{a,c})\cap \tilde\sH(W^L_{b,d})
 \supset \tilde\sH(I_1 \times I_2)
 = \sH(I_1) \otimes \sH(I_2)
  = \tilde \sH(W^R_{a,c} \cap W^L_{b,d}). 
  \end{equation}
  We now consider $\tilde{\sH}^\m$, the maximal net with respect to $G$.
  In \cite[Sect.~4.4.2]{MT19} it is proved that
  $\tilde {\sH}^\m(I_1 \times I_2) = \sV_A$  properly
  contains $\tilde \sH(I_1 \times I_2) = \sH(I_1) \otimes \sH(I_2)$. The idea of the proof is that
  the net $\tilde \sH$  is  $\Mob\times\Mob$-covariant by
  construction, but the net on Minkowski space 
  \[\RR^{1+1}\supset I_1\times I_2\longmapsto \tilde \sH^\m (I_1\times I_2)\subset\cH,
    \qquad I_1,I_2\subset\RR\]
  is only $G$-covariant and. Consequently, they have to be different.
  It is easy to see (again by construction) that the net
  $\tilde\sH^\m$ is $G$-covariant with respect to
  $U^{(2)}\res_G$.
  In order to prove that it is not $\Mob\times \Mob$-covariant, one
  can argue as follows:   The representation
  \[ (U\otimes U)|_{\cP^\up}=\int_{\RR_+}^\oplus U_m d\nu(m)\]
  disintegrates to a direct integral of all positive mass
  representations $(U_m, \cH_m), m > 0$, of $\cP^\up$.  
   On wedge regions, the net is the BGL net,
   hence disintegrates into the BGL nets $\sH_m$
   over $\R_+ = (0,\infty)$
\[ \tilde\sH(W)=\int_{\RR_+}^\oplus {\sH_m(W)}\, d\nu(m)\subset\int_{\RR_+}^\oplus\cH_m d\nu(m). \] 
  By (DI2) from Appendix \ref{app:stsub}, we also have 
\[ \tilde\sH^\m(\cD)=\int_{\RR_+}^\oplus {\sH_m^{\rm max}(\cD)}\, d\nu(m)\subset\int_{\RR_+}^\oplus\cH_m d\nu(m)\]
for all open doublecones $\cD=I_1\times I_2$. 
We associate the following subspace to the forward light cone: 
\[ \sK(V_+)
  :=\overline{\sum_{\cD\subset V_+}\tilde\sH^\m(\cD)},\]
{where the union is extended over all double cones $\cD$ contained
  in $V_+$.}

Following {\cite[Prop.~4.3]{MT19}},
we have $\overline{\sum_{\cD\subset V_+}\tilde\sH_m(\cD)}=\cH_m$, so that
$\sK^\m(V_+)$ is  not separating because
\[ \sK(V_+)=\overline{\sum_{\cD\subset V_+}\tilde\sH^\m(\cD)})
  =\oline{\int_{\RR_+}^\oplus\sum_{\cD\subset V_+} \sH_m(\cD)d\nu(m)}
  =\int_{\RR_+}^\oplus\cH_m d\nu(m)=\cH.\]
Let $g\in\Mob\times\Mob$ 
such that $g\cD=V_+$ for some bounded interval~$\cD$.
We conclude that there is no unitary operator 
$Q\in\U(\cH)$, 
implementing $g$ in the sense that 
$Q\tilde\sH^\m(\cD) \supeq {\tilde\sH^\m(\tilde\cD)}$
holds for all double cones {$\tilde\cD \subeq V_+$.}
In fact, the former is a standard subspace and sum of the spaces
on the right is not separating.
\end{ex}

\begin{lem} \mlabel{lem:dense-int} Let $(U,\cH)$ be an  (anti-)unitary  representation
  of $G_{\tau_h}$ for which the cones 
  \[ C_\pm := \pm C_U \cap \g_{\pm 1}(h) \]
  have interior points in $\g_{\pm 1}(h)$
  with respect to the subspace topology.
  Then, for $\sV = \sV(h,U)$,
  the semigroup $S_\sV = \{ g \in G \: U(g)\sV \subeq \sV\}$
  has dense interior,  i.e., $S_\sV = \oline{S_{\sV}^\circ}$.
  \end{lem}

  Note that, if $C_U$ has interior points, then so do
    the cones $C_\pm$, because they are the projections of $\pm C_U$
    onto $\g_{\pm 1}(h)$.
  
\begin{prf} Let $G^r := G/\ker(U)$ and $\fn := \L(\ker U) = \ker(\dd U)$.
  We write $U^r \: G^r \to \U(\cH)$ for the unitary representation
  of $G^r$ defined by~$U$.   Then
  \[ C_U = C_U + \fn \quad \mbox{ and }  \quad
    C_U /\fn = C_{U^r}.\]
  Moreover, for $\fn_\lambda(h) = \fn \cap \g_{\pm \lambda}(h)$ we have 
  \[ \g^r_\lambda(h) \cong \g_\lambda(h)/\fn_\lambda(h)
    \quad \mbox{ for }\quad \lambda =1,0,-1.\]
  Therefore the cones 
  \[ C_\pm^r := \pm C_{U^r} \cap \g^r_{\pm 1}(h)
    = C_\pm/ \fn_{\pm 1}(h) \]
  are generating and
  \[ S^r_\sV := \{ g \in G^r \: U^r(g)\sV \subeq \sV\}
    = G^r_{\sV} \exp(C_+^r + C_-^r) \]
  by \cite[Thm.~3.4]{Ne22}.
  To see that this semigroup has dense interior, it suffices to show
    that $e$ can be approximated by interior points.
    Since both cones $C_\pm^r$ have dense interior
    and the map
    \[ \g_0(h) \times \g_1(h) \times \g_{-1}(h) \to G, \quad
      (x_0, x_1, x_{-1}) \mapsto \exp(x_0) \exp(x_1 + x_{-1}) \]
is a local diffeomorphism around $(0,0,0)$, 
the semigroup $S^r_\sV$ has dense interior.
As $S_\sV \subeq G$ is the full
  inverse image of $S^r_{\sV}$ under the quotient map $G \to G^r$, which has continuous local  sections, 
 it has dense interior as well. 
\end{prf}

\begin{theorem} \mlabel{thm:reg-posen} {\rm(Regularity via positive energy)}
  If $(U,\cH)$ is an (anti-)unitary representation of $G_{\tau_h}$
  for which the cones
  \[ C_\pm := \pm C_U \cap \g_{\pm 1}(h) \]
  are generating in $\g_{\pm 1}(h)$, then $(U,\cH)$ is regular.   
\end{theorem}

\begin{prf} For a subset
  $N \subeq G$ and $g_0 \in G$, we note that
  \begin{equation}
    \label{eq:Nsv}
  U(g_0)\sV \subeq \sV_N \qquad \Leftrightarrow \qquad
  N^{-1} g_0 \subeq S_\sV.
  \end{equation}
  From Lemma~\ref{lem:dense-int} we infer that   $S_\sV$ has an interior point~$g_0$,
  so that the above condition is satisfied for some $e$-neighborhood~$N$.
  As $U(g)\sV$ is cyclic, it follows in particular that $\sV_N$ is cyclic.
\end{prf}

\begin{rem} \mlabel{rem:quant}
  (a) The condition on the cone $C_\pm$ to be generating holds for
{positive energy representations of}
  the M\"obius group. Up to sign, the
  only pointed, generating (in the sense of having interior points) 
  closed convex $\Ad$-invariant cone is
 \[  C := \{ X \in \g \: {V_X} \geq 0\} 
= \Big\{  X= \pmat{a & b \\ c & -a} \: 
b \geq 0, c \leq 0, a^2 \leq -bc\Big\}.\]
For the Euler element $h = \frac{1}{2}\pmat{1 & 0 \\ 0 & -1}$ we have
\[ C_\pm = \pm C \cap \g_{\pm 1}(h), \quad
  C_+ = \R_+ \pmat{0& 1\\ 0 & 0}, \quad 
  C_- = \R_+ \pmat{0& 0\\1 & 0},\]
and the half lines $C_\pm$ in $\g_{\pm 1}(h)$ also have interior points.  
In general the generating property of the cones {$C_\pm$ in $\g_{\pm 1}(h)$}  is rather strong.
For instance it is not satisfied by positive energy representations of the  Poincar\'e group 
on $\R^{1,3}$. 
Theorem \ref{thm:posreg} will show how to derive regularity 
if the cones $C_\pm$ are not generating; see~Remark \ref{rem:poi}.

  \nin (b) From the proof of Theorem~\ref{thm:reg-posen}
  one can derive some more specific
  quantitative information. If
  $N$ is an $e$-neighborhood contained
  in $g_0^{-1} S_\sV$ for some $g_0 \in S_\sV$, then the argument implies that
  $\sV_N$ is cyclic.

  \nin (c) If $[\g_1(h), \g_{-1}(h)] = \{0\}$, then
  $B := \exp(\g_1(h) + \g_{-1}(h))$ is an abelian subgroup of $G$
  and $S_\sV \supeq G^h_e \exp(C_+ + C_-)$.
  If $C \subeq B$ is any compact $e$-neighborhood,
  then there exists a $b_0 \in S_\sV$ with $C^{-1}b_0 \subeq S_\sV$,
  so that $ G^h_e C^{-1} b_0 \subeq S_\sV$ and
  thus $\sV_{C G^h_e} = \sV_C \supeq U(b_0)\sV$
  is cyclic. It follows that $N$ can be chosen
  arbitrarily large, whenever the cones $C_\pm$ are generating.
  {A typical example is given by the $3$-dimensional Poincar\'e algebra
    in dimension~$1+1$.} 
\end{rem}

Note that the subgroups $G_{\pm 1}(h) := \exp(\g_{\pm 1}(h)) \subeq G$
are abelian. 

\begin{theorem}\label{thm:posreg} Suppose that $G = R \rtimes L$ is a semidirect product.
  Let $(U,\cH)$ be an anti\-unitary representation such that
  \begin{itemize}
  \item $(U\res_L,\cH)$ is regular, and 
  \item the cones $C_\pm := \pm C_U \cap \fr_{\pm 1}(h)$ generate
    $\fr_{\pm 1}(h)$. 
  \end{itemize}
  Then $(U,\cH)$ is regular.
\end{theorem}

\begin{prf} First, let $N_L \subeq L$ be an $e$-neighborhood for which
  $\sH := \sV_{N_L}$ is cyclic.
  Our assumption implies that $S_\sV \cap R$ has interior points
in $R$ (Lemma~\ref{lem:dense-int}).  Hence there exists $r_0 \in (S_\sV \cap R)^\circ$
  and an $e$-neighborhood $N_R \subeq R$ with $r_0 N_R^{-1} \subeq S_\sV$.
  Then
  \[  U(\ell) U(r) U(r_0)^{-1} \sV \supeq U(\ell) \sV \supeq \sH
    \quad \mbox{ for } \quad \ell \in N_L, \ r \in N_R,\]
  and so regularity follows.
\end{prf}

\begin{rem}\label{rem:poi} {The condition on the cones
    $C_\pm$ in Theorem~\ref{thm:reg-posen} is stronger than
    the positive energy condition $C_U \not=\{0\}$.} The latter assumes
  the existence of a positive cone $C$ in the Lie algebra that $-i\partial U(x)\geq0$ for every $x\in C$ but does not require the generating property. Theorem \ref{thm:posreg} shows that, in order to recover the regularity of the net on Minkowski spacetime, one has to look at the  representation of the Poincar\'e group $\cP_+^\uparrow=\RR^{1,3}\rtimes \cL_+^\uparrow$  and to check the
non-triviality 
    of the one-dimensional cones $C_\pm$ in the eigenspaces
    $\fr_{\pm 1}(h) = \R (\be_0 \pm \be_1)$ (light rays)
    in the subalgebra $\fr \cong \R^{1,3}$ corresponding to translations,
  and the regularity property for the restriction of the representation
  to the {identity component $\cL_+^\up$ of the Lorentz group.}
  The first property is {equivalent to} the usual positive
  energy condition on Poincar\'e representations,
  namely the {joint spectrum of the translations is contained in}
  $\{x\in \RR^{1,3}:x^2\geq0,x_0\geq0\}$. The second one holds for every representation of the Lorentz group, see Example \ref{ex:reg} and
  Theorem~\ref{thm:local-reduct} below.
\end{rem}

\begin{rem} (a) If $G$ is simply connected,
  then $G \cong R \rtimes S$, where $S$ is semisimple and $R$
  is the solvable radical. In view of
  Theorem~\ref{thm:local-reduct}, which guarantees localizability
  for representations of $S$, Theorem~\ref{thm:posreg}
  applies whenever the cones
  $C_U \cap \fr_{\pm 1}(h)$ are generating, i.e.,
  the restriction of the representation to the abelian subgroups
  $R_\pm := \exp(\fr_{\pm 1}(h))$ have a generating positive cone.

  \nin (b) A similar remark applies to (coverings of) 
  identity components of   real algebraic groups.
  They are semidirect products $G= N\rtimes L$,
  where $N$ is unipotent and $L$ is reductive (\cite[Thm. VIII.4.3]{Ho81}).
  For these groups Theorem~\ref{thm:posreg} 
  applies whenever the cones
  $C_U \cap \fn_{\pm 1}(h)$ are generating.

  \nin (c) Presently we do not know if all (anti-)unitary representations
  of Lie groups of the form $G_{\tau_h}$, $h \in \g$ an Euler element,
  are regular. The preceding discussion shows that, to
  answer this question, a more detailed analysis of the case of solvable
  groups has to be undertaken.  
\end{rem}

\begin{prop} \mlabel{prop:2.18}
  Let $h \in \g$ be an Euler element and $G_{\tau_h}$ as above.
  An  (anti-)unitary  representation  $(U,\cH)$ of
  $G_{\tau_h}$ is regular if and only if its
  restriction to the connected normal subgroup $N_h^\natural$ with Lie algebra
  \[ \fn_h^\natural := \g_1(h) + (\R h + [\g_1(h), \g_{-1}(h)]) + \g_{-1}(h) \]
  is regular.   
\end{prop}

Note that the equality of
 $\g = \fn_h^\natural$ is equivalent to the Euler
element $h$ being anti-elliptic in~$\g$
(cf.\ Definition~\ref{def:antiell} below). 

\begin{prf} {
Since $\g = \fn_h^\sharp + \g_0(h)$ on the Lie algebra level,
    we obtain $G = N_h^\natural G^h_e$ for the corresponding integral
    subgroups, where $N_h^\sharp$ is a normal subgroup with Lie algebra
    $\fn_h^\sharp$ and $\L(G^h_e) = \g_0(h)$.}
  Then $G^h_e \subeq G^{h,\tau_h}$ implies that
  $G^h_e \subeq G_\sV$. For any $e$-neighborhood $N \subeq N_h^\natural$,
  we therefore
  have
  \[\bigcap_{g \in N G^h_e} U(g) \sV   = \bigcap_{g \in N} U(g) \sV.\]
  Therefore $U$ is regular if and only if $U\res_{N_h^\natural}$ is regular.   
\end{prf}

\begin{prop} We consider a group
  $G = E \rtimes \R$, where $E$ is a finite-dimensional vector space
  with Lie algebra of the form 
  \[ \g = E \rtimes \R h, \]
  where $h$ is an Euler element.
Then all  (anti-)unitary  representations of $G$ are regular.   
\end{prop}

\begin{prf} Let $E_j := \{ v \in E \: [h,v] = jv \}$ be the
  $h$-eigenspaces in~$E$. 
  By Proposition~\ref{prop:2.18}, it suffices to verify regularity
  on the subgroup $N_h^\sharp = (E_1 \oplus E_{-1}) \rtimes \R$.
  Using systems of imprimitivity, it follows that all irreducible
  unitary representations of such groups
  factor through representations of groups for which
  $\dim E_{\pm 1} \leq 1$.
  {In fact, all
    all orbits of $e^{\R \ad h}$ in $E^* = E^*_{-1} \oplus E^*_1$
    are contained in an at most $2$-dimensional subspace 
    because, for $\alpha = \alpha_{-1}  + \alpha_1$, we have
    \[ e^{\ad h}.\alpha
      = e^{-t} \alpha_{-1} + e^t \alpha_1 \in
      \R \alpha_{-1} +  \R \alpha_1.\]
    As irreducible unitary representations of $G$ are build from
    $\exp(\R h)$-ergodic covariant projection-valued measures on $E^*$,
    we can mod out $\ker \alpha_{\pm j}$ to reduce to the situation
    where $\dim E_{\pm 1} \leq 1$.}

This reduces the problem to
  the cases where $\g$ is abelian, $\aff(\R)$ or
  $\fp(2) = \R^{1,1} \rtimes \so_{1,1}(\R)$. The simple orbit structure for $\R$ on the dual space $E^*$
  implies that in this case the cones
  \[ C_\pm := \pm C_U \cap E_{\pm 1} \]
  are always non-trivial, hence generating.
  Now regularity of all irreducible  (anti-)unitary  representations
  follows from Theorem~\ref{thm:reg-posen}.

  Moreover, Remark~\ref{rem:quant} implies that, for all
  compact $e$-neighborhoods $N \subeq G$
  (which project to compact identity neighborhoods in the three types
  of quotient groups), the subspaces $\sV_N$ are cyclic. 
  As $N$ is independent of the representation, we can
  use Lemma~\ref{lem:g-intera} to obtain the result in general.
\end{prf}

\begin{rem} Let $(U,\cH)$ be an irreducible  (anti-)unitary  representation
  of the connected Lie group~$G$ and
  $0 \not=v \in \cD(\Delta^{1/2})$ be an analytic vector.
  If $\xi \in \sV_A$, then $U(g)^{-1} \xi \in \sV$ holds for all
  $g$ in~$A$ and, if $A^\circ \not=\eset$, then the analyticity of the map 
  $G \to \sV, g \mapsto U(g)^{-1} v$ and the closedness of $\sV$
  imply that $U(G)v \subeq \sV$, so that
  \[ \sV_A \cap \cH^\omega \subeq \sV_G.\]
  If $\sV_A \cap \cH^\omega$ is dense in $\sV_A$ and $\sV_A$ is cyclic,
  it follows that $\sV_G$ is cyclic.
  Its invariance under the modular group of $\sV$
  then implies that $\sV = \sV_G$
  (\cite[Prop.~3.10]{Lo08}).
  Therefore $\sV$ is $G$-invariant and thus $h$ is central in $\g$
  if $\ker(U)$ is discrete. 
In  view of \cite[Thm.~7.12]{BN23},
  one should not expect that $\sV$ contains non-zero analytic vectors
  if $\sV_G = \{0\}$.  For more details on the subspace $\sV_G$, we refer to
  Section~\ref{subsec:svg} below. 
\end{rem}

\subsection{Localizability}
\mlabel{subsec:4.2}

In this section we study localizability properties of unitary
representations of a connected Lie group $G$.

\begin{defn} \mlabel{def:local}
    We say that the (anti-)unitary representation 
  $(U,\cH)$ of $G_{\tau_h}$ is 
  {\it $(h,W)$-localizable in those open subsets $\cO \subeq M$  
    for which  $\sH^{\rm max}(\cO)$ is cyclic.}
\end{defn}

{The following remark show that already the localizability condition
  in the wedge region $W$ has consequences for the representation.}

\begin{rem} By Lemma~\ref{lem:direct-net}(c)
  the property of $(h,W)$-localizability implies
  $S_W \subeq S_\sV$.
From  \cite[Thm.~3.4]{Ne22} we recall that 
\begin{equation}
  \label{eq:svform}
 S_\sV := \{ g \in G \: U(g)\sV \subeq \sV\}
  = G_{\sV} \exp(C_+ + C_-) \quad \mbox{ with } \quad
  C_\pm = \pm C_U \cap \g_{\pm 1}(h)
\end{equation}
if $\ker(U)$ is discrete.
If the Lie wedge
\[ \L(S_W) = \{ x \in \g \: \exp(\R_+ x) \subeq S_W \} \]
is  {not contained in $\g_0(h)$ (see Proposition~\ref{prop:LSW}
  for a description of this cone for positivity domains), this implies that}
one of the two cones
  \[ \L(S_\sV) \cap \g_{\pm 1}(h) = C_\pm
=  \pm C_U \cap \g_{\pm 1}(h) \] 
  is non-zero and thus $C_U  \not= \{0\}$.
  If $S_W = G_W$ is a group, this conclusion is not possible,
  so that localizability {does not require any spectral condition,
    in particular $C_U = \{0\}$ is possible.}
\end{rem}

\
\begin{rem}
For the canonical nets obtained from pairs $(h,W)$ on a homogeneous space 
$M = G/H$ through two (anti-)unitary representations $U_1$, $U_2$ of $G_{\tau_h}$,
as in \eqref{eq:def-ho}, Lemma~\ref{lem:tensprostand}  shows that, 
for a tensor product representation $U = U_1 \otimes U_2$, we have 
\[ \sH^{\m}(\cO) \supeq \sH_{1}^{\m}(\cO) \otimes \sH_{2}^{\m}(\cO), \]
and in general equality does not hold (Example~\ref{ex:sl2-inc}).

\end{rem}

\begin{lem} \mlabel{lem:loc-imp-reg}
    {\rm(Localizability implies regularity)}
  Let $\eset\not=\cO \subeq W \subeq M$ be open subsets such that
  $N := \{ g \in G \: g^{-1}\cO \subeq W \}$ is an $e$-neighborhood.
  If $(U,\cH)$ is an (anti-)unitary representation for which
  $\sH^{\m}(W) = \sV$ and $\sH^{\m}(\cO)$ is cyclic, then it is regular.
  \end{lem}

  \begin{prf} By assumption $\sH^{\m}(\cO)$ is cyclic, and
  \[ \sH^{\m}(\cO) \subeq  \bigcap_{g \in N} \sH^{\m}(gW)
    = \bigcap_{g \in N} U(g) \sH^{\m}(W) 
=  \bigcap_{g \in N} U(g) \sV = \sV_N. \]
  It follows that $\sV_N$ is cyclic. 
\end{prf}

{Nets satisfying (Iso) and (Cov) can easily be constructed as follows. 
  Given an (anti-)unitary representation $(U, \cH)$  of $G_\tau$,
  the subspace  
$\cH^\infty \subeq \cH$ of vectors $v \in \cH$ for which the orbit map 
$U^v \: G \to \cH, g \mapsto U(g)v$, is smooth 
({\it smooth vectors}) is dense
and carries a natural Fr\'echet topology for which the action of 
$G$ on this space is smooth (\cite{Go69, Ne10},
\cite[App.~A]{NO21}, and Appendix~\ref{app:b}). 
The space $\cH^{-\infty}$ of continuous antilinear functionals $\eta \: \cH^\infty \to \C$ 
({\it distribution vectors}) 
contains in particular Dirac's kets 
$\la \cdot, v \ra$, $v \in \cH$, so that 
we obtain complex linear embeddings 
\[ \cH^\infty \into \cH \into \cH^{-\infty},\] 
where $G$ acts on all three spaces 
by representations denoted $U^\infty, U$ and $U^{-\infty}$, respectively.

All of the three above  representations can be integrated to the
convolution algebra $C^\infty_c(G,\C)$ of
test functions, for instance $U^{-\infty}(\phi) := \int_G \phi(g)U^{-\infty}(g)\, dg$,
where $dg$ stands for a left Haar measure on $G$.
The operators
$U(\phi)$ are continuous maps $\cH \to \cH^\infty$, so that
their adjoints $U^{-\infty}(\phi)$ define maps $\cH^{-\infty} \to \cH$. 
For any real subspace $\sE \subeq \cH^{-\infty}$, we can therefore
associate to every open subset 
$\cO \subeq G$, the closed real subspace 
\begin{equation}
  \label{eq:HE}
  \sH_\sE^G(\cO) := \oline{\spann_\R U^{-\infty}(C^\infty_c(\cO,\R))\sE}.
\end{equation}
On a homogeneous space $M = G/H$ with the projection map
$q \: G \to M$, we now obtain a ``push-forward net''
\begin{equation}
  \label{eq:pushforward}
  \sH^M_\sE(\cO) := \sH_\sE^G(q^{-1}(\cO)).
\end{equation}
This assignment satisfies (Iso) and (Cov), so that a key problem
is to specify subspaces $\sE$ of distribution vectors for which
(RS) and (BW) hold as well.}

Suppose that $\g$ is simple and $h \in \g$ an Euler element,
  and that $M = G/H$ is the corresponding
  non-compactly causal symmetric space (cf.\ Subsection~\ref{subsec:ncc}).
    In \cite{FNO23} a net of standard subspaces
    $\sH_E^M$ has been constructed on open regions of $M$,
    satisfying (Iso), (Cov), (RS), (BW), where
    $W  = W_M^+(h)_{eH}$. The following lemma applies in
  particular to these nets:

\begin{lem} \mlabel{lem:eimpliesloc}
  Let $(U,\cH)$ be an (anti-)unitary representation
  and $\sE \subeq \cH^{-\infty}$ be a real subspace with 
  $\sV = \sH^M_\sE(W)$. If the net $\sH^M_\sE$ has the
  Reeh--Schlieder property {\rm(RS)}, then
  $\sH^{\m}(\cO)$ is cyclic for any non-empty open subset
  $\cO \subeq M$. 
\end{lem}

\begin{prf} 
  Since $\sH^M_\sE(\cO)$ is cyclic for each non-empty open subset
  $\cO \subeq M$ by (RS), it suffices to verify that
  $\sH^M_\sE(\cO) \subeq \sH^{\rm max}(\cO)$. 
  As the net $\sH^M_\sE$ is covariant,  isotone and has the BW property
  with respect to $h$ and $W$, this follows from Lemma~\ref{lem:maxnet-larger}.
\end{prf}

\begin{ex} \mlabel{ex:4.24} 
We now describe an example of a net 
    $\sH^M_\sE$ constructed from a standard subspace
    $\sV = \sV(h,U)$ for which the corresponding
    maximal net $\sH^{\rm max}$ is strictly larger on some open subsets.
    Here $M = \R$, with its natural causal structure, on which
    we consider the group $G = \Aff(\R)_e$, acting by affine maps.

  On the space $C^\infty_c(\R,\R)$ of real-valued test functions on $\R$,
  we consider the positive definite hermitian form, given by
  \[ \la f,  g \ra_1
    := \int_{\R_+} p \oline{\hat f(p)}\hat g(p)\, dp = \int_{\R_+} p \hat f(-p)\hat g(p)\, dp\]
where  the Fourier transform is defined
$\hat f(p) = \int_\R e^{ipx}f(x)\, dx$. 
 We write $\cH^{(1)}$ for the real Hilbert space obtained
    by completion with respect to this scalar product
    and $\eta \: C^\infty_c(\R,\R) \to \cH^{(1)}$ for the canonical
  inclusion.
  The symplectic form corresponding to its imaginary part
  is
  \begin{equation}
    \label{eq:sigma}
 \sigma_1(f,g)
= \Im\int_{\R_+} p \hat f(-p)\hat g(p)\, dp
=\frac{1}{2i} \int_{\R} p \hat f(-p)\hat g(p)\, dp
= \pi \int_\R f(x)g'(x)\,dx.
  \end{equation}

Let $G := \Aff(\R)_e$ be the connected affine group.
Then the canonical action of $G$ on $C^\infty_c(\R,\R)$ by
$(g.f)(x) := f(g^{-1}x)$ preserves the hermitian form
and the Fourier transform intertwines it with the
unitary representation on $L^2(\R_+, pdp)$ by
\[ (\tilde U(b,a)F)(p) = e^{ibp} a F(ap), \quad b \in \R, a,p \in \R_+.\]
As $\tilde U$ extends to an irreducible unitary representation 
$\tilde U$ of $\PSL_2(\R)$ (cf.\ \cite[\S 5.4]{FNO23}), 
Corollary~\ref{cor:affine-irrep} implies that
$\tilde U$ is irreducible over $\R$.
It follows in particular that 
the Fourier transform $C^\infty_c(\R,\R) \to L^2(\R_+, p\, dp)$
has dense range.
We thus obtain a real linear
isometric bijection $\cH^{(1)} \to L^2(\R_+, p\, dp)$.
Bypassing the Fourier transform, we can also
write the scalar product, extended to complex-valued test
functions, as
\[ \la f,g\ra_1=\int_{\R_+}  \overline{\hat f(p)}\hat g(p)\, pdp
  =\int_\RR\int_\RR \oline{f(x)}g(y)\frac{(-1)}{(y-x+i0)^2}dx\,dy.\]
We consider the unitary representation $U^{(1)}$ of $G$ on $\cH^{(1)}$,
for which the Fourier transform is an intertwining operator onto
$L^1(\R_+, p\, dp)$.
Note that $\cH^{(1)}$ may also be considered as a Hilbert subspace
of $\cS'(\R)$ via the map $\iota(g)(f) = \la f, g \ra_1$
for $f,g\in \cS(\R)$, i.e., 
\[ \iota(g) = g * D \quad \mbox{ with }  \quad
  D(x) = \frac{(-1)}{(-x + i 0)^2}.\]
The antilinear involution
$(jf)(x) := -\oline{ f(-x)}$ on $C^\infty_c(\R)$
induces a conjugation on $\cH^{(1)}$ that extends $U^{(1)}$ to an 
{(anti-)unitary} representation $G_{\tau_h} \cong \R \rtimes \R^\times
= \Aff(\R)$ for the Euler element $h = (0,1)\in \g$.
On $L^2(\R_+,p\, dp)$, $j$ corresponds to the conjugation
defined by $JF = -\oline F$. 
Here $(h,-1) \in \cG_E(\Aff(\R))$ and
$\cW_+ = G.(h,-1)$ can be identified with the set 
 of open real half-lines, bounded from below. 

Clearly, 
\[\sH^{(1)}(\cO) := \overline{\eta(C^\infty_c(\cO,\R))} \] 
defines a net of real subspaces in $\cH^{(1)}$ that is isotone 
and $G$-covariant. 
Furthermore \eqref{eq:sigma} implies that this net is local 
in the sense that disjoint open intervals map to 
  symplectically orthogonal real subspaces.
It also satisfies the Reeh--Schlieder property and also the BW property in 
the sense that 
\[ \sV = \sV(h,U) = \sH^{(1)}(\R_+) \] 
(cf.~\cite{Lo08,NOO21}).
Here the main point is to verify
that the constant function $1$, a distribution vector for the representation
on $L^2(\R_+, p\, dp)$ satisfies the abstract KMS condition
\begin{equation}
  \label{eq:u1-KMS}
 J1 = - 1 = \Delta^{1/2} 1 \quad \mbox{ for }\quad
 \Delta = e^{2\pi i \partial U(0,1)}
\end{equation}
(cf. \cite{BN23}). As $\tilde U(0,e^t)1 = e^t$, the relation
\eqref{eq:u1-KMS} follows immediately.
For $k \geq 2$, we 
also have the following subnets, generated by the derivatives 
of test functions via 
\[ \sH^{(k)}(\cO)= \overline{\{ \eta(f^{(k-1)}) \: f \in C^\infty_c(\cO,\R)\}} 
\subeq \sH^{(1)}(\cO). \] 
These nets are also isotone and $G$-covariant. 
It is known from \cite[Prop.~4.2.3]{Lo08} and \cite{GLW98} 
that, for every bounded interval $I \subeq \R$
and $k < \ell$, the subspace $\sH^{(\ell)}(I) \subeq \sH^{(k)}(I)$
is proper with 
\[ \dim \big(\sH^{(k)}(I)/\sH^{(\ell)}(I)\big) =\ell-k.\]  On the other hand,
when $I = (a,\infty)$ is an unbounded interval,
then $\sH^{(k)}(I) = \sH^{(1)}(I)$ for every $k\in\N$. 
Furthermore, on intervals, 
$\sH^{(k)}$ is a restriction of the 
BGL net associated to the unitary positive energy representation 
$\tilde U^{(k)}$ of $\PSL_2(\RR)$ of lowest weight $k$
(\cite[Thm.~3.6.7]{Lo08}). 

Finally, we explain how to write these nets in the form 
$\sH^\R_{\sE_k}$ for suitable one-dimensional subspaces 
$\sE_k = \R \alpha_k \subeq \cH^{-\infty}$ of distribution 
vectors of the representation $(U^{(1)}, \cH^{(1)})$. 
To this end, we consider the Fourier transform
  $L^2(\R_+, p\, dp) \to \cO(\C_+),
  \cF_1(F)(z) = \int_{\R_+} e^{ipz} F(p)\, p dp$, which maps
  unitarily onto
  the reproducing kernel Hilbert space $\cH_1 \subeq \cO(\C_+)$ 
with reproducing kernel 
\[ Q(z,w) = \frac{-1}{(z - \oline w)^2} \quad \mbox{ for } \quad 
  z,w \in \C_+ = \R + i \R_+\]
(\cite{NOO21}). 
Here $J$ acts by $(JF)(z) := -\oline{F(-\oline z)}$
and the affine group by
\[ (U_1(b,a)F)(z) = a^{-1} F(a^{-1} (z+b)).\] 
The discussion in \cite[\S 5.4]{FNO23} shows that
\[ \alpha_1(x) := (x + i 0)^{-2}, \quad \mbox{ resp. } \quad
\alpha_1(z) = \frac{1}{z^2}, \] is a distribution vector
that is an eigenvector for the dilation group, satisfying
$U_1^{-\infty}(a)\alpha_1 = a \alpha_1$ and $J \alpha_1 = - \alpha_1$. 
For $\sE_1 := \R \alpha_1$, the corresponding standard subspace
$\sH^\R_{\sE_1}(\cO)$ is therefore generated by
the elements $U_1^{-\infty}(\phi)\alpha_1 = \phi^\vee * \alpha_1$,
$\phi \in C^\infty_c(\cO,\R)$, so that we obtain for each
open subset $\cO \subeq \R$: 
\[ \sH^\R_{\sE_1}(\cO) = \sH^{(1)}(-\cO). \]
We also note that
\[ U^{(1)}(\phi^{(k)})\alpha_1
  = \phi^{(k),\vee} * \alpha_1
  = (-1)^k \phi^\vee * \alpha_1^{(k)},\]
so that we obtain
\[ \sH^{(k)}(-\cO) = \sH^\R_{\sE_k}(\cO) \quad \mbox{ for } \quad
  \sE_k = \R \alpha_k, \quad \alpha_k := \alpha_1^{(k-1)}.\]
\end{ex}

An example on 1+1-dimensional Minkowski spacetime is described in
Remark~\ref{ex:sl2-incb}.

\subsubsection*{Localizability for reductive groups}

In this section we assume that $\g$ is reductive
and that $G$ is a corresponding connected Lie group.
We choose an involution $\theta$ on $\g$ in such a way that 
  it fixes the center pointwise and restricts to a Cartan involution
  on the semisimple Lie algebra $[\g,\g]$. Then
  the corresponding Cartan decomposition $\g = \fk \oplus \fp$ satisfies
$\fz(\g) \subeq \fk$. We write $K := G^\theta$ for the subgroup of
$\theta$-fixed points in $G$.

We write 
\[ \g = \g_0 \oplus \bigoplus_{\gamma \in \Gamma} \g_\gamma, \]
  where $\g_0 = \z(\g)$ is the center
  and each ideal $\g_\gamma$ is simple.
  Accordingly, we have 
  \[ h = h_0 + \sum_\gamma h_\gamma,\]  where 
$h_\gamma \in \g_\gamma$ either vanishes or is an Euler element in~$\g_\gamma$.
We assume that $\theta(h_\gamma) = - h_\gamma$ for each
$\gamma \in \Gamma$. We decompose $\Gamma$ as
\begin{equation}
  \label{eq:gammapm}
 \Gamma = \Gamma_0 \dot\cup \Gamma_1 \quad \mbox{ with } \quad
 \Gamma_0 := \{ \gamma \in \Gamma \: h_\gamma = 0\},
\end{equation}
so that
$h = h_0 + \sum_{\gamma \in \Gamma_1} h_\gamma$.
Then
we obtain an involutive automorphism $\tau$ on $\g$ by
\[ \tau(x) = 
  \begin{cases}
    x \qquad\quad \text{ for }\quad x \in \g_0 = \fz(\g), \\ 
    x \qquad\quad \text{ for } \quad x \in \g_\gamma, \gamma \in \Gamma_0,  \\ 
    \tau_h\theta(x) \quad\, \text{for } \quad x \in \g_\gamma,
    \gamma \in \Gamma_1, 
  \end{cases}
  \] 
  and we assume that $\tau$ integrates to an involutive automorphism
  $\tau^G$ of $G$.
  We write $\fh := \g^\tau$ and $\fq := \g^{-\tau}$
  for the $\tau$-eigenspaces in $\g$.
    Then there exists in $\fq$ 
    a unique maximal pointed generating $e^{\ad \fh}$-invariant
    cone $C$ containing $h' := \sum_{\gamma \in \Gamma_1} h_\gamma$ in its interior
    (\cite{MNO23a})
    We choose an open $\theta$-invariant
    subgroup $H \subeq G^\tau$ satisfying $\Ad(H)C = C$.
    By \cite[Cor.~4.6]{MNO23a}, this is
    equivalent to $H_K= H \cap K$ fixing~$h$.
    {Here we use that $H$ has a polar decomposition
      $H = H_K \exp(\fh_\fp)$, so that the above condition implies that
      $\Ad(H)h = e^{\ad \fh_\fp}h$.} Then
    \begin{equation}
      \label{eq:ncc}
      M = G/H
    \end{equation}
    is called the corresponding {\it non-compactly causal symmetric space}.
    The normal subgroups $G_0 = Z(G)_e$ and $G_j$ for $h_j = 0$,
    are contained in $H$, hence act trivially on $M$.
    The homogeneous space $M$ carries a $G$-invariant causal
    structure, represented by a field $(C_m)_{m \in M}$ of closed convex cones
    $C_m \subeq T_m(M)$, which is uniquely determined by
    $C_{eH} = C \subeq \fq \cong T_{eH}(M)$.

The modular vector field 
\begin{equation}
  \label{eq:xhdef2}
 X_h^M(m)  =  \frac{d}{dt}\Big|_{t = 0} \exp(th).m \end{equation} 
on $M$ determines a positivity region 
\begin{equation}\label{def:WM2}
 W_M^+(h) := \{ m \in M \: X^M_h(m) \in C_m^\circ  \} 
 \end{equation}
and the connected component $W := W_M^+(h)_{eH}$ of the base point
$eH \in M$ is called the {\it wedge region in $M$}.

{Note that the following theorem does not require any
  assumption concerning the irreducibility of the representation.
  Although its proof draws heavily from \cite{FNO23}, which deals
  with irreducible representations, Proposition~\ref{prop:3.2}
  is a convenient tool to reduce to this situation.} 
  
\begin{thm} \mlabel{thm:local-reduct}
  {\rm(Localization for real reductive groups)}
  If the universal complexification
  $\eta \: G \to G_\C$ of the connected reductive group $G$ is injective and
  $(U,\cH)$ is an (anti-)unitary representation
  of $G_{\tau_h}$, then the canonical net $\sH^{\m}$ on
  the non-compactly causal symmetric space $M = G/H$
  associated to $h$ as in \eqref{eq:ncc} satisfies 
  \begin{itemize}
  \item[\rm(a)] $\sV = \sH^{\m}(W)$, i.e., $S_W \subeq S_\sV$, and 
  \item[\rm(b)] $\sH^{\m}(\cO)$ is cyclic for every
        non-empty open subset $\cO \subeq M$. 
  \end{itemize}
\end{thm}

\begin{prf} In view of Lemma~\ref{lem:direct-net}(c), assertion (a)
  follows from~(b). So it suffices to verify~(b).
  By Proposition~\ref{prop:3.2} we may further assume that
  $(U,\cH)$ is irreducible. 
    Replacing $G$ by a suitable covering group, we may assume that
  the universal complexification $G_\C$ is simply connected, and then
  \[ G_\C \cong G_{0,\C}  \times \prod_{\gamma \in \Gamma} G_{\gamma,\C} \]
leads to the product structure 
\[ G \cong G_0  \times \prod_{\gamma \in \Gamma} G_\gamma. \]
Moreover,
\[ \fq   = \bigoplus_{\gamma \in \Gamma_1} \fq_\gamma \quad \mbox{ and } \quad
  C = \sum_{\gamma \in \Gamma_1} C_\gamma \quad \mbox{ with } \quad
  C_\gamma = C \cap \fq_\gamma\]
(cf.\ \eqref{eq:gammapm}).

We first consider irreducible representations of the factor groups
$G_{\gamma,\tau_h}$.
  If $h_\gamma \in \g_\gamma$ is trivial or central, then the standard subspace
  $\sV$ is $G_j$-invariant, so that $\sV = \sV_{G_j}$. 
  For all other simple factors $(h_\gamma,W_\gamma)$-localizability 
  in the family of all non-empty open subsets
  of the associated non-compactly causal symmetric space follows from
  \cite[Thm.~4.10]{FNO23}.
  This implies the assertion for all irreducible
  (anti-)unitary representations of the factor groups
  $G_{\gamma,\tau_h}$ and $G_{0,\tau_h}$.

Let $U_0 \otimes \bigotimes_{\gamma \in \Gamma} U_\gamma$ be an
 irreducible unitary representation of $G$ and extend it by some
 conjugation of the form
 $J = J_0 \otimes \bigotimes_{\gamma \in \Gamma} J_\gamma$ to an
 irreducible (anti-)unitary representation $(U, \cH)$ of $G_{\tau_h}$
 on a Hilbert space that is a subspace of
 the tensor product of the spaces
 \[\tilde\cH_\gamma = \cH_\gamma \oplus \oline{\cH_\gamma}.\] 
 By Remark~\ref{rem:tenspro-anti}, all irreducible
 (anti-)unitary representations of $G_{\tau_h}$ are subrepresentations of
 tensor products  of irreducible (anti-)unitary representations of the factor groups.
 We thus obtain all irreducible (anti-)unitary
 representations of $G_{\tau_h}$. 
 Therefore the assertion follows from the fact that
  (b) is inherited by subrepresentations, direct sums, and
  finite tensor products (Lemma~\ref{lem:direct-net}(d)).
\end{prf}

\begin{cor}   \mlabel{cor:real-red} {\rm(Regularity for linear reductive groups)} 
  Let $G$ be a connected linear reductive Lie group,
  i.e., its universal complexification is injective and
  $G_\C$ is a complex reductive algebraic group.
  Then there exists an $e$-neighborhood $N \subeq G$
  such that for every separable (anti-)unitary representation
  $(U,\cH)$ of $G_{\tau_h}$, the real subspace
  \[  \sV(h,U)_N = \bigcap_{g \in N} U(g) \sV(h,U) \]
  is cyclic. In particular, $(U,\cH)$ is $h$-regular.   
\end{cor}

\begin{prf} Let $\cO \subeq W \subeq M = G/H$ be an open subset
  whose closure $\oline\cO$ is relatively compact.
  In Theorem~\ref{thm:local-reduct} we have seen that
  $\sH^{\m}(\cO)$ is cyclic. Further
  \[ N := \{g \in G \: g\cO \subeq W\} \supeq \{g \in G \: g\oline\cO \subeq W\} \]
  is an $e$-neighborhood because $\oline\cO \subeq W$ is compact.
  Theretofore the $h$-regularity of $(U,\cH)$ follows from
  Lemma~\ref{lem:loc-imp-reg}.
\end{prf}

\subsubsection*{Localizability for the Poincar\'e group}

The following result is well-known
(\cite[Thm.~4.7]{BGL02}). Here we derive it naturally in the
context of our theory for general Lie groups. 
It connects regularity, resp., localizability  with
  the positive energy condition.

\begin{thm} {\rm(Localization for the Poincar\'e group)}
\mlabel{thm:poinloc} Let $(U,\cH)$ be an (anti-)unitary representation
  of the proper Poincar\'e group
  $\cP_+ = \R^{1,d} \rtimes \cL_+$ (identified with $\cP_{\tau_h}$)
  and consider the standard boost
  $h$ and the corresponding Rindler wedge $W_R \subeq \R^{1,d}$.
  Then $(U,\cH)$ is $(h,W_R)$-localizable in
{the set of all spacelike open cones} 
  if and only if it is a positive energy representation, i.e.,
  \begin{equation}
    \label{eq:v+}
 C_U \supeq
 \oline{V_+} := \{ (x_0,\bx) \: x_0 \geq 0, x_0^2 \geq \bx^2 \}.
  \end{equation}
  These representations are regular. 
\end{thm}

Note that $\Ad(\cP_+^\uparrow)$ acts transitively on the set $\cE(\fp)$
of Euler elements, so that the choice of a specific
Euler element~$h$ is inessential. 

\begin{prf} First we show that the positive energy condition is necessary
  for localizability in spacelike cones.
  In fact, the localizability condition implies in particular
  that $\sH(W_R)$ is cyclic, so that Lemma~\ref{lem:direct-net} implies
  $S_{W_R} \subeq S_\sV$. As a consequence, 
  $\be_1 + \be_0 \in C_U$, and thus $\oline{V_+} \subeq C_U$ by
  Lorentz invariance of $C_U$. Therefore $(U,\cH)$ is a positive
  energy representation.
  
  Now we assume that $(U,\cH)$
  is a positive energy representation. 
  For the standard boost we have
  $h \in \fl \cong \so_{1,d}(\R)$, and the
  restriction $(U\res_{L_+}, \cH)$ is $(h,W)$-localizable
  in the family of all non-empty open subsets of $\dS^d$,
  where $W = W_R \cap \dS^d$ is the canonical wedge region 
  (Theorem~\ref{thm:local-reduct}).
  
  Next we recall from \cite[Lemma~4.12]{NO17} that
  \[ S_{W_R} = \{ g \in \cP_+^\uparrow \: gW_R \subeq W_R \}
    = \oline{W_R} \rtimes \SO_{1,d}(\R)^\up_{W_R},\]
  where
\[ \SO_{1,d}(\R)^\up_{W_R} 
  = \SO_{1,1}(\R)^\up \times \SO_{d-2}(\R)\]
is connected, hence coincides with $L^h_e$.
It follows that
\[ S_{W_R} = G^h_e \exp([0,\infty)(\be_0 + \be_1))
  \exp([0,\infty)(-\be_0 + \be_1)).\] 

Let us assume that $(U,\cH)$ is a positive energy representation, i.e.,
that { $C_U \supeq \oline{V_+}$ (cf.\ \eqref{eq:v+})}. Then
\[ C_\pm = [0,\infty)(\be_1 \pm \be_0) \subeq \oline{W_R},
\quad \mbox{ so that } \quad S_{W_R} \subeq S_\sV.\]
By Lemma~\ref{lem:direct-net}(c), the net $\sH^{\m}$
satisfies $\sH^{\m}(W_R) = \sV$.

Now suppose that $\cC \subeq W_R$ is a spacelike cone, so that
\[ \cC = \R_+ (\cC \cap \dS^d), \]
where $\cC \cap \dS^d$ is an open subset of the wedge region
$W = W_R \cap \dS^d$ 
in de Sitter space. For $g^{-1} = (v, \ell) \in \cP_+^\uparrow$, 
the condition $\cC \subeq g.W_R$ is equivalent to
\[ g^{-1}.\cC = v + \ell.\cC \subeq W_R, \]
which in turn means that $v \in \oline{W_R}$ and $\ell.\cC \subeq W_R$.
Then
\[ U(g) \sV  = U(\ell)^{-1} U(v)^{-1} \sV \supeq U(\ell)^{-1} \sV \]
follows from $\oline{W_R} \subeq S_\sV$, and therefore
{
  \begin{align*}
 \sH^{\m}(\cC)
&  = \bigcap_{\cC \subeq g.W_R} U(g)\sV
 \supseteq \bigcap_{\cC \subeq \ell^{-1}.W_R} U(\ell)^{-1} \sV\\
&  = \bigcap_{\cC \cap \dS^d \subeq \ell^{-1}.(W_R \cap \dS^d)} U(\ell)^{-1} \sV
    = \sH_{U\res_L}^{\m}(\cC \cap \dS^d).
  \end{align*}
}
We conclude that, on spacelike cones with vertex in $0$,
the net $\sH^{\m}$ coincides with the net $\sH_{U\res_L}^{\m}$ on de Sitter space.
As the latter net has the Reeh--Schlieder property
by Theorem~\ref{thm:local-reduct}, and all spacelike
cones can be translated to one with vertex~$0$,
localization in spacelike cones follows. 

Finally we show that $(U,\cH)$ is regular.
For $v \in W_R$ and a pointed
spacelike cone $C$ with $v + C \subeq W$, there exists an $e$-neighborhood
 $N \subeq G$ with $v + C \subeq g.W$ for all $g \in N$. This
 implies that $\sH^{\m}(v + C) \subeq \sV_N$, so that
 $(U,\cH)$ is regular.
\end{prf}

\begin{rem} 
Infinite helicity representations $(U,\cH)$ of $\cP_+$ in $\RR^{1,d}$ are 
{\bf not} localizable in double cones (Definition~\ref{def:minkcaus}).
Let  $\sV=\sH_U^{\mathrm BGL}(W)$ for $W = (h,j_h)$ be
  as in Example~\ref{ex:desit}.
  In  \cite[Thm. 6.1]{LMR16} it is proved that, if $\cO \subeq \R^{1,d}$ is a double cone, then
  \begin{equation}\label{eq:infspin}  \sH^{\m}(\cO)
    = \bigcap_{\cO \subeq g.W_R} U(g) \sV = \{0\}.
  \end{equation}
  The argument to conclude \eqref{eq:infspin} can be sketched as follows.
  Infinite spin representations are massless representations,
i.e., the support of the spectral measure of the space-time
    translation group is  
  \[ \partial V_+=\{(x_0,\bx) \in\RR^{1,d}:x_0^2 - \bx^2 =0,x_0 \geq 0\}.\]
Covariant nets of standard subspaces on double cones 
  in massless representations are also dilation covariant in the sense that   
  the representation of $\cP_+$ extends to the
Poincar\'e and dilation group 
$\R^{1,d} \rtimes (\R^+\times \cL)$, and that the net is also
covariant under this larger group, cf. \cite[Prop.~5.4]{LMR16}.
 When $d=3$, this follows from the fact that due to
  the Huygens {Principle}, one can associate by additivity 
  a standard subspace to the forward lightcone
  $\sH(V_+)=\overline{\sum_{\cO\subset V_+}\sH(\cO)}$
(sum over all double cones in $V_+$) and
  the modular group of $\sH(V_+)$ is geometrically implemented by
  the dilation group.
 As massless infinite helicity representations are not
  dilation covariant, it follows that they do not
  permit localization in double cones.
Properties of the free wave equation permit to  extend this  
 argument  to any space dimension $d\geq2$
  including  even space dimensions, and the 
  Huygens Principle fails (\cite[Sect.~8.2]{LMR16}).
    However, in
  Theorem~\ref{thm:poinloc}, we recover in our general setting the result contained in \cite[Thm.~4.7]{BGL02} that all   positive energy   representations of $\cP_+$
  are localizable in spacelike cones.    
\end{rem}

\section{Moore's Theorem and its consequences}
\label{sect:vna}

In this section we continue the discussion of 
  applications of our results to von Neumann algebras $\cM$
  with cyclic separating
vector $\Omega$, started in Subsection~\ref{subsec:strich-opalg}.
First we explain the consequences of 
Moore’s Eigenvector Theorem~\ref{thm:moore} (cf.\ \cite[Thm.~1.1]{Mo80}). 
Here the main point is that the properties (Mod) and (M)
(from Subsection~\ref{subsec:strich-opalg})  
imply that $\Omega$ is fixed by the one-parameter group $U(\exp(\R h))$
and Moore's Theorem allows us to find conditions for $G$
under which this always implies that $\Omega$ is fixed under~$G$.
Note that, for semisimple Lie groups Moore's Theorem also follows from
  the Howe--More Theorem   on the vanishing of matrix coefficients
  at infinity for all unitary representations non
  containing non-zero fixed vectors (cf.~ \cite[Thm.~2.2.20]{Zi84}).

  The first main result in this section are Theorem~\ref{thm:vg-ident},
  characterizing for (anti-)unitary representation
  $(U,\cH)$ of $G_{\tau_h}$ the subspace 
  $\sV_G = \bigcap_{g \in G} U(g) \sV$ as the set of fixed points
  of a certain normal subgroup specified in Moore's Theorem. 
  The second one is Theorem~\ref{thm:3.8} that combines Moore's Theorem
  with Theorem~\ref{thm:2.1-alg} to obtain a criterion for $\cM$ to be
  a factor of type III$_1$. The third one is Proposition~\ref{prop:disnet}
  which shows that all the structure we discuss survives the
  central disintegration of $\cM$, provided
  $\cM'$ and $\cM$ are conjugate under~$U(G)$.

\subsection{Moore's Theorem}
\mlabel{subsec:5.moore} 

\begin{thm} {\rm(Moore's Eigenvector Theorem)}  \mlabel{thm:moore}
Let $G$ be a connected finite-dimensional Lie group with Lie algebra 
$\g$ and $h \in \g$. Further, let $\fn_h \trile \g$ be the 
smallest ideal of $\g$ such that the image of $h$ in
the quotient Lie algebra $\g/\fn_h$ 
is elliptic. 

Suppose that $(U, \cH)$ is a continuous unitary representation of $G$ and 
$\Omega\in \cH$ an eigenvector for the one-parameter group  $U(\exp \R h)$. 
Then 
\begin{itemize}
\item[\rm(a)] $\Omega$ is fixed by the normal subgroup 
$N_h := \la \exp \fn_h \ra \trile G$, and 
\item[\rm(b)] the restriction of 
$i\cdot\partial U(h)$ to the orthogonal complement of the space 
$\cH^{N_h}$ of $N_h$-fixed vectors has absolutely continuous spectrum. 
\end{itemize}
\end{thm}

The ideal $\fn_h\trile \g$ has the property that the corresponding
closed normal subgroup $N_h \trile G_e$ generated by  $\exp(\fn_h)$
fixes $\Omega$, hence acts trivially on the projective orbit
$G.[\Omega] \subeq \bP(\cH)$.
As $\ad h$ induces an elliptic element on $\g/\fn_h$, 
the group $G/N_h$ has a basis 
of $e$-neighborhoods invariant under $\exp(\R h)$.

\begin{cor} \label{cor:moore-eigenvector} 
Let $G$ be a connected finite-dimensional Lie group. 
Suppose that $(U, \cH)$ is a 
unitary representation of $G$ with discrete kernel and 
that $h \in \g$ is such that $\partial U(h)$ has a
$G$-cyclic eigenvector in $\cH$. Then $\ad(h)$ is elliptic. 
\end{cor}

\begin{prf} It suffices to show that $\fn_h = \{0\}$.
  As the subgroup $N_h \trile G$ is normal, the subspace
  $\cH^{N_h}$ of $N_h$-fixed vectors is $G$-invariant:
 For $\xi \in \cH^{N_h}$, $g \in G$ and $n \in N_h$, we have
  \[ U(n) U(g)\xi = U(g)U(g^{-1}ng)\xi = U(g)\xi.\]
The 
$G$-cyclic eigenvector $\Omega$ of $\partial U(h)$ is contained
in $\cH^{N_h}$ by Moore's Theorem, so that $\cH = \cH^{N_h}$.
     Therefore
  $\fn_h \subeq \ker(\dd U) =  \{0\}$.    
\end{prf}

In many situations, Moore's Theorem implies that
eigenvectors of one-parameter subgroups are actually fixed by~$G$.
These cases are easily detected with the following concept: 

\begin{defn} \mlabel{def:antiell}
  We call $h \in \g$ {\it anti-elliptic} if $\fn_h + \R h = \g$.
  \end{defn}

  \begin{rem} In \cite{Str08} a closely related property has been
    introduced for Lie algebra elements:
    An element $x \in \g$ for which $\ad x$ is diagonalizable is said to
    be {\it essential} if
    \[ \g = \R x + [x,\g] + \Spann [[x,\g], [x,\g]].\]
    As $\g = \sum_{\lambda \in \R} \g_\lambda(x)$ and
$[x,\g] = \sum_{\lambda\not=0} \g_\lambda(x)$, this is equivalent to  
\[ \g_0(x) = \R x + \sum_{\lambda \not=0} [\g_\lambda(x), \g_{-\lambda}(x)].\]
In this case the ideal $\fn_x$ contains all
eigenspaces $\g_\lambda(x)$ for $\lambda \not=0$, hence also the brackets
$[\g_\lambda(x), \g_{-\lambda}(x)]$. As
\[ \fri := \sum_{\lambda \not=0} \g_\lambda(x) + 
  \sum_{\lambda \not=0} [\g_\lambda(x), \g_{-\lambda}(x)] \]
is an ideal of $\g$ for which the image of $x$ in $\g/\fri$ is central,
it follows that $\fri = \fn_x$. Therefore an $\ad$-diagonalizable
element is essential if and only if it is anti-elliptic.
In this sense our concept of intrepidity extends
Strich's concept of essentiality to general Lie algebra elements.   
\end{rem}

  \begin{remark} 
  The assumption of $h$ to be anti-elliptic holds 
  if $h$ is an Euler element in a simple Lie algebra. 
 But $h=\frac12\diag(1,-1)$ is an Euler element in
  the reductive Lie algebra
  $\gl_2(\R)$ with $\fn_h = \fsl_2(\R)\ni h$. So it is
  not anti-elliptic. 
  \end{remark}

  Moore's Theorem immediately yields:  
  \begin{cor} \mlabel{cor:inheritfixed} If $h \in \g$ is anti-elliptic and
    $(U,\cH)$ is a unitary representation of a connected Lie group
    $G$ with Lie algebra $\g$, then $\ker(\partial U(h)) = \cH^G$.
  \end{cor}

  \begin{prf} As $\ker(\partial U(h))$ consists of
      eigenvectors for $U(\exp \R h)$, Moore's Theorem implies that
      they are fixed by $U(N_h)$. Anti-ellipticity of $h$ further
      implies that $G = N_h \exp(\R h)$, so that they are fixed by~$G$.    
  \end{prf}

  \begin{exs} \mlabel{exs:antiell} 
 (a) If $\g$ is simple and $h \in \g$ is not elliptic, then
  $\fn_h \not=\{0\}$ implies $\fn_h = \g$, so that
  $h$ is anti-elliptic. If, more generally, $\g$ is reductive such that
  $\g = \R h + [\g,\g]$ and no restriction of $\ad h$ to a simple
  ideal of $\g$ is elliptic, then $h$ is anti-elliptic. 

\nin (b) Consider a semidirect sum of Lie algebras
    $\g = \fr \rtimes \fl$ and an element
    $h \in \fl$ such that
    \begin{equation}
      \label{eq:imag-spec}
      \Spec(\ad h\res_{\fr}) \cap i\R = \eset
    \end{equation}
    and $h$ is anti-elliptic in $\fl$. Then
    $h$ is anti-elliptic in $\g$. In fact, our assumption implies that
    $\fr \subeq \fn_h$, so that $\g/\fn_h \cong \fl/(\fl \cap \fn_h)
    \cong \fl/\fl_h$ is linearly
    generated by the image of $h$. This implies that $\g = \fn_h + \R h$.

\nin (c) If $\g= \R x + \R h$ with $[h,x] = \lambda x$ and $\lambda \not=0$,
then $\fn_h = \R x$, so that $h$ is anti-elliptic
  (cf.\ \cite{Str08}). 
    
\nin (d) Consider the boost generator $h \in \so_{1,1}(\R) \subeq \fp(2)
    = \R^{1,1} \rtimes \so_{1,1}(\R)$, the $2d$-Poincar\'e--Lie algebra. 
    Then $\fn_h = \R^{1,1}$ and $\g = \fn_h +  \R h$, so that
    $h$ is anti-elliptic.

\nin (e) From (a) and (b) it follows immediately that, for $d \geq 3$,
    any boost generator $h \in \so_{1,d-1}(\R)
    \subeq \fp(d) = \R^{1,d-1} \rtimes \so_{1,d-1}(\R)$ 
    is anti-elliptic. Here we use that the representation
    of $\so_{1,d-1}(\R)$ on $\R^{1,d-1}$ is irreducible. 

    \nin (f)   Suppose that $\g$ is reductive and $h \in \g$ is an
    Euler element. {Since every ideal of a reductive Lie algebra
      possesses a complementary ideal (\cite[Def.~5.7.1]{HN12}),
      we can write
      $\g = \fn_h \oplus \fb$.
  We write accordingly $h = h_0 + h_1$ with $h_0 \in \fn_h$ and
  $h_1 \in \fb$.
  If $\fn_h$ is not central, then $h_0$ is an Euler element of $\fn_h$.
  Further, $h_1$ is elliptic in $\fb \cong \g/\fn_h$. From the direct
  sum decomposition we thus infer that $h_0$ is an Euler element of $\g$
  and that $h_1$ is elliptic.}
\end{exs}

\begin{lem} \mlabel{lem:eul-nh}
  If $h \in \g$ is an Euler element, then 
  \[  \fn_h = \g_1(h) + [\g_1(h), \g_{-1}(h)] + \g_{-1}(h). \]
  In particular, $h$ is anti-elliptic if and only if
  \[ \g_0(h)\subeq \R h + [\g_1(h), \g_{-1}(h)].\] 
\end{lem}

\begin{prf} Clearly, $\g_{\pm 1}(h) \subeq \fn_h$  implies that
  $\fn_h$ contains the ideal
  \[  \fn := \g_1(h) + [\g_1(h), \g_{-1}(h)] + \g_{-1}(h). \]
  As the image of $h$ in $\g/\fn$ is central,
  we have $\fn_h = \fn$. 
  Hence $h$ is anti-elliptic if and only if
  $\g_0(h) \subeq \R h + [\g_1(h), \g_{-1}(h)]$.
\end{prf}

\begin{rem} If $h$ is an Euler element, then
  Lemma~\ref{lem:eul-nh} shows that
  \[ \g = \fn_h + \g_0(h),\]
  so that {the summation map is a surjective homomorphism
    $\fn_h \rtimes \g_0(h) \onto \g$. Hence 
  $\g$ is a quotient of $\fn_h \rtimes \g_0(h)$,
  where $h \in \g_0(h)$ is central.}
\end{rem}

\begin{rem} If $h$ is an Euler element, then
  \[ \fn_h^\natural :=
    \fn_h + \R h =  \g_1(h) + (\R h + [\g_1(h), \g_{-1}(h)]) + \g_{-1}(h) \]
  is an ideal of $\g$. {It is the minimal ideal} containing $h$,
  and therefore
  the corresponding integral subgroup of $G$ is generated by
  $\exp(\Ad(G)h)$.
  Therefore $h$ is anti-elliptic if and only if the modular groups
  $\exp(\Ad(g) \R h)$ generate~$G$.   
\end{rem}

\subsection{Non-degeneracy}
\mlabel{subsec:svg}

Let $(U,\cH)$ be an (anti-)unitary
  representation of $G_{\tau_h}$, where $h \in \g$ is an Euler element
  and $\sV = \sV(h,U)$ is the canonical standard subspace.

We consider the $G$-invariant closed real subspace 
\[ \sV_G = \bigcap_{g \in G} U(g) \sV. \]

We call the couple $(U,\sV)$ \textit{non-degenerate} if $\sV_G=\{0\}$.
  We shall see in this context how this property is related to the structure introduced in the previous section.

\begin{thm} \mlabel{thm:vg-ident}
  Suppose that $G$ is connected,
  $h \in \g$ is an Euler element, $(U,\cH)$ an \break 
  {(anti-)unitary} representation of $G_{\tau_h}$, and
  $\sV = \sV(h,U)$ the corresponding standard subspace.
  Then $\sV_G = \sV \cap \cH^{N_h}$, where
  $N_h$ is the normal subgroup from {\rm Moore's Theorem~\ref{thm:moore}}. 
\end{thm}

\begin{prf} Let $\cH_1 := \cH^{N_h}$ and $\cH_2 := \cH_1^\bot$.
  As $N_h \trile G$ is a normal subgroup of $G_{\tau_h}$,
  the decomposition $\cH = \cH_1 \oplus \cH_2$ is $U(G_{\tau_h})$-invariant,
  so that $U = U_1 \oplus U_2$, accordingly. Since this group contains
  $J_\sV$ and the modular group, it follows that
    \[ \sV = \sV_1 \oplus \sV_2 \quad \mbox{ with } \quad
    \sV_1 = \sV \cap \cH^{N_h} \quad \mbox{ and }  \quad 
    \sV_2 = \sV \cap (\cH^{N_h})^\bot,\]
  where $\sV_1 = \sV(h, U_1)$.

  \nin  ``$\supeq$'': On $\cH_1$  the group $N_h$ acts trivially,
  so that $\g = \fn_h + \g_0(h)$ 
  (Lemma~\ref{lem:eul-nh}) implies that $U_1(G)  
  = U_1(\la \exp \g_0(h) \ra)$ commutes with
  the modular group $U_1(\exp \R h)$ of $\sV_1$.
  Further $\g_0(h) =\g^{\tau_h}$ shows that $U_1(G)$
  also commutes with $J_1 = U_1(\tau_h^G)$,
  and therefore $\sV_1$ is $U_1(G)$-invariant.
  This proves that $\sV_1 \subeq \sV_G$.

  \nin ``$\subeq$'': We consider the closed $U(G)$-invariant subspace
$\cH_0 := \oline{\sV_G + i \sV_G}$ 
  and note that $\sV_G$ is a standard subspace of~$\cH_0$. 
  As $\sV_G$ is invariant under $U(\exp \R h) = \Delta_\sV^{i\R}$,
the modular group of $\sV$,  it follows from \cite[Cor.~2.1.8]{Lo08} that
  \[ \Delta_{\sV_G} = e^{2 \pi i \, \partial U_0(h)}
    \quad \mbox{ for } \quad U_0(g) := U(g)\res_{\cH_0}.\]
The $U_0(G)$-invariance of the standard subspace $\sV_G$ 
implies that $U_0(G)$ commutes with its modular operator,
{hence with $\partial U_0(h)$, and thus
  $\partial U([h,x]) = 0$ for $x \in \g$.}
This implies
that $[h,\g] \subeq \ker \dd U_0,$
  so that the ideal $\ker(\dd U_0) \trile \g$ contains $\g_{\pm 1}(h)$,
  hence also
  \[ \fn_h = \g_1(h) + [\g_1(h),\g_{-1}(h)] + \g_{-1}(h)\]
  (cf.~Lemma~\ref{lem:eul-nh}). 
  This is turn shows that $\cH_0 \subeq \cH^{N_h}$,
hence $\sV_G \subeq \sV \cap  \cH^{N_h}$.   
\end{prf}

\begin{cor} \mlabel{cor:4.10}  If $G$ is connected and $h \in \fn_h$, then
  \[ \sV_G= \sV \cap \sV'.\] 
\end{cor}

\begin{prf} Theorem~\ref{thm:vg-ident}
  shows that $\sV_G \subeq \cH^{N_h}$, and
 since $h \in \fn_h$ by assumption, $\sV_G$ is fixed by its modular
  group, hence contained in $\Fix(\Delta_\sV) \cap \sV = \sV \cap \sV'$.
    
  If, conversely, $v \in \sV \cap \sV'$, then
  $v$ is fixed by $U(\exp \R h) = \Delta_\sV^{i\R}$, hence by
  definition of $N_h$ also by $N_h$,
  so that $v \in \sV \cap \cH^{N_h}  = \sV_G$
  (Theorem~\ref{thm:vg-ident}).
\end{prf}

With the standard subspace
$\sV_G \subeq \cH^{N_h}$, the preceding corollary
yields an orthogonal decomposition 
\[  \sV = \sV_G \oplus \sV_{\rm symp}, \]
where $\sV_{\rm symp} \subeq (\cH^\R,\omega)$ is a symplectic subspace
for $\omega = \Im \la \cdot, \cdot \ra$ and
$\sV_{\rm symp} = \sV(h,U_s)$ for the
(anti-)unitary representation $U_s$ of $G_{\tau_h}$ on $(\cH^{N_h})^\bot$.

\begin{cor} \mlabel{cor:4.11} If $G$ is connected and $\fn_h = \g$, then the following
  are equivalent:
  \begin{itemize}
  \item[\rm(a)] $\sV_G = \{0\}$, i.e.~$(U,\sV)$ is  non-degenerate.
  \item[\rm(b)] $\cH^G = \{0\}$. 
  \item[\rm(c)] $\sV \cap \sV' = \{0\}$. 
  \item[\rm(d)] The closed real subspace $\tilde\sV$ generated by
    $U(G)\sV$ coincides with $\cH$. 
  \end{itemize}
\end{cor}

\begin{prf}  Theorem~\ref{thm:vg-ident} implies that
  $\sV_G = \sV \cap \cH^G$, which is a standard subspace of the
  $G_{\tau_h}$-invariant subspace $\cH^G$. This implies the equivalence
  of (a) and (b).
  The equivalence of (a) and (c) follows from Corollary~\ref{cor:4.10}.
  To connect with (d),  we note that
\[ J_\sV \sV_G
  = \bigcap_{g \in G} J_\sV U(g) \sV
= \bigcap_{g \in G} U(\tau(g)) J_\sV \sV
= \bigcap_{g \in G} U(\tau(g)) \sV' 
= \bigcap_{g \in G} U(g) \sV' ={ (U(G)\sV)'}\]
shows that (d) is equivalent to (a).
\end{prf}

\begin{remark}
(a) Let $h\in\fg$ be an Euler element, if $h$ is symmetric then the condition $h\in\fn_h$ is satisfied. Indeed in this case there exists a subalgebra $\fh\subset\fg$  such that  $\fh\simeq\fsl_2(\R)$ and $h$ is an Euler element of $\fh$ \cite[Corollary 3.14]{MN21}. Then $h\in [\fh_1,\fh_{-1}]\subset\fn_h$.

\nin (b) If $h$ is not symmetric, then Corollary \ref{cor:4.10} does not hold. 
Indeed let {$(\sH(\cO))_{\cO}$} be the  one-particle net associated
to the   free field mass in dimension $1+1$ with mass $m>0$
and let $U$ be the mass $m$ representation of the identity
component $\cP_+^\uparrow=\RR^{1,1}\rtimes \cL_+^\up$ of the Poincar\'e group.
The wedge subspaces $\sV:=\sH(W_R)$ and $\sH (W_L)$ are
{mutually orthogonal symplectic factor subspaces satisfying}
\[ \sH(W_R)'=\sH (W_L) \quad \mbox{ and } \quad \sH(W_R)\cap\sH (W_L)=\{0\}.\] 
Here the wedge $W_R$ is associated to an Euler couple $(h,\tau_h)$
(cf. Example \ref{ex:desit}), and since $h$ is {neither symmetric 
  in $\cP_+^\uparrow$ nor} in $\cL_+^\up$ (note that $\so_{1,1}(\R) \cong \R$ is abelian),
there is no $g$ such that $gW_R=W_L$.
One can restrict the symmetry group to $H:=\cL_e$ as well as the representation $U|_{H}$, acting as automorphisms of $\sH(W_R)$.
We conclude that $\sV_H=\sV\neq \sV\cap \sV'=\{0\}$ since the subspace
{$\sV=\sH(W_R)$ is symplectic.} 

\nin (c) {The containment $h\in \fn_h$ does not imply that
  $h$ is symmetric: For instance no Euler element $h\in\fsl_3(\RR)$ is 
  symmetric, but $h\in \fg=\fn_h$ follows from the simplicity of
  $\fsl_3(\RR)$.}
\end{remark}

\subsection{Consequences of Moore's Theorem for operator  algebras}
\mlabel{subsec:cons-moore}

{For the discussion in this section, we recall the conditions
  (Uni), (M), (Fix), (Mod) and (Reg) from Section~\ref{subsec:strich-opalg}.}
  
\begin{theorem} \mlabel{thm:3.8}
  Let $G$ be a connected Lie group with
    Lie algebra $\g$ and $h \in \g$ anti-elliptic. 
  Let $(U,\cH)$ be a unitary 
  representation of $G$ with discrete kernel,
  $\cN\subset \cM \subeq B(\cH)$ an inclusion of von Neumann algebras, 
  and $\Omega\in\cH$ a unit vector which is
  cyclic and separating for~$\cN$ and $\cM$. 
  Assume that 
  \begin{itemize}
  \item[\rm(Mod)] $e^{2\pi i \partial U(h)} = \Delta_{\cM,\Omega}$,  and 
  \item[\rm(Reg')] $\{ g \in G \: \Ad(U(g)) \cN \subeq \cM\}$
    is an $e$-neighborhood in $G$.
  \end{itemize}
  Then the following assertions hold:
  \begin{itemize}
  \item[\rm(a)]   $h$ is an Euler element. 
  \item[\rm(b)]  The conjugation $J := J_{\cM,\Omega}$ satisfies
  \begin{equation}\label{eq:jcov} J U(\exp x) J
    = U(\exp \tau_h(x)) \quad \mbox{ for } \quad
    \tau_h = e^{\pi i \ad h}, x \in \g.\end{equation}
\item[\rm(c)] $\cH^G = \ker(\partial U(h))$.   
\item[\rm(d)] The restriction of $i\partial U(h)$ to the
  orthogonal complement of the subspace 
  $\cH^{N_h}$ of fixed vectors of the codimension-one normal subgroup~$N_h$,
  has absolutely continuous spectrum. 
  \end{itemize}
 If, in addition, $\cH^G = \C \Omega\not=\cH$, 
  then $\cM$ is factor of type {\rm III}$_1$.
\end{theorem}

\begin{proof} Our assumptions clearly imply (Uni), (M) and (Mod).
Let $N \subeq G$ be the $e$-neighborhood specified by (Reg'). 
Then $\cM_N \supeq \cN$, so that (Reg) is also satisfied.
As $h$ is anti-elliptic and $\Omega \in \ker(\partial U(h))$ by (Mod),
Corollary~\ref{cor:inheritfixed} implies that
\[ \Omega \in \cH^G = \ker(\partial U(h)),\]
which is (c). Now Theorem~\ref{thm:2.1-alg} implies~(a) and~(b). 
Further, (d) follows from Moore's Theorem.

If, in addition, $\cH^G = \C \Omega\not= \cH$, then
\[ \C \Omega = \ker(\partial U(h)) = \ker(\Delta_{\cM,\Omega}-\1),\]
so that $\cM$ is a factor of type III$_1$ by Proposition~\ref{prop:4.1}(e) 
because $\cH = \oline{\cM\Omega}$ implies $\cM \not= \C\1$
and $\C \Omega = \ker(\Delta_{\cM,\Omega}-\1)$ implies
$\Delta_{\cM,\Omega} \not=\1$. 
\end{proof}

In our context, Theorem 6.2 of \cite{BB99} becomes
  the following corollary. We use the notation from~\ref{ex:desit}.
  \begin{corollary} {\rm(Borchers--Buchholz Theorem)} Let $(U,\cH)$ be a 
    unitary representation of the Lorentz group $G = \SO_{1,d}(\R)^\up$
    acting covariantly on an isotone net
    $(\cA(\cO))_{\cO \subeq \dS^d}$ 
    of von Neumann algebras on open non-empty subsets of de Sitter spacetime,
 i.e.,~$\cO_1\subset \cO_2$ implies $\cA(\cO_1)\subset\cA(\cO_2)$ (isotony) and $\Ad(U(g))(\cA(\cO))=\cA(g\cO)$ with $g\in G$ (G-covariance).  
 Let $\Omega\in\cH$ be a fixed vector of $U(G)$ that is cyclic and separating for any $\cA(\cO)$. Assume that  the vacuum state $\omega(\cdot)=\langle\Omega,\cdot\,\Omega\rangle$  is a KMS state for $\cA(W_R)$
 with inverse temperature $\beta>0$ with respect to the
 one-parameter group $(U(\exp th))_{t \in \R}$, namely for every pair $X,Y\in \cA(W_R)$, there exists an analytic function $F_{X,Y}$ on
 the strip $\{z\in\mathbb C: 0<\Im z<\beta\}$
 with continuous boundary values satisfying
 \[ F(t)=\omega(X\Ad(U(\exp th))(Y)),
   \quad F(t+i\beta)=\omega(\Ad(U(\exp th))(Y)X),\quad t\in\RR. \] 
Then $\beta=2\pi$
\end{corollary}

\begin{proof} For $\cO\Subset W_R$,
  there exists an open neighborhood of the identity
  $N\subset\SO_{1,d}(\R)^\up$ such that $\cO\subset g W_R^{\rm{dS}}$
  for all $g\in N$. Let $\cM:= \cA(W_R^{\rm{dS}})$.
  By covariance, $\cN:=\cA(\cO)$ satisfies
  (Reg') in Theorem \ref{thm:3.8}. 
  The KMS property implies that $\Ad(U(\exp th))
  =\Ad(\Delta^{-{it}/{\beta}}_{\cA,\Omega})$ (cf.~\cite[Thm.III.4.7.2 ]{Bl06}) and,
  since the representation of $\cA(W_R^{\dS})$ on $\cH$
  is the GNS representation for w.r.t.~$\omega$,
  we have that $U\left(\exp\left(\frac{\beta t}{2\pi}h\right)\right)
  =\Delta_{\cA(W_R^{\dS}),\Omega}^{-\frac{it}{2\pi}}$, and 
  Theorem \ref{thm:3.8} applies. We conclude that $\frac{\beta  }{2\pi}h$ is an Euler element, but since $h$ is also an Euler element in $\so_{1,d}(\R)$,
  it follows that $\beta=2\pi$.
\end{proof}

\begin{defn} \mlabel{def:cA-tildecA} 
  We write $\displaystyle{\cA := \big(\bigcup_{g \in G} \cM_g)'' \subeq B(\cH)}$
  for the von Neumann algebra generated by all algebras $\cM_g=U(g)\cM U(g)^{-1}$.
  Let $ (\cM')_G:=\bigcap_{g \in G} \cM_g'$ and note that
  \begin{equation}
    \label{eq:caprime}
    \cA' = \bigcap_{g \in G} \cM_g' =(\cM')_G.
  \end{equation}
  We also write $\tilde\cA$ for the von Neumann algebra generated by
  $\cA$ and $J\cA J$ with $J=J_{\cM,\Omega}$, i.e., by all algebras $\cM_g$ and $(\cM')_g$, $g \in G$.
  Then $\tilde\cA' \subeq \cM \cap \cM' = \cZ(\cM)$ and, more precisely,
  \begin{equation}
    \label{eq:tildea'}
\tilde\cA' = \cZ(\cM)_G  = \bigcap_{g \in G} \cZ(\cM)_g
  \end{equation}
  is the maximal $G$-invariant subalgebra of $\cZ(\cM)$.
\end{defn}

\begin{lem} \mlabel{lem:5.1}
  Let $\alpha_t := \Ad(\Delta^{it}) \in \Aut(\cM)$ be the modular automorphisms
  of the von Neumann algebra  $\cM$ corresponding to the cyclic separating vector~$\Omega$. 
  If {\rm(Uni)}, {\rm(M)}, {\rm(Fix)}, {\rm(Reg)}
  and {\rm(Mod)} are satisfied and
  $h$ is anti-elliptic, then
  \begin{itemize}
  \item[\rm(a)] $\cA'\subeq \cM'$ is invariant under $\Ad(U(G))$. 
  \item[\rm(b)] $(\cM')^G = (\cM')^\alpha = (\cA')^G$. 
  \item[\rm(c)] $\cZ(\cM) \subeq \cM^G = \cM^\alpha$. 
  \end{itemize}
\end{lem}

\begin{prf} (a) $\cA' \subeq \cM'$ holds by definition, and
    $\cA'$ is $U(G)$-invariant.

  \nin   (b) By (Mod), we have $(\cM')^G \subeq (\cM')^\alpha$.
  To show the converse, suppose that $A \in \cM'$ is fixed by $\alpha$.
  As $h$ is anti-elliptic, $A\Omega \in \cH^\Delta = \cH^G$
  (Corollary~\ref{cor:inheritfixed}), which  implies that
  \[ U(g)A U(g)^{-1} \Omega = U(g)A \Omega = A \Omega.\]
  If $g \in N$, {with $N$ as in (Reg)}, then $\cM' \cup \cM'_g \subeq \cM_N'$ and $\Omega$ is separating
  for $\cM_N'$, so that we obtain
  \[  U(g)A U(g)^{-1} = A.\]
  We conclude that $A$ commutes with $U(N)$, and since the connected group $G$
  is generated by the identity neighborhood~$N$,
  it follows that $A$ commutes with $U(G)$. {This shows that
    $(\cM')^G = (\cM')^\alpha$.}

As $\cA$ is $G$-invariant, so it holds $\cA'\subeq \cM'$. Further,
 {\[ (\cA')^G \subeq   (\cM')^G \subeq (\cM')_G = \cA'\]}
    by \eqref{eq:caprime}. This implies that
    $(\cA')^G = (\cM')^G$.    
  
  \nin (c) Using the relation $\cM = J \cM' J$ and the fact that $J$ normalizes $U(G)$
  (Theorem~\ref{thm:2.1-alg})
  and commutes with $U(\exp \R h)$, the equality
  $\cM^G = \cM^\alpha$ follows from (b) by conjugating with~$J$.   
  Further $\cZ(\cM) \subeq \cM^\alpha$ follows from the fact that 
  modular automorphisms fix the center pointwise
  (\cite[Prop.~5.3.28]{BR96}).
\end{prf}

\begin{prop} \mlabel{prop:3.12}
  Suppose that {\rm(Uni)}, {\rm(M)}, {\rm(Fix)},
  {\rm(Mod)} and {\rm(Reg)} are satisfied,
  that $h$ is anti-elliptic, and that $\Delta \not=\1$.
  For the assertions
  \begin{itemize}
  \item[\rm(a)] The net $(\cM_g)_{g \in G}$ is irreducible, i.e., $\cA = B(\cH)$.
  \item[\rm(b)] $\cA' = (\cM')_G = \bigcap_{g \in G} \cM_g' = \C \1$. 
  \item[\rm(c)] $\cM_G = \bigcap_{g \in G} \cM_g = \C \1$. 
  \item[\rm(d)] $\cH^G = \C \Omega$. 
  \item[\rm(e)] $\cM$ is a type III$_1$ factor. 
  \end{itemize}
  we have the implications: 
  \[ {\rm(a)} \Leftrightarrow{\rm(b)} \Leftrightarrow{\rm(c)}
    \Rarrow     {\rm(d)} \Rarrow{\rm(e)}.\]
\end{prop}

Note that {\rm(d)} is stronger than $\cZ(\cM) = \C \1$.
  
\begin{prf}   (a) $\Leftrightarrow$ (b) follows from $\cA' = \bigcap_{g \in G} (\cM_g)'
  = (\cM')_G$.

    \nin  (b) $\Leftrightarrow$ (c): As $J U(G) J = U(G)$ by
  Theorem~\ref{thm:2.1-alg} and $J\cM J = \cM'$, we have
  $J \cM_G J = (\cM')_G$. Therefore (b) and (c) are equivalent. 

  \nin  (c) $\Rarrow$ (d):  
  From Proposition~\ref{prop:4.1}(a)
    and  Lemma~\ref{lem:5.1}(c), we know that
    \begin{equation}
      \label{eq:chg}
 \cH^G = \cH^\Delta
\   \ {\buildrel \ref{prop:4.1}\over  =}\ \  \oline{\cM^\alpha\Omega} 
= \oline{\cM^G\Omega}.
    \end{equation}
  Therefore $\cM^G \subeq \cM_G = \C \1$ implies that
  $\cH^G = \C \Omega$.

\nin  (d) $\Rarrow$ (e): As $h$ is anti-elliptic, we have
  $\cH^G = \cH^\Delta$ (Corollary~\ref{cor:inheritfixed}),
  so that  Proposition~\ref{prop:4.1}(e) implies that
  $\cM$ is a factor of type  III$_1$. 
\end{prf}

  \begin{rem} \mlabel{rem:tildea}
    If $G = \R$ acts as the modular group of $(\cM,\Omega)$, then
    $\cA = \cM$, $\tilde \cA = (\cM \cup \cM')''$, and
    $\tilde\cA' = \cZ(\cM)$. So $\tilde\cA' = \C\1$ is equivalent to
    $\cM$ being a factor, but, in general, this does not imply that
    $\cH^G = \cH^\Delta = \C \Omega$ because we may have
    $\cM^\alpha \not= \C \1$ (cf.\ Remark~\ref{rem:nonsym}(b)). 
\end{rem}

\begin{rem} \label{rem:nonsym}
  (a) The implication (e) $\Rightarrow$ (c) holds if there exists a $g\in G$
  such that $\cM_g = U(g)\cM U(g)^{-1}\subeq \cM'$. 
{Then $\cM_G \subeq \cZ(\cM)$, and if $\cM$ is a factor, it follows that
    $\cM_G = \C \1$, so that (e) implies (c).} 
    
  If the Euler element $h$ is not symmetric, i.e.,
  there exists no $g\in G$ such that $\Ad(g)h=-h$,
  then (e) does not always imply (a).
  For instance, let $\RR^{1,1}\supset \cO\rightarrow\cM(\cO)$
  be the free field of  mass $m>0$ in $1+1$ dimensions and let
  $U$ be the mass $m$ representation of the
  identity component of the Poincar\'e group
  $\cP_+^\up=\RR^{1,1}\rtimes \cL_+^\up$. The algebras $\cM(W_R)$ and $\cM (W_L)$
  corresponding to the right
  and left wedges are invariant under the Lorentz action
  and of type III$_1$. This follows from uniqueness of the vacuum
  state and Proposition~\ref{prop:3.12}.
  In particular, the ``one wedge net'' $W_R \rightarrow \cM(W_R)$
 together with the representation $U|_{\cL_+^\up}$ satisfies (Uni), (M), (Fix), (Mod) and  (Reg)
  but the algebra generated by  $\Ad(U(\cL_+^\up))\cM(W_R)=\cM(W_R)$
  is properly contained in $\cB(\cH)$ (see also Example~\ref{ex:3.7}).

  \nin (b)
{  The implication   ``(e) $\Rarrow$ (d)'' is related to
  the   ergodicity of the state on he type III$_1$-factor
  $\cM$ specified by $\Omega$: By \eqref{eq:chg},
  ergodicity of the state defined by $\Omega$
  is equivalent to $\cH^G = \C\Omega$. 
This does in   general not follow from (e) because  non-ergodic
  states always exist for a type III$_1$-factors
(Remark~\ref{rem:a.2}). Concretely, such states can be
    obtained as follows:} 
    Consider a type III$_1$ factor $\cM\subset\cB(\cH)$
    and the algebra $M_2(\C)$ of complex $2\times 2$-matrices.
    Then $\tilde \cM=\cM\otimes M_2(\C)$ is a type III$_1$
    factor (\cite[Thm.~V.2.30]{Ta02}). For a faithful normal state $\omega$
    on $\cM$, we  consider the state on $\tilde\cM$ specified by 
    \[ (\omega\otimes \phi_{11})(m\otimes x)=\omega(x)x_{11}.\]
    This is a non-ergodic (non-faithful) state on the type
    III$_1$ factor~$\tilde \cM$. 

\nin(c) Suppose that $\cM = \cM_G$, i.e., that
$\cM$ is normalized by $U(G)$.
  Then $G = G_\cM$ and $\Omega \in \cH^G$ imply
$G = G_{\sV_\cM}$, so that $h$ is central in $\g$
and therefore $\tau_h = \id_G$.
The example described in point (a) with $G=\cL_e$ is of this type.
\end{rem}

\subsection{The {degenerate}  case} \label{sect:disnet}

Proposition~\ref{prop:3.12} describes  the {non-degenerate} case, where
$\cH^G=\bC\Omega$. If $\cH^G$ is not one-dimensional, we now obtain a
direct integral decomposition, in accordance with
the AQFT literature, see \cite[Cor.~6.2.10]{Lo08b},
\cite[Sect.~4.4]{Ara76}, \cite[Sect.~5]{BB99}. 

The following  proposition extends \ref{prop:3.12} to the case where the vacuum $\Omega$
is not cyclic. We will comment on conditions (a) and (b) in Remark \ref{rem:disint} below. 

\begin{proposition} \label{prop:disnet} 
  Suppose that $\cH$ is separable.
  Let $(\alpha_t)_{t \in \R}$ be the modular automorphisms
  of $\cM$ with respect to the cyclic separating vector~$\Omega$
  and $(U,\cH)$ a unitary representation of $G$, such that: 
  {    \begin{itemize}
  \item[\rm(a)]   {\rm(Uni)}, {\rm(M)}, {\rm(Fix)}, {\rm(Reg)}  and {\rm(Mod)}
  and $h$ is anti-elliptic in $\g$.
  \item[\rm(b)] $\cM'=\cM_{g_0}$ for some $g_0\in G$.
  \end{itemize} 
\nin  Then we have direct integral decompositions
  \[ \cM=\int_X^\oplus\cM_x\, d\mu(x),\qquad U=\int_X^\oplus U_x\, d\mu(x),
{    \quad \mbox{ and }  \quad
  \cA=\int_X^\oplus B(\cH_x) d\mu(x).}\]
We have a measurable decomposition $X = X_0 \dot\cup X_1$, where 
$\dim \cH_x = 1$ for $x \in X_0$ and the representations
$(U_x)_{x \in X_0}$ are trivial. For $x \in X_1$, the algebras
$\cM_x$ are factors of type III$_1$ and
$(\cM_x, \Omega_x, \uline U_x)$ satisfies {\rm(Uni), (M), (Fix), (Reg)}
and {\rm(Mod)}, where $\uline U_x$ is the representation of
$G/\ker(U_x)$ induced by $U_x$. }
\end{proposition}

\begin{proof} From $\cM' = \cM_{g_0}$ for some $g_0 \in G$,
  we derive that {$\cA' \subeq \cZ := \cM \cap \cM'$. 
    Using
    Lemma~\ref{lem:5.1}(b),(c)}, we obtain 
  \begin{equation}
    \label{eq:a'}
 \cA' = (\cA')^G \subeq \cZ = \cZ^G
 \subeq (\cM')^G = (\cM')^\alpha = (\cA')^G = \cA',
  \end{equation}
  so that 
  \begin{equation}\label{eq:aaaa} \cZ^G = \cZ = (\cM')^\alpha= \cA' . \end{equation}

By \cite[Thm.~4.4.3]{BR87}, there exists a
  {finite} standard measure space $(X,\mu)$,  a unitary $\Phi$
  such that $$\Phi\cH=\int_X^\oplus \cH_xd\mu(x)$$ and $U\cZ U^*$ acts on the direct integral as the algebra {$L^\infty(X,\mu)$} of diagonal operator.  From \cite[Thm.~4.4.6(a)]{BR87}, passing to the commutant one can easily see that $\cA {= \cZ'}$ can be represented as  {the} direct integral von Neumann algebra
  x{of decomposable operators:} 
\[ \Phi\cA \Phi^* = \int_X^{\oplus} B(\cH_x)\,  d\mu(x). \]

If $\cC$ is a von Neumann subalgebra of $\cA$, then $\Phi\cC \Phi^*\subset \Phi\cA \Phi^*$ and there exists a measurable family of von Neumann algebras  $X \ni x\mapsto \cC_x\subset B(\cH_x)$ for almost every $x\in X$
\cite[Thms.~8.21, 8.23]{Ta02} .
In particular $U\cC U^*=\int_X^\oplus\cC_x\, d\mu(x)$. Since $U$ does not depends on the subalgebra hereafter in the proof we will work on the direct integral Hilbert space, i.e. we will assume $\cH=\int_X^\oplus \cH_x\,d\mu(x)$.

With this argument
we can also assume that on the same standard finite measure space $(X,\mu)$  we have 
\begin{equation}\label{eq:dec}
 (\cM,\cH) = \int_X^{\oplus} (\cM_x, \cH_x)\, d\mu(x),\end{equation}
for which $\cZ\cong L^\infty(X,\mu)$ is the diagonal algebra
and almost every $\cM_x$ is a factor \cite[Cor. 8.20]{Ta02}.

As $\cZ$ commutes with $U(G)$, we have
  \begin{equation}
    \label{eq:ginz'}
 U(G) \subeq \cZ' = \cA'' = \cA.    
\end{equation}
Hence the separable $C^*$-algebra
$ C^*(U(G))$ is contained in $\cZ' = \cA$, so that
\cite[Cor.~4.4.8]{BR87} 
yields a direct integral decomposition
of the unitary representation
\[ (U,\cH) = \int_X^{\oplus} (U_x, \cH_x)\, d\mu(x).\]
For $x \in X$, the kernel $\ker U_x$ may not be discrete, so that
(Uni) holds for $(U_x, \cH_x)$ only as a representation $\uline U_x$
of $G/\ker(U_x)$. 

Since $U$ is a direct integral representation, we have 
\begin{equation}\label{eq:dec2}
 (\cM_g,\cH) = \int_X^{\oplus} ((\cM_g)_x, \cH_x)\, d\mu(x).\end{equation}

By Proposition~\ref{prop:4.1}(a),
  $\Omega \in \cH^G \subeq \cH^\Delta$
  is a cyclic separating vector for $\cZ = (\cM')^\alpha$. 
  Writing $\Omega = (\Omega_x)_{x \in X}$, it follows that
  almost no $\Omega_x$ vanishes, and thus 
\[ \cH^G= \int_X^\oplus \C \Omega_x\, d\mu(x) \cong L^2(X,\mu).\]

Replacing $\cN$ in (Reg) by the von Neumann algebra 
$\cM_N= \bigcap_{g \in N} \cM_g$, where $N \subeq G$ is an
$e$-neighborhood satisfying (Reg), we see that
$\cM_N \subeq \cZ'$ also decomposes according to the direct integral.
We also obtain 
\[ \cM_N = \int_X^\oplus (\cM_x)_N\, d\mu(x), \] from
Lemma \ref{lem:denseunion}. 
Theorem~\ref{thm:3.8} now shows that $\partial U(h)$ also  decomposes in
such a way that
\begin{equation}
  \label{eq:h-red}
  \ker(\partial U_x(h)) = \C \Omega_x
\end{equation}
for almost every $x \in X$.

Since $\Omega$ is cyclic and separating for $\cM$,
the vectors $\Omega_x \in \cH_x$ must be cyclic separating for the
von Neumann algebras $\cM_x$ for almost every $x \in X$
 (easy argument by contradiction, we also refer to {\cite[Thm.~VIII.4.8]{Ta03}} for a more general case).
We therefore obtain (Uni), (M), (Fix), (Mod) and (Reg) for
the algebras $\cM_x \subeq B(\cH_x)$ and
the representations $\uline U_x$ of $G/\ker(U_x)$ on $\cH_x$.  
Finally, since $\cA'$ is the diagonal algebra 
\[ \C \1 = (\cA_x)'  = \bigcap_{g \in G} (\cM'_x)_g  \]
holds for almost every $x \in X$ (Lemma \ref{lem:denseunion}
and \cite[Thm.~4.4.5]{BR87}). 

The condition $\Delta_x \not= \1$ is by \eqref{eq:h-red}
equivalent to $\dim \cH_x > 1$, and in this case
Proposition~\ref{prop:3.12} applies to the configuration
in the Hilbert space $\cH_x$ and shows that $\cM_x$
is a type III$_1$-factor. If $\dim \cH_x = 1$, then
$\cM_x = \C \1$  and $\partial U_x(h) = 0$ implies
the triviality of the representation $U_x$ because
\[ \cH^G_x = \ker(\partial U(h))_x = \C \Omega_x = \cH_x \] 
(Theorem~\ref{thm:3.8}(c)).

We now define $X_1 := \{ x \in X \:  \dim \cH_x > 1\}$ and
$X_0 := \{ x \in X \:  \dim \cH_x = 1\}$.
Then the triples $(\cM_x, \cH_x, U_x)$ satisfy 
(M), (Fix), (Reg), (Mod), and (Uni) for the
representation $\uline U_x$ of $G/\ker(U_x)$.
\end{proof}

\begin{remark}\label{rem:disint}
  (a) If $h$ is not a symmetric Euler element,
  the condition $\cM'\subset\cM_{g_0}$ may not hold
  (Remark \ref{rem:nonsym}(a)). 

  \nin (b) In Proposition \ref{prop:disnet} it was crucial that
  $\cM'=\cM_{g_0}$ for some $g_0 \in G$,
  in order to obtain the disintegration. Furthermore,
  $\cA'=\cZ = \cZ^G$ implies $U(G)\subset\cA$.
  In the general case it is not clear when the group 
  $U(G)$ is contained in $\cA$.
  In \cite[Prop.~4.1]{BB99}, this follows from the
  KMS property of the wedge modular groups together with their geometric action,
   {  where it is used that  boosts generate the Lorentz group}
  to see that $U(G) \subeq \cA'' = \cA$.
  In our argument  $U(G) \subeq \cA'' = \cA$ does not need that $G$ is generated by an orbit of Euler elements.

  \nin (c) In the proof of Proposition \ref{prop:disnet}, we disintegrated $\cM=\int_X^\oplus\cM_x\, d\mu(x)$ and  $U=\int_X^\oplus U_x\, d\mu(x)$  in order to apply
Proposition \ref{prop:3.12} fiberwise
and conclude that, for almost every {$x \in X_1$,
the algebra $\cM_x$ is a type III$_1$ factor.} 
We actually have deduced (M), (Fix), (Reg) , (Mod) for almost every the triple $(\cM_x, U_x, \Omega_x)$ and (Uni) for $(\cM_x, \uline U_x, \Omega_x)$. In particular we could apply Proposition \ref{prop:3.12} for almost every triple $(\cM_x, \uline U_x, \Omega_x)$, where all the properties (M), (Fix), (Reg), (Mod) and (Uni) hold. Actually, it is not needed to assume (Uni) on $U_x$ to conclude the  type III$_1$ property of $\cM_x$.  Along this paper, (Uni) is necessary to ensure that {$\dd U$ is
  injective and in particular that $\dd U(h)$ determines $h$
  uniquely.} 
 In the proof of Proposition \ref{prop:disnet} we only need that
 \begin{equation}\label{eq:Zx}(\cZ)_x=(\cM^\alpha)_x=\C\cdot \textbf{1}_{\cH_x}
 \end{equation}
  to apply  Proposition \ref{prop:4.1}(e).   We can conclude \eqref{eq:Zx} as follows: let $g_0\in G$, such that $\cM'=\cM_{g_0}\in \cA$, then we have $\cM'_x=(\cM_{g_0})_x $, hence $\cZ(\cM_x)=\cZ(\cM)_x=\cZ_x$ for a.e.~$x\in X$. Furthermore, $\cM^\alpha=\int_X^\oplus(\cM^\alpha)_x\, d\mu(x)$, and since $\cZ=\cM^\alpha=\C\cdot \textbf{1}$, then $(\cZ)_x=(\cM^\alpha)_x=\C\cdot \textbf{1}_{\cH_x}$  for almost every $x\in X$.

  \nin (d) Condition (b) in Proposition~\ref{prop:disnet} implies that $\cM'\subset \cA$. If $\cM'\not\subset \cA$ then Proposition  \ref{prop:disnet} does not hold  in the present form. One may to consider the larger von Neumann algebra
    $\tilde\cA$  generated by the $G$-transforms of $\cM$ and $\cM'$.
    Lemma \ref{lem:5.1}(c) then implies that
    $G$ acts trivially on $\cZ(\cM)$,
    so that \eqref{eq:tildea'} entails 
$\tilde\cA' = \cZ(\cM)$. 
Then $\tilde \cA$ contains $U(G)$, and one can repeat large portions of
the proof of Proposition \ref{prop:disnet} to disintegrate
the triple $(\cM,U,\tilde \cA)$.
However, in this situation the conclusion one can draw
from $\cZ(\cM_x) = \C \1$, i.e., if $\cM_x$ is a factor, are weaker.
In particular, $\cM_x^\alpha$ can be larger than $\C \1$, so that
$\cM_x$ need not be of type III$_1$ (cf.\ Remark~\ref{rem:tildea}). 
\end{remark}

\section{Outlook}\label{sect:out}

This paper develops a language concerning properties of
  nets of standard subspaces that provides  
descriptions on several levels of abstraction. It also incorporates 
a series of recent results from a new point of view.
\cite{BB99,BEM98} aim to deduce properties of QFT on
de Sitter/anti-de Sitter spacetime from the thermal
property of the vacuum state for a geodesic observer. In \cite{BS04},
the authors deduce AQFT properties
from the assumption on the state on the quasi-local algebra to be passive for a uniformly accelerated observer in $n$-dimensional anti-de Sitter spacetime
for $n\geq2$. 
\cite{Str08} aims to unify the previous approaches  by considering passive states for an observer traveling along 
worldlines in order to prove the
thermal property of the vacuum and the Reeh-Schlieder property.
His purpose was also to look for an abstract setting that, at the end,
was lacking concrete examples. Our context may provide the
proper setting in which such questions can be investigated and
where one has a large zoo of diverse examples. 

{If one starts with a standard subspace $\sV$ and a
  unitary representation $(U,\cH)$ of $G$, 
  then there are many ways to formulate conditions
  on a net of standard subspaces containing $\sV$
  that ensure the Bisognano--Wichmann property,
  or at least modular covariance,
  in the sense that the modular groups associated to wedge regions
  act geometrically;  see \cite{Mo18, MN21}.
  Results in these directions have recently been
  established in \cite{MN22}, and our Euler Element Theorem
  (Theorem~\ref{thm:2.1}) can also be considered as a tool
  to verify the Bisognano--Wichmann property.
  However, a satisfying answer to the  long-standing questions
  related to modular covariance for nets of 
  standard subspaces and the Bisognano--Wichmann property
  in free and interacting nets of von Neumann algebras
  requires further research. For a recent 
  approach to the situation for Minkowski spacetime
  through scattering theory, we refer to \cite{DM20} and references therein.}

In this paper, we do not analyze locality properties. 
Indeed, in our AQFT context it may happen that, 
on the same symmetric space $M$,
  there are no causally complementary wedge regions.
This happens if  the Euler element corresponding to the wedge $W$
is not symmetric, so that there exists no $g\in G$ with 
$gW=W'$ (cf.~\cite{MNO23b}). If $h$ is a symmetric Euler element and
  the center of $G$ is non-trivial, many complementary wedges appear. This
has been studied in \cite{MN21} at the abstract level,
but an analysis on symmetric spaces is still missing. Once a one-particle net is established one would aim to make a second quantization procedure which should take care of a
one-particle Spin-Statistics Theorem anticipated in \cite{MN21}.
Interesting new possibilities for {twisted} second quantization procedures
may be derived from the recent paper \cite{CSL23}.

Wedges on causal homogeneous space have been described in \cite{NO23,MNO23a,MNO23b}. Then the construction of covariant local nets
of {standard subspaces on open regions  have}
been described in \cite{FNO23,NO23}. {Having now}
understood that Euler elements are the {natural generators}
of the geometric flows of modular Hamiltonians (see Theorem \ref{thm:2.1} and Theorem \ref{thm:3.8}) on a causal homogeneous space, one is interested in a general geometric description of entropy and energy inequalities on symmetric spaces and their relation with the representation theory of Lie groups
(\cite{MTW22, CF20, CLRR22}).

\appendix
\section{Factor types and modular groups}
\mlabel{app:a.1}

We assume that $\Omega \in \cH$ is a cyclic and separating unit
vector for the von Neumann algebra $\cM \subeq B(\cH)$. 
 We consider the automorphism group $(\alpha_t)_{t \in \R}$ of $\cM$ defined by
the modular group via
\[  \alpha_t(M) = \Delta^{it} M \Delta^{-it}, \quad t \in \R, M \in \cM.\]
We write $\cM^\alpha$ for the subalgebra of $\alpha$-fixed elements
and $\cH^\Delta := \ker(\Delta - \1)$ for the subspace of
fixed vectors of the modular group. 

\begin{prop} \mlabel{prop:4.1}
  The following assertions hold:
\begin{itemize}
\item[\rm(a)] $\cM^\alpha \Omega \subeq \cH^\Delta$ is a dense subspace.
\item[\rm(b)] $\cH^\Delta = \C \Omega$ if and only if $\cM^\alpha = \C \1$, i.e., that
  $(\cM,\R,\alpha)$ is ergodic. 
\item[\rm(c)] $\cM^\alpha \supeq \cZ(\cM) = \cM \cap \cM'$. In particular, 
  $\cM$ is a factor if $(\cM,\R,\alpha)$ is ergodic.
\item[\rm(d)] The von Neumann algebra $\cM$ is semi-finite if and only if
the modular automorphisms $(\alpha_t)_{t \in \R}$ are inner,
  i.e., can be implemented by  a unitary one-parameter group of $\cM$.
If $\Delta$
is non-trivial and inner, then $\cM^\alpha \not=\C \1$.
\item[\rm(e)] If $\cH^\Delta = \C \Omega$ 
  and $\Delta\not=\1$, then $\cM$ is a factor of type {\rm III}$_1$.
  \end{itemize}
\end{prop}

\begin{prf} (a)  The inclusion
  $\cM^\alpha \Omega \subeq \cH^\Delta$ is clear.
  That $\cM^\alpha\Omega$ is dense in $\cH^\Delta$ follows from
  \cite[Prop.~6.6.4]{Lo08b}, applied with $G = \R$ and $U_t = \Delta^{it}$.
  
  \nin (b) This follows from (a) and the fact that $\Omega$ is a
  separating vector. 

  \nin (c) Here we use that modular groups fix the 
  center pointwise; see \cite[Prop.~5.3.28]{BR96}.

  \nin (d) The first assertion follows from \cite[Thm.~3.1.6]{Su87}.
 If  $(\alpha_t)_{t \in \R}$ is inner and non-trivial,
  then the spectral projections of the corresponding infinitesimal generator
    are contained in $\cM^\alpha$, showing that $\cM^\alpha \not=\C \1$. 

  \nin (e) From (b) we infer that $\cM^\alpha = \C \1$, so that
  (c) implies that $\cM$ is a factor. 
    By (d) it is of type III because $\Delta$ is non-trivial
    (here we use $\cM \not=\C\1$),
    but  cannot be inner by  ergodicity.
    \begin{footnote}
      {At this point \cite[Prop.~6.6.5]{Lo08b} implies that $\cM$
        is of type III$_1$, but as Longo's argument is very condensed,
        we provide some more details.}
    \end{footnote}
  We have to exclude the
  types III$_0$ and III$_\lambda$ for $\lambda \in (0,1)$.
  By \cite[Prop.~XII.3.15]{Ta03}, if $\cM$ is of type III$_0$,
  then the center of $\cM^\alpha$ is non-atomic. As this is not
  the case for $\cM^\alpha = \C$, this case is excluded.

  Let $\Gamma(\cM) \subeq \R^\times_+ \cong \hat\R$ denote the 
    Connes spectrum of $\alpha$ on $\cM$, which by
    \cite[Prop.~3.3.3]{Su87} coincides with the spectrum of $\alpha$ on $\cM$. 
   Now \cite[Prop.~3.4.7]{Su87} asserts that, if
    $\cM$ and $\cM^\alpha$ are factors, then
    \[ \Gamma(\cM) = S(\cM) \cap \R^\times_+
    = \R^\times_+ \cap \sigma(\Delta_\omega) \]
    for any faithful separating normal state $\omega$.
    If $\cM$ is of type III$_\lambda$, then
    $\Gamma(\cM) = \lambda^\Z$ (cf.\ \cite[Def.~3.3.10]{Su87}),
    so that the modular group $\alpha$ is periodic.
    By \cite[Exer.~XII.2]{Ta03},     this implies that
    $\cM^\alpha$ is a factor of type II$_1$, contradicting
    $\cM^\alpha = \C \1$. So type III$_\lambda$ is also ruled out.
    Alternatively, one can use \cite[Lemma~4.2.3]{Co73}, asserting that,
    if $\cM$ is a factor and $1$ is isolated in $\sigma(\Delta_\omega)$,
    then $\cM^\alpha$ contains a maximal abelian subalgebra of $\cM$.
    In our context this contradicts $\cM^\alpha  = \C \1$.     
\end{prf}

\begin{rem} \mlabel{rem:a.2}
  We have seen above that $\cM$ is a type III$_1$-factor
  if $(\cM,\R,\alpha)$ is ergodic.
  According to \cite{MV23}, the converse also holds
  in the sense that, if $\cM$ is a type III$_1$-factor,
  then the set of ergodic states is a dense $G_\delta$
  in the set of all faithful normal states. 
{That there are also faithful normal states that are not
    ergodic follows from \cite[Cor.~8]{CS78},
    that asserts for each hyperfinite factor $\cR$
    the existence of faithful normal states of $\cM$ with
    $\cM^\alpha \supeq \cR$.}
\end{rem}

\begin{rem} From Proposition~\ref{prop:4.1}(a) it follows that the
  $J$-fixed vector $\Omega$ is cyclic and separating in $\cH^\Delta$
  for the subalgebra $\cM^\alpha$.
  Hence $J\cM J = \cM'$ implies that the same holds
  of $(\cM')^\alpha$ because $J\cH^\Delta = \cH^\Delta$.
  We therefore have a standard form representation of
  $\cM^\alpha$ on $\cH^\Delta$. Note that
  the standard subspace $\sV = \sV_{M,\Omega}$ satisfies
  \[ \sV \cap \cH^\Delta = \sV^{\Delta} = \sV \cap \sV'
  = \sV \cap \cH^J\]
  and contains the standard subspace
  $\oline{\cM^\alpha_h.\Omega}$ of $\cH^\Delta$.
  This implies that the corresponding modular operator is
  trivial, so that $\omega_\Omega(A) := \la \Omega, A \Omega \ra$
  is a trace on $\cM^\alpha$ (\cite[Prop.~5.3.3]{BR96}).
\end{rem}
  
\begin{rem}
 Suppose that $\cM = \cR(\sV)$ is a second quantization algebra.
  Then $\cR(\sV \cap \sV') = \cR(\sV) \cap \cR(\sV)'$ by the Duality Theorem,
  so that $\cR(\sV)$ is a factor if and only if $\sV$ is symplectic, which
  is equivalent to
  \[ \ker(\Delta_\sV -\1) = \{0\}.\]
  We also have
  $\Delta = \Gamma(\Delta_\sV)$ for the corresponding standard subspace $\sV$.
  Therefore 
  $\cF(\cH)^\Delta = \C \Omega$ implies that $\cH^{\Delta_\sV} = \{0\}$,
  which is equivalent to $\cR(\sV)$ being a factor,
  but we have seen in Proposition~\ref{prop:4.1}(a) that
  $\cF(\cH)^\Delta = \C \Omega$ even implies that $\cM$ is a factor of
  type III$_1$.
  
  If $\cR(\sV)$ is a factor of type I, then the modular group
  is inner and, if $\sV \not=\{0\}$, it follows that $\cR(\sV)^\alpha \not=\C\1$.
  In view of Proposition~\ref{prop:4.1}(a),
  this implies that $\cF(\cH)^\Delta \not=\C\Omega$.
\end{rem}

\section{Smooth and analytic vectors}
\mlabel{app:b}

For a unitary representation $(U,\cH)$ of a Lie group $G$,
  we write $\cH^\infty \subeq \cH$ for the subspace of {\it smooth vectors,} i.e.,
  elements $\xi \in \cH$ whose orbit map
  \[ U^\xi \: G \to \cH, g \mapsto U(g)\xi \]
  is smooth.
  For $x \in \g$, we write
  $\partial U(x)$ for the infinitesimal generator of the one-parameter
  group $U(\exp tx)$, so that $U(\exp tx) = e^{t\partial U(x)}$. 
  On this dense subspace  we have the {\it derived representation}
  \[ \dd U \: \g_\C \to \End(\cH^\infty), \quad
    \dd U(x + i y)\xi := \partial U(x)\xi + i \partial U(y) \xi \quad \mbox{
    for } \quad x,y \in \g, \xi \in \cH^\infty \]
  for the derived representation of $\g_\C$ on this dense subspace.
  We also write $\cH^\omega \subeq \cH^\infty$ for the subspace of analytic
  vectors which is dense in $\cH$
  (\cite[Thm.~4]{Nel59}, \cite{Ga60}).
  As $\cH^\infty$ is dense and $U(G)$-invariant,
  $\partial U(x)$ is the closure of $\dd U(x)$
  (\cite[Thm.~VIII.10]{RS75}). 

    For an analytic vector $\xi \in \cH^\omega$, we then have 
  \[ U^\xi(\exp x) = U(\exp x)\xi
    = \sum_{n = 0}^\infty\frac{1}{n!}
    (\dd U(x))^n \xi \]
  for every $x$ in a sufficiently small $0$-neighborhood
  $U_\g^\xi \subeq \g$. Analytic continuation implies that,
  after possibly shrinking $U_\g^\xi$, the power series on the right
  converges on the $0$-neighborhood $U^\xi_{\g_\C} := U_\g^\xi + i U_\g^\xi
  \subeq \g_\C$ and defines a holomorphic function
  \begin{equation}
    \label{eq:etaxi}
 \eta_\xi \: U^\xi_{\g_\C} \to \cH, \quad
 \eta_\xi(z) := \sum_{n = 0}^\infty\frac{1}{n!}    (\dd U(z))^n \xi.
  \end{equation}

  If $\ker(U)$ is discrete, then $\dd U$ is injective on $\g$.
  But for $z \in \g_\C$ the adjoint $\dd U(z)^\dagger$ on $\dd U(z)$ on
  the pre-Hilbert space $\cH^\infty$ satisfies
  \[ \dd U(x + iy)^\dagger = -\dd U(x) + i \dd U(y)
  = \dd U(-x+iy)\quad \mbox{ for } \quad
    x,y \in \g.\]
  This implies that $\dd U \:  \g_\C \to \End(\cH^\infty)$ is also injective
  because $0 = \dd U(x+iy) = \dd U(x) + i \dd U(y)$
  implies that the hermitian and the skew-hermitian part of this operator
  on $\cH^\infty$ vanish, and thus $\dd U(x) = \dd U(y) = 0$.

\begin{lem}\mlabel{lem:1.1b}
    For $z \in \g_\C$, let $\dd U^\omega(z)$ denote the restriction of
    $\dd U(z)$ to $\cH^\omega$. Then
    \begin{equation}
      \label{eq:dualrel}
      \dd U(z) \subeq \dd U^\omega(-\oline z)^*
    \end{equation}
    In particular, the representation $\dd U^\omega$ of $\g_\C$ is injective if 
    $\ker(U)$ is discrete. If this is the case, then
    $\dd U^\omega(z)$ is skew-symmetric if and only if $z \in \g$.
  \end{lem}

  \begin{prf} We have 
    \[ \la \xi, \dd U^\omega(z) \eta \ra =  \la \dd U(-\oline z)\xi,  \eta \ra
      \quad \mbox{ for all }  \quad \xi \in \cH^\infty, \eta \in \cH^\omega, \]
    which is \eqref{eq:dualrel}.
    In particular, we see that $\dd U^\omega(z) = 0$ implies $\dd U(z) = 0$,
    so that $\ker(\dd U)= \ker(\dd U^\omega)$.
    Suppose that $\ker(U)$ is discrete, so that $\dd U$ and
    $\dd U^\omega$ are  injective. Then $\dd U^\omega(z)$ is skew-symmetric if and only
    if $z - \oline z \in \ker(\dd U^\omega) = \{0\}$, which is equivalent to
    $z \in \g$.
  \end{prf}

\section{Some facts on direct integrals}\label{app:stsub}

Let $\cH = \int^\oplus_X \cH_m \, d\mu(m)$ be a direct
integral of Hilbert spaces on a standard measure space
$(X,\mu)$. We call a closed real subspace
$\sH \subeq \cH$ {\it decomposable} if it is of the form
\[ \sH = \int_X^\oplus \sH_m\, d\mu(m), \]
where $(\sH_m)_{m \in X}$ is a measurable field of closed real subspaces.
Now let $(\sH^k)_{k \in K}$ be an at most countable family of decomposable
real subspaces. Then we have (\cite[Lemma~B.3]{MT19}): 
\begin{itemize}
\item[\rm(DI1)] $\sH' = \int_X^\oplus \sH_m'\, d\mu(m)$. 
\item[\rm(DI2)] $\bigcap_{k \in K} \sH^k = \int_X^\oplus
  \bigcap_{k \in K} \sH_m^k  \, d\mu(m)$. 
\item[\rm(DI3)] $\oline{\sum_k \sH^k} = \int_X^\oplus
  \oline{\sum_k \sH_m^k}  \, d\mu(m).$ 
\end{itemize}

\begin{lem} \mlabel{lem:di1}
  The subspace $\sH$ is cyclic/separating/standard
  if and only if $\mu$-almost all $\sH_m$ have this property. 
\end{lem}

\begin{prf} (a) First we deal with the separating property.  By
  (DI2) we have
  \[ \sH \cap i\sH = \int_X^\oplus (\sH_m \cap i \sH_m)\, d\mu(m),\]
  and this space is trivial if and only if $\mu$-almost all spaces
  $\sH_m \cap i \sH_m$ are trivial, which means that
  $\sH_m$ is separating.

  \nin (b) The subspace $\sH$ is cyclic if and only if $\sH'$ is separating.
  By (DI1) and (a) this means that $\mu$-almost all $\sH_m'$ are separating,
  i.e., that $\sH_m$ is cyclic.

  \nin (c) By (a) and (b) $\sH$ is standard if and only if
  $\mu$-almost all $\sH_m$ are cyclic and separating, i.e., standard.
\end{prf}

\begin{lem} \mlabel{lem:dirint2} For a countable family
  $(\sH^k)_{k \in K}$ of decomposable cyclic closed real subspaces, the
  intersection $\sV := \bigcap_{k \in K} \sH^k$
  is cyclic if and only if, 
  for $\mu$-almost every $m \in X$, the subspace
$\sV_m := \bigcap_{k \in K} \sH_m^k$ is cyclic. 
\end{lem}

\begin{prf} By (DI2), we have $\sV = \int_X^\oplus \sV_m\, d\mu(m)$,
  so that the assertion follows from Lemma~\ref{lem:di1}.   
\end{prf}

For a direct integral
\[ (U,\cH) = \int_X^\oplus (U_m, \cH_m)\, d\mu(m) \]
of (anti-)unitary representations of $G_{\tau_h}$, the canonical
standard subspace $\sV = \sV(h, U)\subeq \cH$ from \eqref{eq:def-V(h,U)}
is specified by the decomposable operator $J \Delta^{1/2}
= U(\tau_h) e^{\pi i\, \partial U(h)}$, hence decomposable: 
\begin{equation}
  \label{eq:v-dirint}
  \sV = \int_X^\oplus \sV_m\, d\mu(m).
\end{equation}

\begin{lem} \mlabel{lem:g-inter}
  Assume that $G$ has at most countably many components.
  For any subset $A \subeq G$ and a real subspace $\sH \subeq \cH$, we put
  \begin{equation}
    \label{eq:VA}
    \sH_A := \bigcap_{a \in A} U(g)\sH.
  \end{equation}
  Then the following assertions hold:
  \begin{itemize}
  \item[\rm(a)] If $\sH$ is decomposable, then
    $\sH_A = \int_X^\oplus \sH_{m,A}\, d\mu(m)$. 
  \item[\rm(b)] $\sH_A$ is cyclic if and only if $\mu$-almost all
    $\sH_{m,A}$ are cyclic. 
  \end{itemize}
\end{lem}

\begin{prf} (a) As $G$ has at most countably many components,
  it carries a separable metric. Hence there exists a countable
  subset $B \subeq A$ which is dense in $A$.
  For $\xi \in \cH$, we have
  \[ \xi \in \sH_A \quad \mbox{ if and only if } \quad
    U(A)^{-1}\xi \subeq \sH.\]
  Now the closedness of $\sH$ and the density of $B$ in $A$ show that
  this is equivalent to $U(B)^{-1}\xi \subeq \sH$, i.e., to
  $\xi \in \sH_B$. This shows that $\sH_A = \sH_B$.
  We likewise obtain $\sH_{m,A} = \sH_{m,B}$ for every $m \in X$.
  Hence the assertion follows by applying (DI2) to $\sH_B = \sH_A$.

\nin (b) follows from (a) and Lemma~\ref{lem:di1}.   
\end{prf}

We refer to \cite{BR87} for basic definition on direct integral objects.

\begin{lemma} \label{lem:denseunion}Let $\cH=\int_X^\oplus\cH_xd\mu(x)$, a  direct integral von Neumann algebra $\cA=\int_X^\oplus\cA_xd\mu(x)$ and a strongly continuous, unitary, direct integral representation of a connected Lie group $G$, $(U,\cH)=\int_X^\oplus(U_x,\cH_x)d\mu(x)$. Let $N\subset G$ a subset, then $$\bigcap_{g\in N}\cA_g=\int_X^\oplus \bigcap_{g\in N} (A_g)_xd\mu(x)$$ where $\cA_g=U(g)\cA U(g)^*$.
\end{lemma}
\begin{proof}

As $G$ has at most countably many components, 
  it carries a separable metric. Hence there exists a countable
  subset $N_0 \subeq N$ which is dense in $N$. 
  For $A \in B(\cH)$, the map
  \[ F \: G \to B(\cH), \quad F(g) = U(g) A U(g)^*,\]
  is weak operator continuous,
  so that the set of all $g \in G$ with
 $F(g) \in \bigcap_{g \in N_0} \cA_g$ 
  is a closed subset, hence contains $N$. We conclude that
  \[ \bigcap_{g \in N_0} \cA_g = \bigcap_{g \in N} \cA_g. \]
  We likewise obtain for every $x \in X$ the relation
  \[ \bigcap_{g \in N_0} \cA_{x,g}
    = \bigcap_{g \in N} \cA_{x,g} 
    \quad \mbox{ for } \quad \cA_{x,g} = U_x(g) \cA_x U_x(g)^*. \]
  From \cite[Prop.~4.4.6(b)]{BR87} we thus obtain
  \[ 
    \bigcap_{g \in N} \cA_g 
    =  \bigcap_{g \in N_0} \cA_g  
    = \int_X^\oplus  \bigcap_{g \in N_0} \cA_{x,g} \,
    d\mu(x) 
    = \int_X^\oplus  \bigcap_{g \in N} \cA_{x,g} \,
    d\mu(x).\]
  Finally, we observe that, for every $g \in G$
  \[ \cA_g
    = \int_X^\oplus (\cA_g)_x \, d\mu(x)
    = \int_X^\oplus \cA_{x,g}\, d\mu(x) \]
  follows by the uniqueness of the direct integral decomposition.
\end{proof}

\section{Some facts on (anti-)unitary representations} 
\mlabel{app:d}

\subsection{Standard subspaces in tensor products}

\begin{lem} \mlabel{lem:tensprostand} 
Suppose that $(U,\cH) = \otimes_{j =  1}^n (U_j, \cH_j)$ is a
tensor product of  (anti-)unitary  representations of $G_{\tau_h}$.
Then the standard subspace $\sV = \sV(h,U)$ is a tensor product
\[  \sV = \sV_1 \otimes \cdots \otimes \sV_n, \]
and for every non-empty subset $A \subeq G$
the subset $\sV_A := \bigcap_{g \in A} U(g)\sV$ satisfies
\begin{equation}
  \label{eq:vtens}
  \sV_A \supeq \sV_{1,A} \otimes \cdots \otimes \sV_{n,A}.
\end{equation}
\end{lem}

\begin{prf} We have $\xi \in \sV_A$ if and only if
  $U(A)^{-1} \xi \subeq \sV$.
  This shows that any
  $\xi = \xi_1 \otimes \cdots \otimes \xi_n$ with $\xi_j \in \sV_{j,A}$
  is contained in $\sV_A$, which is \eqref{eq:vtens}.
\end{prf}

\subsection{Existence of standard subspaces for unitary
  representations} 

The following theorem characterizes those Euler elements
which, in every unitary representation, generate a modular group
of some standard subspace. 

\begin{thm} \mlabel{thm:hsym} {\rm(Euler elements
      generating modular groups)} 
    Let $G$ be a connected Lie group and $h \in \g$ an Euler element.
    We consider the following assertions:
\begin{itemize}
\item[\rm(a)] $h \in [\g_1(h), \g_{-1}(h)]$. 
\item[\rm(b)] For all quotients $\fq = \g/\fn$, $\fn \trile \g$,
  in which the image of $h$ is central, we have $h \in \fn$, so that
  its image in $\fq$ vanishes. 
\item[\rm(c)] For all unitary representation $(U,\cH)$ of
  $G$, the selfadjoint operator $i\partial U(h)$ is unitarily
  equivalent to $-i\partial U(h)$.
\item[\rm(d)] For all unitary representation $(U,\cH)$ of
  $G$, there exists a standard subspace $\sV$ such that 
  $\Delta_\sV = e^{2\pi i \partial U(h)}$.
\end{itemize}
Then we have the implications
\[ (a) \Leftrightarrow (b) \Rarrow (c) \Leftrightarrow (d),\]
and if $G$ is simply connected, then all assertions are equivalent.
\end{thm}

\begin{prf}  (a) $\Leftrightarrow$ (b):
  The $\pm 1$-eigenspaces for the image of $h$ in $\fq$
  are the spaces $\fq_{\pm 1} = \g_{\pm 1}(h)/\fn_{\pm 1}(h)$.
  That the image of $h$ is central in $\fq$ means that
  both these spaces are trivial,
  i.e., that $\g_{\pm 1}(h) \subeq \fn$. As $\fn$ is a subalgebra,
  this means that 
  \[  \fri := \g_1(h) + \g_{-1}(h) + [\g_1(h), \g_{-1}(h)] \subeq \fn.\]
  As $\fri$ is an ideal of $\g$, condition (b) means that
  $h \in \fri$, but as $h \in \g_0(h)$, this is equivalent to (a).

  \nin (b) $\Rarrow$ (c): We argue by induction on $\dim G$. 
  Passing to the quotient group 
  $G/\ker(U)$, we may w.l.o.g.\ assume that $U$ has discrete kernel.
  If $h$ is central, then $h = 0$, so that (c) holds trivially because
  $\pm i \partial U(h) = 0$. 
  
  So we may assume that $h$ is not central. Hence
  there exists a non-zero $x \in \g_{\pm 1}(h)$.
  We consider the $2$-dimensional subalgebra 
  $\fb := \R h + \R x \cong \aff(\R)$ and the corresponding
  integral subgroup $B := \exp(\R x) \exp(\R h)$, which is
  isomorphic to $\Aff(\R)_e$. 

We may w.l.o.g.\ assume that 
$\cH^G = \{0\}$ because (c) obviously holds for
trivial representations. Then Moore's Theorem~\ref{thm:moore} implies that
\begin{equation}
  \label{eq:moore-x}
  \ker(\partial U(x)) \subeq \cH^{N_x},
\end{equation}
where $N_x \trile G$ is a normal integral subgroup
whose Lie algebra $\fn_x$ is the 
smallest ideal of~$\g$ such that the image $\oline x$ of $x$ in
the quotient Lie algebra $\g/\fn_x$ is elliptic.
As $x = \pm [h,x]$ is $\ad$-nilpotent
(the $h$-eigenspace decomposition implies that $(\ad x)^3 = 0$),
its image $\oline x$ in $\g/\fn_x$ must be central. So
$\oline x = \pm [\oline h,\oline x] = 0$ implies $x \in \fn_x$. 
Using that $N_x$ is a normal subgroup, we see that $\cH^{N_x}$ is $G$-invariant, and
the representation of $G$ on this space factors through a
representation of the quotient group $G/\oline{N_x}$
of strictly smaller dimension.
By the induction hypothesis, our assertion holds for this representation.

We may therefore consider the representation of $G$
on the orthogonal complement $(\cH^{N_x})^\bot$.
In view of \eqref{eq:moore-x}, we may assume that
$\ker(\partial U(x)) = \{0\}$.
Then the restriction of $U$ to the $2$-dimensional subgroup~$B$
is a direct sum or irreducible representations of~$B$
in which $x$ acts non-trivially, and every such representation
is equivalent to one of the representations
$(U_\pm, L^2(\R))$, where 
\begin{equation}
  \label{eq:affonl2b}
  (U_\pm(\exp(sx)\exp(th))f)(p) = e^{\pm i s e^p} f(p + t)
  \quad \mbox{ for } \quad s,t,p \in \R
\end{equation} 
(cf.~\cite[Prop.~2.38]{NO17}). 
For both these representations, 
the operator $i\partial U_\pm(h)$ is equivalent to the selfadjoint
operator $i\frac{d}{dp}$ on $L^2(\R,dp)$. 
This implies that $i\partial U(h)$ is unitarily
equivalent to $-i\partial U(h)$.

 \nin (c) $\Leftrightarrow$ (d):
The existence of a standard subspace
  $\sV$ with $\Delta_\sV = e^{2\pi i \partial U(h)}$
  is equivalent to the existence of a conjugation $J$
  commuting with $\partial U(h)$.
  In view of \cite[Prop.~3.1]{NO15}, {this is equivalent to the
    existence of a unitary operator $S$ with $S i\partial U(h) S^{-1}
    = -i\partial U(h)$. Therefore (c) and (d) are equivalent.} 

  \nin (c) $\Rarrow$ (b): We assume that $G$ is simply connected.
  If (b) is not satisfied, then there exists a
 quotient $\fq = \g/\fn$ in which the image $\oline h$ of
 $h$ is central but non-zero.
 Hence the corresponding quotient group $Q := G/N$
 (as $G$ is simply connected, $N$ is closed and $Q$ exists \cite{HN12}) has a
 non-trivial irreducible unitary representation $(U,\cH)$
 with $\partial U(\oline h) \not=0$. The irreducibility
 of $U$ implies that $\partial U(\oline h) = i\lambda \1$ for some
 $\lambda \in \R^\times$. Then $-i \partial U(\oline h) = \lambda\1$
 is not unitarily equivalent to $-\lambda\1 = i\partial U(\oline h)$. 
 Composing $U$ with the quotient map $G \to Q$, we see that (c) cannot
 be satisfied. This shows that (c) implies (b). 
\end{prf}

\begin{cor} 
    If $\g$ is semisimple and $h \in \g$ is an Euler element,
    then there exists for every  unitary representation $(U,\cH)$ of $G$
    a standard subspace $\sV$ with $\Delta_\sV = e^{2\pi i \partial U(h)}$.
  \end{cor}

  \begin{prf} As all quotients of the semisimple Lie algebra $\g$
    are semisimple, hence have trivial center, condition
    (b) in Theorem~\ref{thm:hsym} is satisfied.     
  \end{prf}

  \begin{ex} \mlabel{ex:d.4} (An example where $(c)\Rarrow (b)$ fails) 
    We consider the group $G_1 := \T^2 \times \tilde\SL_2(\R)$.
    Then $Z := Z(\tilde\SL_2(\R)) \cong \Z$, and there exists a homomorphism
    $\gamma \: Z \to \T^2$ with dense range because the element
    $(e^{\pi i \sqrt{2}}, e^{\pi i \sqrt{3}})$ generates a dense subgroup of~$\T^2$.
    Now \[ D := \{ (\gamma(z),z) \: z \in Z \} \]  is a discrete central subgroup
    in $G_1$, so that $G := G_1/D$ is a connected reductive Lie group
    with Lie algebra $\g = \R^2 \oplus \fsl_2(\R)$. Its commutator group
    $(G,G)$ is the integral subgroup corresponding to $\fsl_2(\R)$.
    As it contains a dense subgroup of the torus $\T^2$, it is dense in~$G$.

    Let $h = h_z + h_s \in \g$ be an Euler element with  $h_z \not=0$
    and $h_s \not=0$. Then $\g_{\pm 1}(h) = \g_{\pm 1}(h_s) \subeq \fsl_2(\R)$
    shows that (b) fails. We now verify (c), so that
    (c) does not imply (b) for all connected Lie groups.

    Pick a non-zero $x \in \g$ with
    $[h,x] = x$. As in the proof of ``(b) $\Rarrow$ (c)'' above,
    we see that $x \in \fn_x$, so that $\fsl_2(\R) = [\g,\g] \subeq \fn_x$.
    Hence $(G,G) \subeq N_x$, and the density of $(G,G)$ implies
    $\oline{N_x} = G$. We conclude that, for every unitary representation
    $(U,\cH)$ of $G$, we have $\ker(\partial U(x)) = \cH^G$.
    Clearly, (c) holds for the trivial representation of $G$ on $\cH^G$,
    and by the argument under ``(b) $\Rarrow$ (c)'' it also holds 
    for the representation on $\ker(\partial U(x))^\bot$.
    Therefore (c) holds for~$G$.
\end{ex}
  
\begin{rem} (a) If $G$ is a connected Lie group with Lie algebra
    $\g$, then its simply connected covering $q_G \: \tilde G \to G$
    is a simply connected Lie group with Lie algebra~$\g$.
    All unitary representations
    of $G$ yield by composition with $q_G$ unitary representations
    of $\tilde G$, but not all representations of $\tilde G$
    are obtained this way. If (c) holds for $G$, it may still fail
    for~$\tilde G$ {(Example~\ref{ex:d.4}).}

\nin (b) For a semidirect product
    $\g = \fr \rtimes \fs$ with $\fr$ solvable and
    $\fs$ semisimple, where $h$ is an Euler element contained in
    $\fs$, the equivalence of (a) and (b) in Theorem~\ref{thm:hsym}
    implies that $h \in [\fs_1(h), \fs_{-1}(h)] \subeq [\g_1(h),\g_{-1}(h)]$,
    so that Theorem~\ref{thm:hsym} applies to any simply connected Lie
    group $G$ with Lie algebra $\g$.

    This argument applies in particular to the Poincar\'e Lie algebra
    $\g = \R^{1,d} \rtimes \so_{1,d}(\R)$ and the Euler element
    $h \in \so_{1,d}(\R)$ generating a boost. 
  \end{rem}

\subsection{A criterion for real irreducibility}

The following lemma is needed in the discussion of
Example~\ref{ex:4.24} below. 

\begin{prop} \mlabel{prop:real-irred} 
    Any irreducible unitary representation  
    $(U,\cH)$  of $G$ for which $C_U \not= - C_U$
    is also irreducible as a real representation.
\end{prop}

\begin{prf}
  Let $(U^\R, \cH^\R)$ be the underlying real representation. 
  Then its complexification is of the form
  $U^\R_\C \cong U \oplus \oline U$, as complex representations,
  where $C_{\oline U} = - C_U$. As $C_U \not = - C_U$, the representations
  $U$ and $\oline U$ are not equivalent.
    Therefore the commutant of $U^\R_\C$ is isomorphic to~$\C^2$,
  and this implies that the commutant of $U^\R(G)$ in $B(\cH^\R)$ 
  cannot be larger than $\C \1$. Hence it contains no non-trivial
  projections, and thus  $(U^\R,\cH^\R)$ is irreducible.   
\end{prf}

\begin{cor} \mlabel{cor:affine-irrep} 
    For any irreducible unitary positive energy representation 
    $(U,\cH)$  of $\tilde\SL_2(\R)$, and any Euler element
    $h \in \fsl_2(\R)$, the restriction to the subgroup
  $P = \exp(\R h) \exp(\g_1(h))$ is irreducible as a 
  real orthogonal representation.   
\end{cor}

\begin{prf} We know that, in all cases,
  the representation $U_P := U\res_P$ of $P \cong \Aff(\R)_e =
  \R \times \R_+$ is equivalent to the
  representation on $L^2(\R_+,\C)$, given by
  \[ (U_P(b,a)f)(p) = a^{1/2} e^{ibp} f(ap).\]
  Hence $(U_P,\cH)$ is the unique irreducible positive energy
  representation of~$P$. Now the assertion follows from
  Proposition~\ref{prop:real-irred}.
\end{prf}

\end{document}